\documentclass[11pt, reqno]{amsart}
\usepackage[margin=3.6cm]{geometry}

\usepackage{array,booktabs,tabularx}
\usepackage{graphicx}
\usepackage{amssymb,amsmath, amsthm}
\usepackage{enumerate}
\usepackage{xfrac} 
\usepackage{nicefrac}
\usepackage{color}
\usepackage{hyperref}
\usepackage{esint}
\usepackage{mathtools}
\usepackage{dsfont}
\usepackage{ esint }
\usepackage[final]{ed}
\usepackage[final]{showlabels}
\usepackage{comment}
\usepackage{enumitem}
\usepackage{soul}
\usepackage{needspace}
\usepackage[normalem]{ulem}
\usepackage{tikz}
\usepackage{float}
\newcommand{\Hs}[1]{\text{\raisebox{.2cm}{${_{\!H_{\scl}^{#1}\!\!(M)}}$}}}
\newcommand{\sob}[1]{H_{\scl}^{#1}(M)}
\DeclareMathOperator{\scl}{scl}

\usepackage{multicol}
\usepackage{tabto}

\newcommand{\f}{f}
\newcommand{\TM}{T^*\!M}

\usepackage{xcolor}

\newtheorem{theorem}{Theorem}

\newtheorem{lemma}{Lemma}
\newtheorem{proposition}[lemma]{Proposition}

\theoremstyle{definition}
\newtheorem{definition}{Definition}
\newtheorem{remark}{Remark}

\newcommand{\ti}{\mathfrak{t}}
\newcommand{\FR}{\mathfrak{I}_{_{\!H}}}

\newcommand{\R}{{\mathbb R}}

\newcommand{\ep}{\varepsilon}

\newcommand{\Hl}{\Hs{\frac{k-3}{2}}}
\newcommand{\Hln}{\Hs{\!\!{\frac{n-3}{2}}}}

\newcommand{\Decay}{\mathbf{D}}

\newcommand{\Id}{\operatorname{Id}}
\newcommand{\mc}[1]{\mathcal{#1}}
\newcommand{\re}{\mathbb{R}}
\newcommand{\Span}{\operatorname{span}}

\newcommand{\ca}{\mathbf{B}}

\newcommand{\SNH}{S\!N^*\!H}

\newcommand{\sig}{\sigma_{_{\!\!S\!N^*\!H}}}

\newcommand{\SM}{S^*\!M}

\newcommand{\SigH}{\SNH}

\newcommand{\LambdaH}{\Lambda^\tau_{_{\!\SNH}}}

\newcommand{\LM}{{_{\!L^2(M)}}}

\newcommand{\class}{\delta}
\newcommand{\cm}{c_{_{\!M,p}}}

\newcommand{\w}{{\bf w}}

\renewcommand{\c}{\mathfrak{c}_0}

\def\XXint#1#2#3{{\setbox0=\hbox{$#1{#2#3}{\int}$} \vcenter{\hbox{$#2#3$}}\kern-.5\wd0}}

\DeclareMathOperator{\vol}{vol}

\DeclareMathOperator{\diam}{diam}

\DeclareMathOperator{\supp}{supp}
\DeclareMathOperator{\inj}{inj}

\newcommand{\ran}{\operatorname{ran}}
\newcommand{\e}{\varepsilon}
\newcommand{\Tinj}{\tau_{_{\!\text{inj}H}}}

\numberwithin{equation}{section}
\numberwithin{lemma}{section}

\title[]{Improvements for Eigenfunction Averages:\\An application of geodesic beams}

\author{Yaiza Canzani}
\address{Department of Mathematics, University of North Carolina at Chapel Hill, Chapel Hill, NC, USA}
\email{canzani@email.unc.edu}

\author{Jeffrey Galkowski}
\address{Department of Mathematics, University College London, London, UK}
\email{j.galkowski@ucl.ac.uk}

\date{}

\begin{document}

\begin{abstract}
Let $(M,g)$ be a smooth, compact Riemannian manifold and $\{\phi_\lambda \}$ an $L^2$-normalized sequence of Laplace eigenfunctions, $-\Delta_g\phi_\lambda =\lambda^2 \phi_\lambda$. Given a smooth submanifold $H \subset M$  of codimension $k\geq 1$, we find conditions on the pair $(M,H)$, even when $H=\{x\}$, for which
$$
\Big|\int_H\phi_\lambda d\sigma_H\Big|=O\Big(\frac{\lambda^{\frac{k-1}{2}}}{\sqrt{\log \lambda}}\Big)\qquad \text{or}\qquad  |\phi_\lambda(x)|=O\Big(\frac{\lambda ^{\frac{n-1}{2}}}{\sqrt{\log \lambda}}\Big),
$$
as $\lambda\to \infty$.
These conditions require no global assumption on the manifold $M$ and instead relate to the structure of the set of recurrent directions in the unit normal bundle to $H$. Our results extend {all} previously known conditions guaranteeing improvements on averages, including those on sup-norms. For example, we show that if $(M,g)$ is a surface with Anosov geodesic flow, then there are logarithmically improved averages for any $H\subset M$. We also find weaker conditions than having no conjugate points which guarantee $\sqrt{\log \lambda}$ improvements for the $L^\infty$ norm of eigenfunctions. Our results are obtained using geodesic beam techniques, which {yield} a mechanism for {obtaining} general quantitative improvements for averages and sup-norms.
\end{abstract}

\maketitle

\section{Introduction}

On a smooth compact Riemannian manifold without boundary of dimension $n$, $(M,g)$, we consider sequences of Laplace eigenfunctions $\{\phi_\lambda\}$ solving 
\[
(-\Delta_g-\lambda^2)\phi_\lambda=0,\qquad{\|\phi_\lambda\|_{L^2(M)}=1.}
\]
We study the average oscillatory behavior of $\phi_\lambda$ when restricted to a submanifold $H\subset M$ without boundary. { In {particular}, we examine the behavior of the integral average $\int_H\phi_\lambda d\sigma_H$ as $\lambda \to \infty$, where $\sigma_H$ is the volume measure on $H$ induced by the Riemannian metric. {Since} we allow $H$ to consist of a single point, our results include the study of sup-norms $\|\phi_\lambda\|_{_{L^\infty(M)}}$.

The study of these quantities has a long history. In general
\begin{equation}
\label{e:zelBound}
\int_H\phi_\lambda d\sigma_H=O(\lambda^{\frac{k-1}{2}}) \qquad \text{and} \qquad \|\phi_\lambda\|_{_{L^\infty(M)}}=O(\lambda^{\frac{n-1}{2}}),
\end{equation}
where $k$ is the codimension of $H$, and $H$ is any smooth embedded submanifold. 
The sup-norm bound in~\eqref{e:zelBound} is a consequence of the well known works~\cite{Ava,Lev,Ho68}.
The bound on averages was first obtained in~\cite{Good} and~\cite{Hej}, for the case in which $H$ is a periodic geodesic in a compact hyperbolic surface. The general bound in \eqref{e:zelBound} for integral averages was proved by Zelditch in \cite[Corollary 3.3]{Zel}.}

Since it is easy to find examples on the round sphere which saturate the estimate~\eqref{e:zelBound}, it is natural to ask whether the bound is typically saturated, and to understand conditions under which the estimate may be improved.

In~\cite{CG17,Gdefect,CGT, GT}, the authors (together with Toth in the latter two cases) gave bounds on {integral averages} based on understanding  microlocal concentration as measured by defect measures (see~\cite[Chapter 5]{EZB} or~\cite{Gerard} for a description of defect measures). In particular,~\cite{CG17} gave a new proof of~\eqref{e:zelBound} and studied conditions on $({\{\phi_\lambda\}},H)$ guaranteeing
\begin{equation}
\label{e:averageImproved1}
\int_H\phi_\lambda d\sigma_H =o\big(\lambda^{\frac{k-1}{2}}\big).
\end{equation} 
These conditions generalized and weakened the assumptions in~\cite{SZ02,SoggeTothZelditch,CS,SXZ,Wym,Wym2,Wym3, GT,Gdefect,CGT,Berard77,SZ16I,SZ16II} which guarantee at least the improvement~\eqref{e:averageImproved1}. However, the results in~\cite{CG17} neither recovered the bound
\begin{equation}
\label{e:averageImproved2}
\int_H\phi_\lambda d\sigma_H =O\Bigg(\frac{\lambda^{\frac{k-1}{2}}}{\sqrt{\log \lambda}}\Bigg), 
\end{equation}
obtained in~\cite{SXZ,Wym2,Wym18} under various conditions on $H$ when $M$ has non-positive curvature, nor recovered the improvement on sup-norms given in~\cite{Berard77, Bo16,Randol} when $k=n$ and $M$ has no conjugate points. In the present article, we address such quantitative improvements.

To the authors' knowledge, this article improves and extends \emph{all} existing bounds on averages over submanifolds for eigenfunctions of the Laplacian, including those on $L^\infty$ norms (without additional assumptions on the eigenfunctions; see Remark~\ref{r:extra} for more detail on other types of assumptions).  The estimates from~\cite{CG18d} imply those of~\cite{CG17} and therefore can be used to obtain all previously known improvements of the form~\eqref{e:averageImproved1}. In this article, we make the geometric arguments necessary to apply geodesic beam techniques and improve upon the results of~\cite{Wym18,Wym2,SXZ,Berard77,Bo16,Randol}.

These improvements are possible because the geodesic beam techniques developed in~\cite{CG18d} give {an} explicit bound on averages over submanifolds, $H$, which depend{s} only on microlocal information about $\phi_\lambda$ near the {unit} conormal {bundle} to $H$, {$\SNH$}. {In particular, microlocally near the conormal bundle to $H$, the quasimodes are decomposed into what we call geodesic beams: $\phi_\lambda=\sum_{j\in \mc{J}}\chi_{_{\mathcal{T}_j}}\phi_\lambda$ near $H$. Each geodesic beam, $\chi_{_{\mathcal{T}_j}}\phi_\lambda$, is obtained by localizing $\phi_\lambda$ to a length $\sim 1$ geodesic tube $\mathcal{T}_j$ of radius  $R(\lambda)\sim \lambda^{-{1}/{2}+\delta}$ around a geodesic through $\SNH$. The contributions of these tubes are then estimated using an energy estimate due to Koch--Tataru--Zworski~\cite{KTZ}. After recombining, the estimate reads (for the case $H=\{x\}$)
$$
|\phi_{\lambda}(x)|\leq CR(\lambda)^{(n-1)/2}\lambda^{(n-1)/2}\sum_{j\in \mc{J}} \|\chi_{_{\mc{T}_j}} \phi_\lambda\|_{L^2(M)}.
$$
}
This estimate requires no assumptions on the geometry of $H$ or $M$ and is purely local. It is only with this bound in place that~\cite{CG18d} applies Egorov's theorem to {$\log \lambda$ time in order to} obtain a purely dynamical estimate (see also Theorem~\ref{t:coverToEstimate}) of the {form
\begin{equation}
\label{e:basicNonlooping}
|\phi_{\lambda}(x)|\leq CR(\lambda)^{(n-1)/2}\lambda^{(n-1)/2}\Big(|\mc{B}|^{1/2}+\frac{|\mc{G}|^{1/2}}{|\log \lambda|^{1/2}}\Big)\|\phi_{\lambda}\|_{L^2(M)},
\end{equation}
where $\cup_{j\in \mc{G}} \mc{T}_j$ is non-self looping for $\log \lambda$  time (see~\eqref{e:nonsl}) and $\mc{J}=\mc{G}\cup \mc{B}$. } {See Section~\ref{s:geodesicBeams} for a more detailed explanation of the techniques which includes estimates similar to~\eqref{e:basicNonlooping} which allow for multiple non-looping sets, and~\cite{CG18d} for the proofs of these analytic statements}.

In this article, we apply dynamical arguments to draw conclusions about the pairs $((M,g),H)$ supporting eigenfunctions with maximal averages. While previous works on eigenfunction averages rely on explicit parametrices for the kernel of the half wave-group for large times, the authors' techniques~\cite{GT,Gdefect,CGT,CG17,CG18d}, show that improvements can be effectively obtained by understanding the microlocalization properties of eigenfunctions. 

\begin{remark}
\label{r:extra}
Note that in this paper we study averages of relatively weak quasimodes for the Laplacian with no additional assumptions on the functions. This is in contrast with results which impose additional conditions on the functions such as: 
that they be Laplace eigenfunctions that simultaneously satisfy additional equations~\cite{I-S,GT18a,Ta18}; that they be eigenfunctions in the very rigid case of the flat torus~\cite{B93,Gros}; or that they form a density one subsequence of Laplace eigenfunctions~\cite{JZ}.
\end{remark}

We now state the main results of this article. In order to match the language of~\cite{CG18d}, we will semiclassically rescale, setting $h=\lambda^{-1}$ and sending $h\to 0^+$. Relabeling, $\phi_\lambda$ as $\phi_h$, {the eigenfunction equation becomes}
$$
(-h^2\Delta_g-1)\phi_h=0,\qquad \|\phi_h\|_{L^2}=1.
$$ 

We also recall the notation {for the semiclassical Sobolev norms: } 
\begin{equation}
\label{e:sobolev}
\|u\|_{_{\sob{s}}}^2:=\big\langle (-h^2\Delta_g+1)^{s}u,u\big\rangle_{_{\!L^2(M)}}.
\end{equation}

Let ${\Xi}$ denote the collection of maximal unit speed geodesics for $(M,g)$. {For $m$ a positive integer, $r>0$,  $t\in \R$, and $x \in M$}  define
$$
{\Xi}_x^{m,r,t}:=\big\{\gamma\in \Xi: \gamma(0)=x,\,\exists\text{ at least }m\text{ conjugate points to } x \text{ in }\gamma(t-r,t+r)\big\},
$$
where we count conjugate points with multiplicity. Next, for  a set $V \subset M$ write
$$
\mc{C}_{_{\!V}}^{m,r,t}:=\bigcup_{x\in V}\{\gamma(t): \gamma\in \Xi_x^{m,r,t}\}.
$$
{Note that if $r_t \to 0^+$ as $|t|\to \infty$, then saying that $x \in \mc{C}_x^{n-1,r_t,t}$ for $t$ large indicates that $x$ behaves like  a point that is maximally self-conjugate. This is the case for every point on the sphere.  The following result applies under the assumption that this does not happen and obtains quantitative improvements in that setting. }

\begin{theorem}
\label{t:noConj2}
Let $V \subset M$ and assume that there exist $t_0>0$ and $a>0$  so that 
$$
\inf_{x\in V}d\big(x, \mc{C}_{x}^{n-1,r_t,t}\big)\geq r_t,\qquad\text{ for } t\geq t_0
$$
with $r_t=\frac{1}{a}e^{-at}.$ 
Then, there exist $C>0$ and $h_0>0$ so that for $0<h<h_0$ and $u \in {\mc{D}'}(M)$
$$
\|u\|_{L^\infty(V)}\leq Ch^{\frac{1-n}{2}}\left(\frac{\|u\|_{\LM}}{\sqrt{\log h^{-1}}}\;+\; \frac{\sqrt{\log h^{-1}}}{h}\big\|(-h^2\Delta_g-1)u\big\|_{\Hln}\right).
$$
\end{theorem}

In fact a generalization of Theorem~\ref{t:noConj2} holds not just for $H=\{x\}$, but for any $H\subset M$ of large enough codimension.

\begin{theorem}
\label{t:noConj1}
 Let {$H\subset M$ be a closed embedded submanifold of codimension $k>\frac{n+1}{2}$} and assume that there exist $t_0>0$ and $a>0$  such that
\begin{equation}
\label{e:dist}
d\big(H, \mc{C}_H^{2k-n-1,r_t,t}\big)\geq r_t,\,\qquad \text{ for }t\geq t_0
\end{equation}
with $r_t:=\frac{1}{a}e^{-at}.$
 Then, there exists $C>0$, so that for all $w\in C_c^\infty(H)$ the following holds. There exists $h_0>0$ such that for all $0<h<h_0$ and $u\in \mc{D}'(M)$,
\begin{equation}
\label{e:goalEst}
\Big|\int_H wu d\sigma_{_H}\Big|\leq Ch^{\frac{1-k}{2}}{\|w\|_{_{\!\infty}}}\left(\frac{\|u\|_{\LM}}{\sqrt{\log h^{-1}}}\;+\;\frac{\sqrt{\log h^{-1}}}{h}\big\|(-h^2\Delta_g-1)u\big\|_{\Hl}\right).
\end{equation}
\end{theorem}

\begin{remark}
One should think of the assumption in Theorem~\ref{t:noConj2} as ruling out maximal self-conjugacy of a point with itself uniformly up to time $\infty$. In fact, in order to obtain an $L^\infty$ bound of $o(h^{\frac{1-n}{2}})$ on $u(x)$, it is enough to assume that there is not a positive measure set of directions $A\subset S^*_xM$ so that for each element $\xi\in A$ there is a sequence of geodesics starting at $x$ in the direction of $\xi$ with length tending to infinity along which $x$ is maximally conjugate to itself. 
\end{remark}

Before stating our next theorem, we recall that if $(M,g)$ has strictly negative sectional curvature, then it also has  Anosov geodesic flow~\cite{Anosov}.  Also, both Anosov geodesic flow and non-positive sectional curvature imply that $(M,g)$ has no conjugate points ~\cite{Kling}. 

When $(M,g)$ is non-positively curved (indeed when it has no focal points), if every geodesic encounters a point of negative curvature, then $(M,g)$ has Anosov geodesic flow~\cite[Corollary 3.4]{Eberlein73}. In particular, there are manifolds for which the curvature is positive in some places {while} the geodesic flow is Anosov. However, even in non-positive curvature some geodesics may fail to encounter negative curvature and thus the geodesic flow may not be Anosov. To study this situation, we introduce an integrated curvature condition inspired by that in~\cite{SXZ}: There are  $T>0$, and $c_{_{\!K}}>0$ so that for {every geodesic $\gamma$} of length $t\geq T$ in the universal cover $(\tilde{M}, \tilde g)$  {of $(M,g)$}, {and for all} $0\leq s\leq 1$, 
\begin{equation}
\label{e:intCurve}
\int_{\Omega_{\gamma}(s)}\!\!Kdv_{\tilde g}\leq -c_{_{\!K}}{e^{-\frac{1}{c_{_{\!K}}\sqrt{s}}}}
\end{equation}
where 
$
{\Omega_{\gamma}(s)}:=\{x\in \tilde{M}:\; d(x,\gamma)\leq s\},
$
{and $K$ is the scalar curvature for {$(\tilde M,\tilde g)$}.}
{Note that, unlike the {curvature} conditions in~\cite{SXZ}, {the assumption in} ~\eqref{e:intCurve} allows the curvature to vanish in open sets so long as no geodesic lies entirely in such an open set. Moreover, it allows the curvature to vanish to infinite order at the geodesic.}

\begin{theorem}\label{t:surfaces}
Let $(M,g)$ be a smooth, compact Riemannian surface. Let $H\subset M$ be a closed embedded {curve or a point}. Suppose one of the following assumptions holds{:}
\begin{enumerate}[label=\textbf{\Alph*.},ref=\ref{t:surfaces}.\Alph*]
\item  \label{a3}$(M,g)$ has Anosov geodesic flow.  \smallskip 
\item  \label{a7} $(M,g)$ {has non-positive curvature and satisfies the integrated curvature condition~\eqref{e:intCurve}, and $H$ is a geodesic.} 
\end{enumerate}
Then, there exists $C>0$ so that for all $w\in C_c^\infty(H)$ the following holds. There is $h_0>0$ so that for $0<h<h_0$ and $u\in \mc{D}'(M)$
\begin{equation}
\label{e:subEstSurface}
\Big|\int_Hwud\sigma_H\Big|\leq Ch^{\frac{1-k}{2}}\|w\|_{\infty}\Big(\frac{\|u\|_{\LM}}{\sqrt{\log h^{-1}}}+\frac{\sqrt{\log h^{-1}}}{h}\|(-h^2\Delta_g-1)u\|_{\Hl}\Big).
\end{equation}
\end{theorem}
\begin{remark}
In fact, the proof Theorem~\ref{a7} shows that it is enough to have~\eqref{e:intCurve} for every geodesic $\gamma$ normal to $H$.
\end{remark}

For manifolds of arbitrary dimensions, we also obtain quantitative improvements for averages in a variety of situations.

\begin{theorem}\label{T:applications}
 Let $(M,g)$ be a smooth, compact Riemannian manifold of dimension $n$ and $H\subset M$ be a closed embedded submanifold of codimension $k$. Suppose one of the following assumptions holds{:}
\begin{enumerate}[label=\textbf{\Alph*.},ref=\ref{T:applications}.\Alph*]
\item  \label{a1} $(M,g)$ has no conjugate points and $H$ has codimension $k>\frac{n +1}{2}$. \smallskip 
\item  \label{a2}$(M,g)$ has no conjugate points and $H$ is a geodesic sphere.\smallskip 
\item  \label{a6}$(M,g)$ is non-positively curved and has Anosov geodesic flow, and $H$ has codimension $k>1$.  \smallskip 
\item  \label{a4}$(M,g)$ is non-positively curved and has Anosov geodesic flow, and $H$ is totally geodesic.  \smallskip 
\item  \label{a5} $(M,g)$ has {Anosov geodesic flow} and $H$ is a subset of $M$ that lifts to a horosphere in the universal cover. 
\end{enumerate}
Then, there exists $C>0$ so that for all $w\in C_c^\infty(H)$ the following holds. There is $h_0>0$ so that for $0<h<h_0$ and $u\in \mc{D}'(M)$
\begin{equation}
\label{e:subEst}
\Big|\int_Hwud\sigma_H\Big|\leq Ch^{\frac{1-k}{2}}\|w\|_{\infty}\Big(\frac{\|u\|_{\LM}}{\sqrt{\log h^{-1}}}+\frac{\sqrt{\log h^{-1}}}{h}\|(-h^2\Delta_g-1)u\|_{\Hl}\Big).
\end{equation}
\end{theorem}
We note here that {Theorem} {\ref{a7}} includes the bounds of~\cite{SXZ} as a special case {(see Remark \ref{r:SXZ16} for an explanation)}.  The bounds in~\cite{Wym2,Wym18} are special cases of {Theorem} {\ref{a3}}, {Theorem} {\ref{a6}}, and the results of Theorem~\ref{T:tangentSpace} {below (see the discussion that follows Theorem~\ref{T:tangentSpace}}). {We also note {that} for any smooth compact embedded submanifold, $H_0 \subset M$, satisfying one of the conditions in Theorem~\ref{T:applications}, there is a neighborhood $U$ of $H_0$, in the $C^\infty$ topology, so that the constants $C$ and $h_0$ in Theorem 
\ref{T:applications} are uniform over $H\in U$ and $w$ taken in a bounded subset of {$C_c^\infty(H)$}.} In particular, the sup-norm bounds from~\cite{Berard77,Bo16,Randol} are a special case of {Theorem}~ \ref{a1}. Similar to the $o(h^{\frac{1-k}{2}})$ bounds in~\cite{CG17}, we conjecture that~\eqref{e:subEst} holds whenever $(M,g)$ is a manifold with Anosov geodesic flow, regardless of the geometry of $H$.

{Geodesic beam techniques can also be used to study $L^p$ norms of eigenfunctions~\cite{CG18b} and to give quantitatively improved remainder estimates for the kernel of the spectral projector and for Kuznecov sum type formulae~\cite{CG18c}. The authors are currently studying how to give polynomial improvements for $L^\infty$ norms on certain manifolds with integrable geodesic flow. To our knowledge the only other case where polynomial improvements are available is in~\cite{I-S} for Hecke--Maase forms on arithmetic surfaces or when $(M,g)$ is the flat torus~\cite{B93,Gros}. }

\subsection{Results on geodesic beams}\label{s:geodesicBeams}

The main estimate from~\cite{CG18d} gives control on eigenfunction averages in terms of microlocal data. We now review the necessary notation to state that result.
%

Let $p(x,\xi)=|\xi|_{g(x)}$ defined on $T^*M$
and {consider the geodesic flow} on $T^*M$,
\begin{equation}\label{e:varphi}
\varphi_t:=\exp(tH_p).
\end{equation}
Next, fix a hypersurface
\begin{equation}
\label{e:Hsig}
\mc{H}_{\Sigma}\subset \TM \text{  transverse to }H_p\;\text{ with }\SNH\subset \mc{H}_{\Sigma},
\end{equation}
define $\Psi:\re\times \mc{H}_{\Sigma}\to \TM $ by $\Psi(t,q)=\varphi_t(q)$, and let
\begin{equation}\label{e:Tinj}
\Tinj:=\sup\{{\tau \leq 1}:\; \Psi|_{(-\tau,\tau)\times\mc{H}_{\Sigma}}\text{ is injective}\}.
\end{equation}
Given $A \subset \TM $ define
$$\Lambda_{_{\!A}}^\tau:=\bigcup_{|t|\leq \tau}\varphi_t(A).$$
For $r>0$ and $A\subset \SigH$ we define
\begin{equation}\label{e:tube}
{\Lambda_{_{\!A}}^\tau(r):=\Lambda_{A_r}^{\tau+r},\qquad A_r:=\{\rho\in \mc{H}_{\Sigma}:d(\rho,A)<r\}.}
\end{equation}
where $d$ denotes the distance induced by the Sasaki metric on $TM$ (see e.g. Appendix~\ref{s:jacobi} or~\cite[Chapter 9]{BlairSasaki} for an explanation of the Sasaki metric).

Throughout the paper we adopt the notation
\begin{equation}\label{e:KH}
{K}_{_{\!H}}>0
\end{equation}
for a constant so that all sectional curvatures of $H$ are bounded by $K_{_{\!H}}$ {and the second fundamental form of $H$ is bounded by $K_{_{\!H}}$}. Note that when $H$ is a point, we may take $K_{_{\!H}}$ to be arbitrarily close to $0$.

{We next recall~\cite[{Theorem 11}]{CG18d} which controls eigenfunction averages by covers of $\Lambda^\tau_{\SNH}(h^\delta)$ by ``good" tubes that are non self-looping and ``bad" tubes whose {number} is controlled. In fact, Theorems~\ref{t:noConj2},~\ref{t:noConj1}, and~\ref{T:applications}  are reduced to a purely dynamical argument together with an application of Theorem~\ref{t:coverToEstimate}.}

For $0<t_0<T_0$, we say that $A\subset T^*\!M$ is \emph{$[t_0,T_0]$ non-self looping} if 
\begin{equation}\label{e:nonsl}
\bigcup_{t=t_0}^{T_0}\varphi_t(A)\cap A=\emptyset\qquad \text{ or }\qquad  \bigcup_{t=-T_0}^{-t_0}\varphi_t(A)\cap A=\emptyset.
\end{equation}
We define the \emph{maximal expansion rate }
\begin{equation}\label{e:Lmax}
\Lambda_{\max}:=\limsup_{|t|\to \infty}\frac{1}{|t|}{\log} \sup_{{\SM}}\|d\varphi_t(x,\xi)\|.
\end{equation}
Then, the Ehrenfest time at frequency $h^{-1}$ is 
\begin{equation}\label{e:Tehr}
T_e(h):=\frac{\log h^{-1}}{2\Lambda_{\max}}.
\end{equation}
Note that $\Lambda_{\max}\in[0,\infty)$ and if $\Lambda_{\max}=0$, we may replace it by an arbitrarily small positive constant.

\begin{definition} \label{d: cover} Let $A\subset \SigH$, $r>0$, {$\tau>0$}, and  $\{\rho_j\}_{j=1}^{N_r} \subset A$. We say that the collection of tubes $\{\Lambda_{\rho_j}^\tau(r)\}_{j=1}^{N_r}$ is a \emph{$(\tau, r)$-cover} of a set $A\subset  \SigH$ provided
 $$\Lambda_A^\tau(\tfrac{1}{2}r) \subset\bigcup_{j=1}^{N_r}\Lambda_{\rho_j}^{\tau}(r).$$
\end{definition}
{
{It will often be useful to have a notion of $(\tau,r)$ cover of $\SNH$ without too many overlapping tubes. To that end, we make the following definition.}
\begin{definition}
\label{d:good cover}
Let $A\subset \SigH$, {$r>0$}, $\mathfrak{D}>0$, and  $\{\rho_j\}_{j=1}^{N_r} \subset A$. We say that the collection of tubes $\{\Lambda_{\rho_j}^\tau(r)\}_{j=1}^{N_r}$ is a \emph{$(\mathfrak{D},\tau, r)$-good cover} of a set $A\subset  \SigH$ provided that it is a $(\tau,r)$-cover for $A$ and there exists a partition $\{\mathcal{J}_\ell\}_{\ell=1}^{\mathfrak{D}}$ of $\{1, \dots, N_r\}$ so that for every $\ell\in \{1, \dots, \mathfrak{D}\}$
    \[
    \Lambda_{\rho_j}^\tau (3r)\cap \Lambda_{\rho_i}^\tau(3r)=\emptyset\qquad i,j\in \mathcal{J}_\ell, \quad  i\neq j.
    \]
\end{definition}
\noindent We recall that~\cite[Proposition 3.3]{CG18d} shows the existence of $\mathfrak{D}_n>0$, depending only on $n$, so that for all sufficiently small $(\tau,r)$ there are of $(\mathfrak{D}_n,\tau, r)$ good covers of $\SNH$. We will use this fact freely throughout this article.}


{For convenience we state~\cite[{Theorem 11}]{CG18d}}.
The theorem involves many parameters. These provide flexibility when applying the theorem, but make the statement involved. We refer the reader to the comments after the statement of the theorem for a heuristic explanation of its contents.

\begin{theorem}[{\cite[{Theorem 11}]{CG18d}}]
\label{t:coverToEstimate}
{Let $H\subset M$ be a submanifold of codimension $k$.}
 Let $0<\delta<\frac{1}{2}$, $N>0$ and ${\{w_h\}_h}$ with {$w_h\in S_\class \cap C_c^\infty( H)$}. There exist positive constants $\tau_0=\tau_0(M,g,\Tinj ,H)$, {$R_0=R_0(M,g, K_H,{k,\Tinj })$,} $C_{n,k}$ depending only on $n$ and $k$, and $h_0=h_0(M,g,{\delta, {H}})$, and for each $0<\tau\leq \tau_0$  there exist $C=C(M,g,\tau,\delta,, H)>0$ and $C_{_{\!N}}=C_{_{\!N}}(M,g, N,\tau,\delta,{\{w_h\}_h},H)>0$,  so that the following holds.

Let $8h^\delta\leq R(h)\leq R_0$,{ $0\leq\alpha< 1-2{\limsup_{h\to 0}\frac{\log R(h)}{\log h}}$,} and suppose $\{\Lambda_{_{\rho_j}}^\tau(R(h))\}_{j=1}^{N_h}$  is a {$(\mathfrak{D},\tau, R(h))$ cover of $\SNH$ for some $\mathfrak{D}>0$}.

In addition, suppose there exist $\mc{B}\subset \{1,\dots, N_h\}$ and a finite collection  $\{\mc{G}_\ell\}_{\ell \in \mathcal L} \subset \{1,\dots, N_h\}$ with 
$$
\mathcal J_h(w_h)\;\subset\;  \mc{B} \cup \bigcup_{\ell \in \mathcal L}\mc{G}_\ell,
$$
where
\begin{equation}\label{e:i}
\mathcal J_h(w_h):=\{j:\; \Lambda_{_{\!\rho_j}}^\tau(2R(h))\cap \pi^{-1}(\supp w_h)\neq \emptyset\},
\end{equation}
 and so that for every $\ell \in \mathcal L$ there exist  $t_\ell=t_\ell(h)>0$ and ${T_\ell=T_\ell(h)}\leq {2} \alpha T_e(h)$  so that
$$
\bigcup_{j\in \mc{G}_\ell}\Lambda_{_{\rho_j}}^\tau(R(h))\;\;\text{ is }\;\;[t_\ell,T_{\ell}]\text{ non-self looping for }\varphi_t:=\exp(tH_{|\xi|_g}).
$$
Then, for $u\in \mc{D}'(M)$ and $0<h<h_0$,
\begin{align*}
h^{\frac{k-1}{2}}\Big|\int_{H} w_h u\, d\sigma_{H}\Big|
&\leq\frac{C_{n,k}{\mathfrak{D}}\|w_h\|_{_{\!\infty}}R(h)^{\frac{n-1}{2}}}{\tau^{\frac{1}{2}}}
\Bigg(|\mc{B}|^{\frac{1}{2}}+\sum_{\ell \in \mathcal L }\frac{(|\mc{G}_\ell|t_\ell)^{\frac{1}{2}}}{T^{\frac{1}{2}}_\ell}\Bigg)\|u\|_{\LM} \\
&+\frac{C_{n,k}{\mathfrak{D}}\|w_h\|_{_{\!\infty}}R(h)^{\frac{n-1}{2}}}{\tau^{\frac{1}{2}}} \sum_{\ell \in \mathcal L}\frac{(|\mc{G}_\ell|t_\ell T_\ell)^{\frac{1}{2}}}{h}\;\|(-h^2\Delta_g-1)u\|_{\LM}\! \\
&+Ch^{-1}{\|w_h\|_\infty}\|(-h^2\Delta_g-1)u\|_{\Hl}\\
&+C_{_{\!N}}h^N\big(\|u\|_{\LM}+{\|(-h^2\Delta_g-1)u\|_{\Hl}}\big).
\end{align*}
Here, the constant $C_{_{\!N}}$ depends on $\{w_h\}_h$  only through finitely many $S_\delta$ seminorms of $w_h$. {The constants ${\tau_0},C,C_{_{\!N}},h_0$ depend on $H$ only through finitely many derivatives of its curvature and second fundamental form.}
\end{theorem}

{
\begin{remark}
The estimates in Theorem \ref{t:coverToEstimate} are uniform in $H$. For a precise description see \cite[{Theorem 11}]{CG18d}. In particular, when $H=\{x\}$ and $w=1$, then $k=0$ and $|\int_{H} w_h u\, d\sigma_{H}|$ is replaced with $\|u\|_{L^\infty(B(x, h^\delta))}$.
\end{remark}
}
\medskip

Theorem~\ref{t:coverToEstimate} reduces estimates on averages to construction of covers of $\LambdaH(h^\delta)$ by sets with appropriate structure. To understand the statement, we first ignore the extra structure requirement and assume $(-h^2\Delta_g-1)u=0$. With these simplifications, and ignoring {an $h^\infty\|u\|_{\LM}$ term,} if there is a cover of $\LambdaH(h^\delta)$ by ``good" sets $\{G_\ell(h)\}_{\ell\in L}$ and a ``bad" set $B(h)$ with $G_\ell$, $[t_\ell(h),T_\ell(h)]$ non-self looping, the estimate reads
\begin{multline*}
h^{\frac{k-1}{2}}\Big|\int_H w ud\sigma_H\Big|
\leq \frac{C_{n,k}\|w\|_{_{\!\infty}}}{\tau^{\frac{1}{2}}}
\left(
\![\sig(B)]^{\frac{1}{2}}
+\sum_{\ell \in \mathcal L }\frac{[\sig(G_\ell)]^{\frac{1}{2}}t_\ell^{\frac{1}{2}}}{T^{\frac{1}{2}}_\ell(h)}\right)\!\!\|u\|_{\LM},\\
\end{multline*}
where $\sig$ denotes the volume induced on $\SigH$ by the Sasaki metric on $T^*\!M$ and for $A\subset T^*\!M$, we write $\sig(A)=\sig(A\cap \SigH)$. The additional structure required on the sets $G_\ell$ and $B$ is that they consist of a union tubes $\Lambda_{{\rho_i}}^\tau(h^\delta)$ for some $0\leq \delta<\frac{1}{2}$ and that $T_\ell(h)<2(1-2\delta)T_e(h)$. 
With this in mind, Theorem~\ref{t:coverToEstimate} should be thought of as giving non-recurrent condition on $\SigH$ which guarantees quantitative improvements over~\eqref{e:zelBound}. {This type of non-recurrence was exploited in~\cite{GT18a} to understand $L^\infty$ norms for eigenfunctions at the umbillic points of the tri-axial ellipsoid, a \emph{quantum-completely integrable} situation.}  Taking $t_\ell$, $T_\ell$, $G_\ell$ and $B$ to be $h$-independent can be used to recover the dynamical consequences in~\cite{CG17,Gdefect} (see~\cite{GJEDP}).

\begin{remark}
Note that it is possible to use Theorem~\ref{t:coverToEstimate} to obtain quantitative estimates which are strictly between $O(h^{\frac{1-k}{2}})$ and $O(h^{\frac{1-k}{2}}/\sqrt{\log h^{-1}})$. For example, this happens if $r_t$ is replaced by e.g. $a^{-1}e^{-a t^2}$ in~\eqref{e:dist}. We expect that the construction in~\cite{BuPa} can be used to generate examples where this type of behavior is optimal.
\end{remark}

\subsection{Manifolds with {no focal points} or Anosov geodesic flow}

{In parts {\ref{a3}}, {\ref{a6}}, {\ref{a4}} and {\ref{a5}} of Theorem \ref{T:applications} we assume either that $(M,g)$ has no focal points or that it has Anosov geodesic flow.} We show that these structures allow us to construct non-self looping covers away from the points $\mc{S}_H \subset \SNH$ at which the tangent space to $\SNH$ splits into a sum of  stable and unstable directions. To make this sentence precise we introduce some notation.

If $(M,g)$ has no conjugate points, then for any $\rho \in \SM$ there exist a {weak} stable subspace $E^{{w}}_{+}(\rho)\subset T_{\rho}\SM$ and a {weak} unstable subspace $E^{{w}}_{-}(\rho)\subset T_\rho \SM$ so that 
\[
d\varphi_t :E^{{w}}_\pm(\rho) \to E^{{w}}_\pm(\varphi_t(\rho)),
\]
and 
\[
 |d\varphi_t({\bf{v}})|\leq C|{\bf{v}}| \;\;  \text{for}\;  {\bf{v}}\in E^{{w}}_\pm\;\; \text{and}\; t\to  \pm\infty.
\]
{(see e.g.~\cite[Proposition 2.13]{Eberlein73} which is based on~\cite{Green})}
{We also define the stable ($+$) and unstable ($-$) subspaces as $E_{\pm}(\rho)=E_{\pm}^w(\rho)\cap (\mathbb{R}H_p)^\perp$ where the orthogonal complement is taken with respect to the Sasaki metric.
These subspaces then} have the property that 
$$
T_\rho \SM=(E_+(\rho)+E_-(\rho))\oplus \re H_p(\rho).
$$
{While this particular decomposition happens to be an orthogonal sum, throughout the article we will use $A=A_1\oplus A_2$ to mean direct sum i.e. that $A=A_1+A_2$ and $A_1\cap A_2=\{0\}$.}

{We recall that a manifold has no focal points if for every geodesic $\gamma$, and every Jacobi field $Y(t)$ along $\gamma$ {with $Y(0)=0$ and $Y'(0)\neq 0$}, $Y(t)$ satisfies   $\tfrac{d}{dt}\| Y(t)\|^2>0$ for $t>0$, where $\|\cdot \|$ denotes the norm with respect to the Riemannian metric. {In particular, if $(M,g)$ has non-positive curvature, then it has no focal points (see e.g. \cite[page 440]{Eberlein73})}. It is also known that if $(M,g)$ has no focal points then {$(M,g)$ has no conjugate points and that}
 $E_\pm(\rho)$ vary continuously with $\rho$. (See for example \cite[Proposition 2.13 {and remarks thereafter}]{Eberlein73}.) {See e.g.~\cite{Ruggiero,Eberlein73b,Pesin} for further discussions of manifolds without focal points. }}
 
 {The geodesic flow is said to be Anosov~\cite{Anosov} if there exist $E_{\pm}(\rho)\subset T_\rho\SM$ and} $\ca>0$ so that for all $\rho\in \SM$,
\begin{equation}\label{e:Bdef}
 |d\varphi_t({\bf{v}})|\leq \ca e^{\mp \frac{t}{\ca}}|{\bf{v}}|,\qquad{ {\bf{v}}\in E_\pm(\rho),\quad t\to  \pm\infty,}
\end{equation}
{and 
\begin{equation}
\label{e:directSplitting}
T_\rho \SM=E_+(\rho)\oplus E_-(\rho)\oplus \mathbb{R} H_p.
\end{equation}}
{Recall that a manifold with Anosov geodesic flow does not have conjugate points~\cite{Kling} and hence we use the same notation $E_{\pm}(\rho)$ as in that case. In fact, a manifold has Anosov geodesic flow if and only if it has no conjugate points and~\eqref{e:directSplitting} holds~\cite[Theorem 3.2]{Eberlein73}.} {One consequence of having} Anosov geodesic flow is that the spaces $E_\pm(\rho)$ are H\"older continuous in $\rho$~\cite[Theorem 19.1.6]{KatokHasselblatt}. 

{In order to find examples of manfiolds with Anosov geodesic flow, we recall that any manifold with no focal points in which every geodesic encounters a point of negative curvature has Anosov geodesic flow \cite[Corollary 3.4]{Eberlein73}. In particular, the class of manifolds with Anosov geodesic flows includes those with negative curvature~\cite{Anosov}.}

Below we write  
\begin{equation}\label{e:N pm}
N_{\pm}(\rho):=T_{\rho}(\SNH)\cap E_\pm(\rho),
\end{equation} 
and define {the \emph{mixed} and \emph{split} subsets of $\SNH$ respectively by}
\begin{align}
\mc{M}_H&:=\Big\{\rho \in \SNH:\:\, N_-(\rho)\neq \{0\}\text{ and }N_+(\rho)\neq \{0\}\Big\}, \label{e:MH}\\
\mc{S}_H&:= \Big\{\rho\in \SNH:\;\, T_\rho (\SNH)=N_-(\rho)+N_+(\rho)\Big\}.\label{e:SH}
\end{align}
Then we write
\begin{equation}\label{e:AH}
\mc{A}_H:={\mc{M}_H\cap \mc{S}_H}
\end{equation}
 where we will use $\mc{A}_H$ when considering manifolds with Anosov geodesic flow and {$\mc{S}_H$} when considering those with no focal points.
 
{In what follows, $\pi$ continues to be the canonical projection $\pi:\SNH \to H$.}
\begin{theorem}\label{T:tangentSpace}
Let $H\subset M$ be a closed embedded submanifold of codimension $k$. Suppose that $A\subset H$ and one of the following two conditions holds:
\begin{itemize}
\item  $(M,g)$ has {no focal points} and~$\pi^{-1}(A)\cap{\mc{S}}_H=\emptyset$.
\item $(M,g)$ has Anosov geodesic flow and $\pi^{-1}(A)\cap \mc{A}_H=\emptyset$.
\end{itemize}
Then, there exists $C>0$ so that for all $w\in C_c^\infty(H)$ with $\supp w\subset A$ the following holds. There exists $h_0>0$ so that for $0<h<h_0$ and $u\in \mc{D}'(M)$
\begin{equation*}
\Big|\int_Hwud\sigma_H\Big|\leq C h^{\frac{1-k}{2}}\|w\|_\infty\Bigg(\frac{\|u\|_{\LM}}{\sqrt{\log h^{-1}}}+\frac{\sqrt{\log h^{-1}}}{h}\|(-h^2\Delta_g-1)u\|_{\Hl}\Bigg).
\end{equation*}

\end{theorem}

\noindent {Theorem~\ref{T:tangentSpace} also comes with some uniformity over the constants $(C,h_0)$. In particular, for $(A_0,H_0)$ satisfying one of the conditions in Theorem~\ref{T:tangentSpace}, there is a neighborhood {$U$} of $(A_0,H_0)$ in the $C^\infty$ topology so that the constants $(C,h_0)$ are uniform for $(A,H)\in U$ and $w$ in a bounded subset of $C_c^\infty$. {Here and below when we refer to the $C^\infty$ topology on $(A,H)$ we mean the following. Fix coordinate charts $\{U_j\}_j$ near $H_0$ such that $H_0\subset \cup_j U_j$ and in each $U_j$, $H_0$ is given by $\{(x',x'')\mid x'=0\}$. We define a neighborhood basis near $(A_0,H_0)$ by saying for $\e,k$ that $(A,H)$ is $\e$ close to $H_0$ if $H$ is given by $\{(x',x'')\mid x'=f(x'')\}$ for some $f\in C^k$ with $\|f\|_{C^k}\leq \e$ and $$\sup_{x\in A}\inf_{y\in A_0}d(x,y)+\sup_{x\in A_0}\inf_{y\in A}d(x,y)<\e.$$ Note in particular that since $E_{\pm}(\rho)$ are continuous in $\rho$, if $(A_0,H_0)$ satisfies the assumptions of Theorem~\ref{T:tangentSpace}, then for $\e>0$ small enough, $k$ large enough, and $(A,H)$, $\e,k$ close to $(A_0,H_0)$, the pair $(A,H)$ satisfies the assumptions of Theorem~\ref{T:tangentSpace}. }

We note that  the conclusion of Theorem \ref{T:tangentSpace} holds when $(M,g)$ is a surface with Anosov geodesic flow, since in this case $\mc{A}_H=\emptyset$ {regardless of $H$}. To see this note that if $\dim M=2$, then $\mc{S}_H=\mc{A}_H$ since $\dim T_{\rho}(\SNH)=1$. {Indeed}, it is not possible to have both $N_{+}(\rho)\neq\{0\}$ and $N_{-}(\rho)\neq\{0\}$ {unless $N_+(\rho)=N_-(\rho)=T_\rho(\SNH)$} and hence $\mc{S}_H\subset \mc{A}_H$. Moreover, in the Anosov case, since $E_+(\rho)\cap E_-(\rho)=\{0\}$, $\mc{A}_H=\emptyset.$ 

In~\cite{Wym,Wym2} Wyman works with $(M,g)$ non-positively curved (and hence having no focal points), $\dim M=2$ and $H=\{\gamma(s)\}$ a curve.  He then imposes the condition that for all $s$ the curvature of $\gamma$, $\kappa_\gamma(s)$, avoids two special values  $\mathbf{k}_\pm(\gamma'(s))$ determined by the tangent vector to $\gamma(s)$. He shows that under this condition, when $\phi_h$ is an eigenfunction of the Laplacian,
\[
\int_\gamma \phi_hd\sigma_\gamma=O\Big(\frac{1}{\sqrt{\log h^{-1}}}\Big).
\] 
We note that if $\kappa_\gamma(s)=\mathbf{k}_{\pm}(\gamma'(s))$, then the lift of $\gamma$ to the universal cover of $M$ is tangent to a stable or unstable horosphere at $\gamma(s)$, and $\kappa_\gamma(s)$ is equal to the curvature of that horosphere. Since this implies that $T_{(\gamma(s),\gamma'(s))}SN^*\gamma$ is stable or unstable, the condition there is that $\mc{S}_\gamma=\emptyset.$ Thus, the condition $\mc{S}_H=\emptyset$ is the generalization to higher codimensions and more general geometries of that in~\cite{Wym,Wym2}. 

We also point out that through a small improvement in a dynamical argument, we have replaced the set 
$$
\mc{N}_H:=\mc{S}_H\cup \mc{M}_H
$$
in~\cite[Theorem 8]{CG17} with $\mc{S}_H$ when considering manifolds without focal points.

\subsection{Outline of the paper}
Sections~\ref{s:controlLooping ncp} and~\ref{s:controlLooping anosov} build technical tools for constructing non-self looping covers. Then, Sections~\ref{s:dynNoConj}, and~\ref{s:Anosov} apply these tools to build non-self looping covers under certain geometric assumptions. In particular, Theorems~\ref{t:noConj2} and~\ref{t:noConj1} are proved in Section~\ref{s:dynNoConj}. In Section~\ref{s:Anosov}, we  prove Theorem~\ref{T:tangentSpace} and the remaining cases in Theorem~\ref{T:applications}. {The reader will find below that there are \emph{many} parameters explicitly named in the propositions. We understand that keeping track of these may be painful (and encourage the reader to treat them as some positive constant in most cases). However,  it is important to keep of track of the dependence of our estimates on many of these constants e.g. in the proof of Theorem~\ref{t:noConj2}.}
\subsection{Index of Notation}\ \smallskip
In general we denote points in $\TM$ by $\rho$, and vectors in $T_\rho(\TM)$ in boldface (e.g. $\mathbf{v} \in T_\rho(\TM)$).  Sets of indices are denoted in calligraphic font (e.g $\mathcal I$). When position and momentum need to be distinguished we write $\rho=(x,\xi)$ for $x\in M$ and $\xi \in T_x^*M$. Next, we list symbols that are used repeatedly in the text along with the location where they  are first defined.
\begin{multicols}{3}
 $\varphi_t $            \tabto{1.6cm}    \eqref{e:varphi}\\
 $\mc{H}_{\Sigma}$       \tabto{1.6cm}    \eqref{e:Hsig}\\
 $\Tinj$                 \tabto{1.6cm}    \eqref{e:Tinj}\\
 $\Lambda_{_{\!A}}^\tau(r)$ \tabto{1.6cm}  \eqref{e:tube}\\
 ${K}_{_{\!H}}$            \tabto{1.6cm}  \eqref{e:KH}\\
 $\ca$                    \tabto{1.6cm}    \eqref{e:Bdef}\\
 $\sob{m}$                   \tabto{1.6cm}    \eqref{e:sobolev}
 \vfill\null
\columnbreak
\noindent  $\Lambda_{\max}$          \tabto{1.6cm}    \eqref{e:Lmax}\\
 $T_e(h)$               \tabto{1.6cm}    \eqref{e:Tehr}\\
$N_{\pm}(\rho)$          \tabto{1.6cm}    \eqref{e:N pm}\\
$\mc{M}_H$                \tabto{1.6cm}   \eqref{e:MH}\\
$\mc{S}_H$               \tabto{1.6cm}    \eqref{e:SH}\\
$\mc{A}_H$                \tabto{1.6cm}   \eqref{e:AH}
\vfill\null
\columnbreak
\noindent $F$, $\delta_F$     \tabto{1.6cm}   \eqref{e:defFunction}  \\
  $\psi$          \tabto{1.6cm}   \eqref{e:psi}  \\
 $J_t$                    \tabto{1.6cm}     \eqref{e:Jdef}\\
$\mathbf{D}$              \tabto{1.6cm}   \eqref{e:D} \\
$C_{\varphi}$           \tabto{1.6cm}      \eqref{e:cphi}\\
$\Theta_{\pm}$           \tabto{1.6cm}     \eqref{e:Theta}
\end{multicols}








\medskip
\noindent {\sc Acknowledgements.} Thanks to Pat Eberlein, John Toth, Andras Vasy, and Maciej Zworski for many helpful conversations and comments on the manuscript. {Thanks also to the anonymous referees for their careful reading and many comments which improved the exposition.}
J.G. is grateful to the National Science Foundation for support under the Mathematical Sciences Postdoctoral Research Fellowship  DMS-1502661. {Y.C. is grateful to the Alfred P. Sloan Foundation. }


\addcontentsline{toc}{section}{\quad\;\,\bf{Dynamical Analysis}}
\section{Partial invertibility of $d\varphi_t|_{T\SigH}$ and looping sets}
\label{s:controlLooping ncp}
\renewcommand{\SigH}{\Sigma_{_{\!H,p}}}
\renewcommand{\LambdaH}{\Lambda^\tau_{\!\Sigma_{\!H,p}}}

The aim of this section is to study the set of geodesic loops in $\SNH$ under conditions on the structure of the set of conjugate points of $(M,g)$. However, we work in the general setting in which the Hamiltonian flow is not necessarily the geodesic one. {We do this since we plan to use the results for general Hamiltonian flows in future work.} In particular, let $p\in S^m$ be real valued with 
$$
|p|\geq |\xi|^m/C,\qquad |\xi|\geq C
$$ 
and define $\varphi_t:=\exp(tH_p)$ and $\Sigma_{_{\!H,p}}:=\{p=0\}\cap N^*\!H$ so that in the case $p=|\xi|_g-1$, $\Sigma_{_{\!H,p}}=\SNH$. 
{We assume that $H$ is \emph{conormally transverse} for $p$ in the sense that for any defining functions $f_1,\dots f_k$ for $H$, i.e. $f_i\in C^\infty(M;\mathbb{R})$ with $H=\{x\in M\mid f_i(x)=0,\,i=1,\dots,k\}$ and $df_i|_{H}$ are linearly independent, we have 
\begin{equation}\label{e:conormallyTransverse}
 N^*H\subset  \{p\neq0 \}\cup \bigcup_{i=1}^k \{H_pf_i\neq 0\}.
\end{equation}
Note that with this definition the $\Tinj$ as in~\eqref{e:Tinj} continues to make sense for general $p$ and conormally transvers $H$.}
For such $H$, we define $r_H:\TM\to \re$ by $r_H(\rho)=d(\pi(\rho),H)$, and let 
$$
\FR:=\inf_{\rho\in \SigH} \lim_{t\to 0^+}|H_pr_H(\varphi_t(\rho))|
$$
We now fix once and for all a defining function $F:\TM \to \re^{n+1}$ for $\SigH$ and $\delta_F>0$ so that:\\ \ \\
\qquad For $q \in \TM $ with $d(q,\SigH)<\delta_F$, 
 \begin{align}
 \label{e:defFunction}
 &\bullet\; \SigH=F^{-1}(0) \notag\\
 &\bullet\; \tfrac{1}{2}d(q,\SigH)\leq |F(q)|\leq 2d(q,\SigH), \notag\\
 & \bullet\; dF(q)\text{ has a {right} inverse }{R}_{_{\!F}}(q)\text{ with }\|{R}_{_{\!F}}(q)\|\leq 2,\\
 &\bullet\; \max_{|\alpha|\leq 2}(|\partial^\alpha F(q)|)\leq 2. \notag
 \end{align}
Define also   $\psi:\re\times \TM \to \re^{n+1}$
\begin{equation}\label{e:psi}
\psi(t,\rho)=F\circ \varphi_t(\rho).
\end{equation} 

Working under the assumption that the set of conjugate points can be controlled {and that the dimension of $\dim H<\frac{n-1}{2}$} will allow us to say that if $ \varphi_{t_0}(\rho_0)$ is exponentially close to $\SigH=\SNH$ for some time $t_0$ and some $\rho_0 \in \SNH$, then there exists a tangent vector $ \w \in T_{\rho_0}\SNH$ for which the restriction 
\begin{equation}\label{e:li}
d\psi_{(t_0,\rho_0)}: \R \partial_t \times \re \w\to  T_{\psi(t_0,\rho_0)}\re^{n+1} 
\end{equation}
 has a left inverse whose norm we control. This is proved in Lemma \ref{l:prelimNoConj} and is the cornerstone in the proof of  Theorems \ref{t:noConj1} and~\ref{t:noConj2}. Note, however, that asking  \eqref{e:li} to hold is a very general condition that may not need the control of the structure of the set of conjugate points. We will use this in Section~\ref{s:Anosov}.

The goal of this section is to prove Proposition \ref{p:ballCover} below,
whose purpose is to control the number of tubes that emanate from a subset of $\SigH$
and loop back to $\SigH$. This is done under the assumption that the restriction of $d\psi_{(t_0,\rho_0)}$ in \eqref{e:li} has a left inverse. To state this proposition we first need a lemma  that describes a convenient system of coordinates near $\SigH$.  The statement of this lemma is illustrated in Figure \ref{fig:2lemmas}.

Observe that by~\cite[(C.3)]{DyGu14} for any $\Lambda>\Lambda_{\max}$ and $\alpha$ multiindex, there exists $ C_{_{\!M,p,\alpha}}>0$ depending only on $M,p, \alpha$  so that 
\begin{equation}
\label{e:derFlow}
|\partial^\alpha \varphi_t|\leq C_{_{\!M,p,\alpha}} e^{|\alpha| \Lambda t}.
\end{equation}


\begin{lemma}[Coordinates near $\SigH$]
\label{l:tanSpace}
There exists $\tau_1=\tau_1(M,p,\FR)>0$ and $\mathfrak{c}_0=\mathfrak{c}_0(M,p,\FR)$ so that for $\Lambda>\Lambda_{\text{max}}$ the following holds. 
Let $\rho_0\in \SigH$, $t_0\in \re$ be so that 

\begin{itemize}
\item there exists $\w=\w(t_0,\rho_0)\in T_{\rho_0}\SigH$ so that  the restriction
$$
d\psi_{(t_0,\rho_0)}: \R \partial_t \times \re \w\to  T_{\psi(t_0,\rho_0)}\re^{n+1} $$ 
has left inverse $L_{(t_0,\rho_0)}$ with $\|L_{(t_0,\rho_0)}\|\leq A$ for some $A\geq 1$,\medskip
\item  $d(\varphi_{t_0}(\rho_0), \SigH)\leq \min\big\{\frac{ {e^{-2\Lambda|t_0|}}}{16 \mathfrak{c}_0^2 A^2.},\delta_F\big\}$
\end{itemize}
Then, points $\rho$  in a neighborhood of $\rho_0$ can be written in coordinates $\rho=\rho(y_1,\dots, y_{2n})$, with 
$\rho_0=\rho(0,\dots ,0)$ and $\SigH=\{y_{n}=\dots =y_{2n}=0\}$,
 so that 
{$$
\frac{1}{2}d(\rho(y),\rho(y'))\leq |y-y'| \leq 2d(\rho(y),\rho(y')).
$$}
In addition, there exists a smooth real valued function $f$ defined in a neighborhood of $0\in {\re^{2n-1}}$ so that letting  $r_{t_0}:=\frac{8 {e^{-3\Lambda|t_0|}}}{\mathfrak{c}_0^2 A^2}$  and ${0<}r< \tfrac{1}{128}e^{\Lambda |t_0|}r_{t_0}$,
if
\[
|y|<r_{t_0} \qquad \text{and} \qquad  d(\varphi_t(\rho(y)),\SigH)<r\;\;\text{
 for some}\; t\in [t_0-\tau_1,t_0+\tau_1],
 \]
then 
\[
|y_1-f(y_2,\dots y_{2n})|<{2(1+\c)}Ar \qquad \text{and}\qquad 
|\partial_{y_j}f|<\c Ae^{\Lambda|t_0|}.
\]
 \end{lemma}
 

\begin{figure}[H]
    \centering
    \includegraphics[width=8cm]{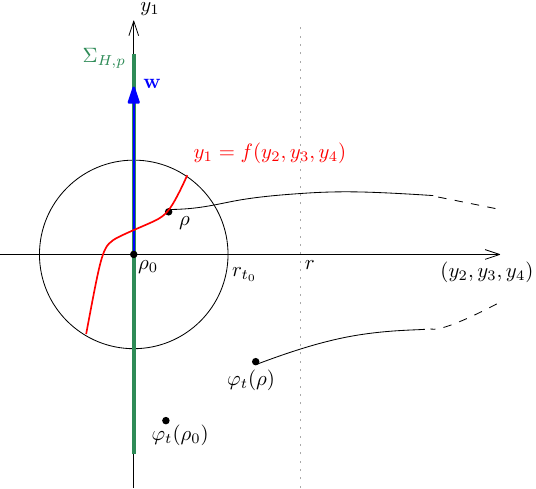}
    \caption{Illustration of the statement in Lemma \ref{l:tanSpace} {when $H$ is a curve and $M$ is a surface}. 
    }
    \label{fig:2lemmas}
\end{figure}

\begin{proof}
 Since $d\psi_{(t_0,\rho_0)}:\re\partial_t\times \re  \w \to \re^{n+1}$ has a left inverse, we {may find an orthogonal matrix $O$ such that $O\circ F=(f_1,\dots,f_{n+1})$ and} with $\tilde F=(f_1, f_2)$, 
 \[
\Psi: \R \times T^*\!M \to \R^2, \qquad \Psi(t,  \rho):=\tilde F \circ \varphi_t(\rho),
 \]
 the restriction $d\Psi:\re\partial_t\times \re \w\to \re^2$ is invertible with inverse $L$ having $\|L\|\leq A$. {Note that since $O$ is orthogonal, $O\circ F$ is a defining function satisfying~\eqref{e:defFunction} with the same 
 $\delta_F$.}
Moreover, since
$$
d\psi_{(t_0,\rho_0)}:\re \partial_t \to T_{\psi(t_0,\rho_0)}\re^{n+1}
$$
has a left inverse, $L_1\in \re$ with $|L_1|< 2\FR^{-1}:=A_0$ we may {choose $O$} so that with {$\Psi(t,\rho)=(\Psi_1(t,\rho),\Psi_2(t,\rho))$,
we have $|\partial_{t}\Psi_1(t_0,\rho_0)|\geq A_0^{-1}$ and 
$
\partial_t\Psi_2(t_0,\rho_0)=0.
$}

 Let ${(t,y)}=(t,{y_1},y_2,\dots, {y_{n-1}, y_{n}},\dots y_{2n})$ be coordinates on $\re \times T^*\!M$ near {$(t_0,\rho_0)$} so that ${(t_0,0)}\mapsto (t_0,\rho_0)$,   $\partial_{y_1}\mapsto \w/\|\w\|$ at ${(t_0,0)}$, and $({y_n},{y_{n+1},\dots, {y_{2n}}})$ define ${\SigH}$. Finally, let $\tilde{y}=e^{\Lambda|t_0|}y.$ We will work with these coordinates on $\re\times T^*\!M$ for the remainder of the proof.

Applying the implicit function theorem (see Lemma~\ref{l:quantImplicit}) with $x_0=t$, $x_1=\tilde{y}$ and $\tilde f:\R \times \R^{2n} \times \R \to \R $ with $\tilde f(x_0,x_1,x_2)=\Psi_1(x_0,x_1)-x_2$ gives that there exists a neighborhood  $U \subset \re^{2n}\times \re$ of {${(0,x_2^0)}$, where $x_2^0:=\Psi_1{(t_0,0)}$}, and a function $x_0=\mathfrak{t}:U \to \re$, so that for $( \tilde{y},x_2)\in U$,
$$
x_2=\Psi_1 \big(\mathfrak{t}(\tilde{y},x_2) , \tilde{y}\big)
$$
with 
\[
|\partial_{x_2} \mathfrak{t}|\leq  A_0,\qquad \qquad  \max_{1\leq j \leq 2n}|\partial_{\tilde y_j}\mathfrak{t}|\leq {\tfrac{c_{_{\!M,p}}}{64 n}}A_0,
\]
where $c_{_{\!M,p}}$ is a positive constant depending only on $(M,p)$. {Here, the $t_0$ independent bounds follow from the chain rule.} {Moreover, we have}
$|\partial_{t,\tilde{y}}^2 \tilde f|\leq {\tfrac{c_{_{\!M,p}}}{64 n}}$, {$|\partial_{t}^2 \tilde f|\leq {\tfrac{c_{_{\!M,p}}}{64 n}}$,} and $|\partial_{\tilde{y}_j} \tilde f|\leq \tfrac{c_{_{\!M,p}}}{64 n}$ for all $j=1, \dots, 2n$.  Then, working with 
\[r_0=\tfrac{8}{ c_{_{\!M,p}}A_0}, \qquad r_1=\min\Big\{\tfrac{32}{c_{_{\!M,p}}^2A_0^2 },\;{\tfrac{8}{c_{_{\!M,p}}A_0}}\Big\}, \qquad r_2=\tfrac{2}{ c_{_{\!M,p}}A_0^2},
\]
\[{B_0={ \tfrac{c_{_{\!M,p}}}{32}},\qquad B_1=  \tfrac{c_{_{\!M,p}}}{64 n}} , \qquad B_2=0, \qquad \tilde B_1=\tfrac{c_{_{\!M,p}}}{64 n} , \qquad \tilde B_2=1,
\]
for $r_0, r_1, r_2$ and $B_0, B_1, B_2, \tilde B_1, \tilde B_2$ as in Lemma \ref{l:quantImplicit}, we obtain that
$U$ can be chosen so that $B(0,r_1)\times {B(x_2^0, r_2)}\subset U$. 
In particular,   it follows that if  
\begin{equation}
\label{e:condBoston}|\mathfrak{t}-t_0|<\tfrac{8}{ c_{_{\!M,p}}A_0},\qquad |\tilde{y}|\leq \min\Big\{\tfrac{32}{c_{_{\!M,p}}^2A_0^2 },{\tfrac{8}{c_{_{\!M,p}}A_0}}\Big\} ,\qquad {|x_2-x_2^0|}<\tfrac{2}{ c_{_{\!M,p}}A_0^2},
\end{equation}
then 
\[
|\mathfrak{t}(\tilde{y},x_2)-\mathfrak{t}(\tilde{y},0) | \leq A_0 {|x_2|}.
\]

 Next, since $d\Psi:\re \partial_t\times \re \w\to \re^2$ is invertible with inverse $L$ satisfying ${\|L\|\leq A}$, we have $|\partial_{\tilde{y}_1}\tilde{f}|^{-1}{\leq} Ae^{\Lambda|t_0|}$ where now we write $\tilde f$ for
 $$
{\tilde{f}(\tilde{y},x_2,x_3)=\Psi_2(\mathfrak{t}(\tilde{y},x_2),\tilde{y})-x_3}.
 $$

Next, we write $\tilde{y}=(\tilde{y}_1,\tilde{y}')$ and  once again apply the implicit function theorem (Lemma~\ref{l:quantImplicit}) with $x_0=\tilde{y}_1$, $x_1=(x_2,\tilde{y}')$, $x_3\in \re$,  to see that there exists $U \subset \re^{2n}\times \re$ {of ${(0,x_3^0)}$, with $x_3^0=\Psi_2({t_0},0)$,} and a function $x_0=\tilde{\mathbf{y}}_1:U \to \re$, so that for $(\tilde{y}',x_3)\in U$,
$$
x_3=\Psi_2 \Big(\mathfrak{t}\big(\tilde{\mathbf{y}}_1(\tilde{y}',x_2,x_3),\tilde{y}',{x_2}\big),\tilde{\mathbf{y}}_1(\tilde{y}',x_2,x_3),\tilde{y}' \Big)
$$
with 
\[
|\partial_{x_3} \tilde{\mathbf{y}}_1|\leq{A e^{\Lambda|t_0|}},\qquad |\partial_{x_2} \tilde{\mathbf{y}}_1|<\c Ae^{\Lambda|t_0|}, \qquad  \max_{2\leq j\leq 2n}|\partial_{\tilde{y}_j} \tilde{\mathbf{y}}_1|\leq \c\,A  e^{\Lambda|t_0|}\]
where $\c$ is a positive constant depending only on $(M,p,A_0)$,  so that 
$|\partial_{(x_2,\tilde{y})}^2 \tilde{f}|\leq \tfrac{\c}{64n}$ and $|\partial_{x_2}\tilde{f}|,\,|\partial_{\tilde{y}_j} \tilde{f}|\leq \tfrac{\c}{64n}$ for all $j=2, \dots, 2n$. {Without loss of generality we assume that $\c \geq  c_{_{\!M,p}}A_0$ {and that $\c>1$}.}
Then, working with 
\[r_0=\tfrac{8 e^{-\Lambda|t_0|}}{ \c A}, \qquad r_1=\min\Big\{\tfrac{32e ^{-2\Lambda|t_0|}}{\c^2A^2 },\tfrac{8 e^{-\Lambda|t_0|}}{\c A}\Big\}, \qquad r_2=\tfrac{2e^{-2\Lambda|t_0|}}{ \c A^2},
\]
\[B_0= {\tfrac{\c}{32}},\qquad   B_1=  \tfrac{\c}{64n} , \qquad B_2=0, \qquad \tilde B_1=\tfrac{\c}{64n} , \qquad \tilde B_2=1,
\]
for $r_0, r_1, r_2$ and $B_0, B_1, B_2, \tilde B_1, \tilde B_2$ as in Lemma \ref{l:quantImplicit}, we obtain that
$U$ can be chosen so that $B({(x_2^0,0)},r_1)\times B({x_3^0}, r_2)\subset U$. 
In particular,   it follows that if 
\begin{equation} 
\label{e:ChapelHill}
\begin{gathered}
|\tilde{\mathbf{y}}_1|<\tfrac{8 e^{-\Lambda|t_0|}}{ \c A},\quad |(\tilde{y}',x_2-{x_2^0})|\leq \min\Big\{\tfrac{32e ^{-2\Lambda|t_0|}}{\c^2A^2 },\tfrac{8 e^{-\Lambda|t_0|}}{\c A}\Big\},\quad
|x_3-x_3^0|<\tfrac{2e^{-\Lambda|t_0|}}{ \c A^2},
\end{gathered}
\end{equation}
 then 
\[
|\tilde{\mathbf{y}}_1(\tilde{y}',x_2,x_3)-\tilde{\mathbf{y}}_1(\tilde{y}',x_2,0) | \leq  A e^{\Lambda|t_0|}{|x_3| }.
\]
Note that this can be done since by assumption {$\c>1$ and }
\begin{equation}\label{e:radiusx3}
|0-x_3^0|=|\Psi_2(t_0, \rho_0)| \leq 2d(\varphi_{t_0}(\rho_0), \SigH)<\tfrac{2e^{-2\Lambda|t_0|}}{ \c A^2}.
\end{equation}
It follows, {after undoing the change $\tilde y=e^{\Lambda|t_0|}y$,}  that if
\begin{align*}
\bullet\;& \max\{|x_2-x_2^0|,|x_3-x_3^0|\}< 
{ \min\Big\{\tfrac{2}{ c_{_{\!M,p}}A_0^2},\; \tfrac{32e ^{-2\Lambda|t_0|}}{\c^2A^2 },\;\tfrac{8 e^{-\Lambda|t_0|}}{\c A},\; \tfrac{2e^{-\Lambda|t_0|}}{ \c A^2}\Big\}},\\
\bullet\;&|y|<
{ \min\Big\{\tfrac{8 e^{-2\Lambda|t_0|}}{ \c A},\;
\tfrac{32e ^{-3\Lambda|t_0|}}{\c^2A^2 },\;
\tfrac{8 e^{-2\Lambda|t_0|}}{\c A},\;
\tfrac{32 e^{-\Lambda|t_0|}}{c_{_{\!M,p}}^2A_0^2 },
\;{\tfrac{8 e^{-\Lambda|t_0|}}{c_{_{\!M,p}}A_0}}
\Big\}},\\
\bullet\;&|t-t_0|
<\tfrac{8}{c_{_{\!M,p}}A_0},
\end{align*}
then
\[
|{\mathbf{y}}_1({y}',x_2,x_3)-{\mathbf{y}}_1({y}',0,0) | \leq (1+\c)A\, {|(x_2,x_3)|}.
\]
{Next, note that since $d(\varphi_t(\rho(y)),\SigH)\leq r$ and $r<\frac{e^{-2\Lambda|t_0|}}{16\c^2 A^2}$, 
then
\[
{|x_2-x_2^0| \leq |x_2| + |x_2^0| \leq 2d(\varphi_t(\rho(y)),\SigH)+2d(\varphi_{t_0}(\rho_0),\SigH)} \leq \tfrac{2e^{-2\Lambda|t_0|}}{ \c A^2},
\]
and similarly, $|x_3-x_3^0|\leq \tfrac{2e^{-2\Lambda|t_0|}}{ \c A^2}$.
In addition, we can assume $c_{_{\!M,p}}>1$. Since  $\c \geq  c_{_{\!M,p}}A_0$, with the above definition of $r_{t_0}$, we obtain that  if $r<\frac{1}{128}e^{\Lambda |t_0|}r_{t_0}$ and $|y|<r_{t_0}$, then
\[
|{\mathbf{y}}_1({y}',x_2,x_3)-{\mathbf{y}}_1({y}',0,0) | \leq 2(1+\c)A r.
\]}

To finish the argument, we note that we may define
$f(y'):={\mathbf{y}}_1({y}',0,0)$ satisfying
$
|\partial_{y'}f|\leq \c Ae^{\Lambda|t_0|}
$
as claimed. Where, as argued in \eqref{e:radiusx3}, this can be done since 
$|0-x_2^0|<\tfrac{2e^{-2\Lambda|t_0|}}{ \c A^2}$  and using that $A \geq 1$, $\c \geq  c_{_{\!M,p}}A_0$.

\end{proof}

\begin{remark}\label{R:c}{ We  proceed to study the number of looping directions and prove the main result of this section.
In what follows $\c$ denotes the constant from Lemma \ref{l:tanSpace}.}
\end{remark}

\begin{proposition} 
\label{p:ballCover} 
 Let $0\leq t_0<T_0$, $0<\tilde c<{\delta_F} $, {$a>0$}, {$\Lambda>\Lambda_{\text{max}}$}, {$c>0$, $\beta \in \re$, $A\subset \SigH$, and $B\subset A$ a ball of radius $R>0$} satisfy the following assumption:  {for all} $(t, \rho) \in [t_0, T_0]\times B$ such that $d(\varphi_t(\rho),{A})\leq \tilde c\, e^{-{a}|t|}$,  there exists $\w \in T_{\rho}\SigH$ for which the restriction
$$ 
d\psi_{(t,\rho)}: \R \partial_t \times \R \w \to  T_{\psi(t,\rho)}\re^{n+1} 
$$ 
has left inverse $L_{(t, \rho)}$ with  $\|L_{(t, \rho)}\|\leq ce^{\beta |t|}$. 
\smallskip

There exist $\alpha_1=\alpha_1(M,p)>0$ and $ \alpha_2=\alpha_2(M,p,c, \tilde c, \delta_F, \FR)$ so that the following holds.\\

Let $r_0,r_1,r_2 >0$ satisfy
\[
r_0 < r_1, \qquad r_1< \alpha_1\, r_2,  \qquad  r_2 \leq \min\{R,{1},  \alpha_2\, e^{-\gamma T_0}\}, \qquad r_0 < {\tfrac{1}{3}}\, e^{-\Lambda T_0} r_2,
\] 
where { $\gamma=\max\{a, 3\Lambda+2\beta\}$}.
Let $0<\tau_0<\frac{\Tinj}{2}$, $0<\tau\leq \tau_0$, and $\{\rho_j \}_{j=1}^{N} \subset {\SigH}$ be a family of points so that 
\vspace{-0.3cm}
\[{\Lambda_{\rho_j}^\tau(r_1)\cap \Lambda_B^\tau(r_0)\neq \emptyset},\qquad \Lambda_B^\tau(r_0) \subset \bigcup_{j=1}^{N} \Lambda_{\rho_j}^\tau( r_1 ),\]
and $\big\{\Lambda_{\rho_j}^\tau(r_1)\big\}_{j=1}^N$ can be divided into ${\mathfrak{D}}$ sets of disjoint tubes.

Then,
there exist a partition of the indices
$\mathcal G \cup \mathcal B= \{1,\dots, N\}$  and a constant  ${{\bf{C}_0}={\bf{C}_0}(M,p,k,c,\beta,\FR)}>0$ so that \smallskip 
\begin{itemize}
    \item $\bigcup_{j\in \mc{G}}\Lambda_{_{\rho_j}}^\tau({r_1})\quad \text{ is }\text{ non-self looping for times in }\;\; [t_0,T_0].$ Moreover, 
    $$d\Big({\Lambda_{A}^\tau(r_0)}\;,\;\bigcup_{t\in[t_0,T_0]}\bigcup_{j\in \mc{G}}\varphi_t(\Lambda_{_{\rho_j}}^\tau(r_1))\Big)>2r_1.$$
    \medskip
    \item $|\mathcal B|\leq {{\bf{C}_0}}{\mathfrak{D}} \;r_2 \;\frac{R^{n-1}}{r_1^{n-1}}\;  {T_0}\,e^{4(\Lambda+\beta)T_0}.$
\end{itemize}

\end{proposition}

\begin{remark}\label{r:D depends on n}
{Note that we will typically apply Proposition~\ref{p:ballCover} with $\{\Lambda_{\rho_j}^\tau(r_1)\}_j$ a subset of a $(\mathfrak{D}_n,\tau,r)$ good cover for $\SigH$. In this case the constant $\mathfrak{D}$ can be absorbed into ${\bf{C}_0}$ since it depends only on $n$.}
\end{remark}

\begin{proof}
 Let {$\tau_1=\tau_1(M,p,\FR)$ be the minimum of $1$ and the constant from Lemma~\ref{l:tanSpace}}, and let
$L$ be the largest integer with $L\leq \frac{1}{{\tau_1}}(T_0-t_0)+1$. Cover $[t_0,T_0]$ by
$$
[t_0,T_0]\subset \bigcup_{\ell=0}^{L} \big[s_\ell - \tfrac{{\tau_1}}{2} ,s_\ell + \tfrac{{\tau_1}}{2}\big],
$$
{where $s_\ell:= t_0 +(\ell +\frac{1}{2}){\tau_1}$.}
We claim that for each $\ell=0, \dots, L$ there exists a partition of indices $ \mathcal G_\ell \cup \mathcal B_\ell = \{1, \dots, N\}$  so that 
\begin{equation}\label{e:mainclaim0}
|\mathcal B_\ell| \leq {{\bf{C}_0}}{\mathfrak{D}}\frac{r_2 R^{n-1}}{r_1^{n-1}} e^{4(\Lambda+\beta)|s_\ell |}
\end{equation}
 and
{
\begin{equation}\label{e:mainclaim}
d\left(\Lambda_{{A}}^\tau(r_0)\;, \bigcup_{s=s_\ell-\frac{{\tau_1}}{2}}^{s_\ell+\frac{{\tau_1}}{2}}\varphi_t\big(\Lambda_{_{\!\rho_k}}^\tau(r_1)\big)\right)\geq {\frac{1}{C_{_{\!S}}}}r_2 -C_{_{\!S}}r_0 \;\;\; \;\;\;\forall k \in \mathcal G_\ell.
\end{equation}
}
Here,  
\begin{gather*}
C_{_{\!S}}:=\sup\big\{ \|d\varphi_t(q)\|:\; q\in \Lambda^1_{\{p=0\}}(\ep_0),\,|t|\leq {\tfrac{4}{3}}\big\},
\end{gather*}
 where $\ep_0>R$ is a constant independent of $r_0, r_1, r_2, R$. 
The result then follows from setting 
\[
\mathcal B:= \bigcup_{\ell=0}^L \mathcal B_\ell\qquad \text{and} \qquad \mathcal G:=\{1, \dots, N\}\backslash \mathcal B,
\]
together with  asking for  $\alpha_1<{\tfrac{1}{2C_{_{\!S}}+C^2_{_{\!S}}}}$ so that {${\frac{1}{C_{_{\!S}}}}r_2 -C_{_{\!S}}r_0 >2r_1$}. {Note that the adjustment depends only on $(M,p)$. }

We  have reduced the proof of the lemma to establishing the claims in \eqref{e:mainclaim0} and \eqref{e:mainclaim}. We next explain that it suffices to prove  \eqref{e:mainclaim}  with $\Lambda_{{A}}^\tau(r_0)$ replaced by ${A}$. To see this,   let $\{t_j\}$ be so that 
\[
[-(3\tau+{\tau_1}{+r_0}),3\tau+{\tau_1}+{r_0}]=\bigcup^{J}_{j=1} [t_j-\tfrac{{\tau_1}}{2},t_j+\tfrac{{\tau_1}}{2}],
\]
where $J$ is the largest integer  with  $J \leq (6\tau+{2r_0})/{\tau_1}+2.$ {Note that since $\tau<\tau_0<1$, ${r_0<\frac{1}{3}}$ and ${\tau_1}$ depends only on $(M,p,\FR)$, the same is true for $J$.}   
Fix $\ell \in \{1, \dots, L\}$. We claim that for each $j\in \{1, \dots, J\}$ there exists a partition $\mathfrak{g}_j^\ell\cup \mathfrak{b}_j^\ell=\{1,\dots,N\}$
with 
\begin{equation}
\label{e:claim-grass0}
|\mathfrak b_j^\ell |\leq {{\bf{C}_0}}{\mathfrak{D}}\frac{r_2 R^{n-1}}{r_1^{n-1}} e^{4(\Lambda+\beta)|s_\ell |},
\end{equation}
and 
\begin{equation}
\label{e:claim-grass}
d\Big({A}, \bigcup_{t=s_\ell+t_j-\frac{{\tau_1}}{2}}^{s_\ell+t_j+\frac{{\tau_1}}{2}}\varphi_t\big(\rho\big)\Big)\geq  r_2 \qquad \text{for all}\;\; \rho\in \bigcup_{k\in \mathfrak g_j^\ell}\Lambda_{\rho_k}^\tau(r_1).
\end{equation}
Suppose the claims in \eqref{e:claim-grass0} and \eqref{e:claim-grass} hold and let 
\[
\mathcal B_\ell:=\bigcup_{j=1}^J \mathfrak b_j^\ell\qquad \text{ and }\qquad \mathcal G_\ell=\{1, \dots, N\}\backslash \mathcal B_\ell.
\]
 Then, by construction, {after possibly adjusting $ {\bf{C}_0}$ to take into account the bound on $J$ (which only depends on $(M,p,\FR)$),} we obtain that \eqref{e:mainclaim0} also holds. To derive \eqref{e:mainclaim} 
suppose $\rho\in \Lambda_{\rho_k}^\tau(r_1)$ for some $k \in \mathcal G_\ell$. In particular,  since $k \in \mathfrak g_j^\ell$ for all $j=1, \dots, J$, relations \eqref{e:claim-grass} yield that
$$
d\Big({A}, \bigcup_{t=s_\ell-3\tau-{\tau_1}-{r_0}}^{s_\ell+3\tau+{\tau_1}+{r_0}}\varphi_t(\rho)\Big)\geq  r_2.
$$
In particular, {using the definition of $ {C_{_{\!S}}}$, { that $\tau<\Tinj\leq 1$}, {and $r_0<\frac{1}{3}$}}
{
$$
d\Big(\Lambda^{\tau+{r_0}}_{{A}}, \bigcup_{t=s_\ell-2\tau-{{\tau_1}}}^{s_\ell+2\tau+{{\tau_1}}}\varphi_t(\rho)\Big)\geq  {\frac{r_2}{C_{_{\!S}}}},
$$
and this proves \eqref{e:mainclaim} after using the definition of $C_S$ once again.}

We have then reduced the proof of the proposition to establishing the claims in \eqref{e:claim-grass0} and \eqref{e:claim-grass}. Fix $\ell \in \{1, \dots, L\}$, $j \in \{1, \dots, J\}$, and set
\[
s:=s_\ell +t_j.
\]
To prove these claims we start by covering $B$ by  balls $B_\alpha^s \subset \TM  $ of radius ${\bf R_s}>0$ (to be determined later) and centers in $B$,
\[
B \subset \bigcup_{\alpha=1}^{I_s} B_\alpha^s,
\]
so that $I_s \leq C_{n} R^{n-1} {\bf R}^{-(n-1)}_s $ for some $C_{n}>0$.
 Fix $B_\alpha^s$ and suppose there exists $\rho_0\in B_\alpha^s$ such that
\begin{equation}
\label{e:badness}
d(\SigH,\rho_0)<r_0\qquad \text{and}\qquad  d\Big({A}, \bigcup_{t=s-\frac{{\tau_1}}{2}}^{s+\frac{{\tau_1}}{2}}\varphi_t(\rho_0)\Big)<  r_2.
\end{equation}
Then there exists $\tilde{s}\in [s-\frac{{\tau_1}}{2},s+\frac{{\tau_1}}{2}]$ with $d(\varphi_{\tilde{s}}(\rho_0),A)<r_2$.  
{Next, since $d(\rho_0,\SigH)<r_0$, there exists $\rho_\alpha \in \SigH$ with 
\[
\varphi_{\tilde{s}}(\rho_\alpha)\in  B(\varphi_{\tilde{s}}(\rho_0), c_{_{\!M,p}}e^{\Lambda|\tilde{s}|}r_0),\qquad d(\rho_0,\rho_\alpha)<r_0,
\]}
{for some $c_{_{\!M,p}}>0$.}
In addition, {letting ${\bf{\bar r}}_{{s}}=c_{_{\!M,p}}e^{\Lambda|\tilde{s}|}r_0$,}
\begin{equation*}
\begin{gathered}
d(\SigH,\varphi_{\tilde{s}}(\rho_\alpha))\leq d(A,\varphi_{\tilde{s}}(\rho_\alpha)) \leq d(A, \varphi_{\tilde{s}}(\rho_0)) + d( \varphi_{\tilde{s}}(\rho_0),\varphi_{\tilde{s}}(\rho))< r_2+ {\bf \bar r_s}.
\end{gathered}
\end{equation*}
We then assume that $\alpha_2<{\frac{3}{3+c_{_{\!M,p}}}\min\{\tfrac{\tilde c}{2},\,{\tfrac{\delta_F}{2}},\,\tfrac{1}{{32}\c ^2c^2}\}}$  so that  
$$ r_2+ {\bf \bar r_s}<{\min\Bigg\{\tilde c e^{-a|\tilde{s}|},\, \frac{e^{-2(\Lambda+\beta)|\tilde{s}|}}{16\c ^2c^2},\,\delta_F\Bigg\}}$$
{where $\c$ is from Lemma~\ref{l:tanSpace}.}
Then, by assumption there exists $\w=\w(\tilde s, \rho_\alpha)\in T_{\rho_\alpha}\SigH$ so that 
the restriction $d\psi_{(\tilde{s},\rho_\alpha)}: \R \partial_t \times \re \w \to  T_{\psi(\tilde{s} ,\rho_\alpha)}\re^{n+1} $ 
has left inverse $L_{(\tilde{s},\rho_\alpha)}$ with $\|L_{(\tilde{s},\rho_\alpha)}\|\leq ce^{\beta|\tilde s|}$. By Lemma \ref{l:tanSpace} the {points $\rho$  in a neighborhood of $\rho_\alpha$ can be written in coordinates $\rho=\rho(y_1,\dots, y_{2n})$ with 
$\rho_\alpha=\rho(0, \dots, 0)$ and $\SigH=\{y_{n}=\dots =y_{2n}=0\}$}
 so that 
$
\frac{1}{2}d(\rho(y),\rho(y'))<|y-y'|<2d(\rho(y),\rho(y')).
$
Let 
\[
r_{\tilde s}:=\frac{ {8}e^{-(3\Lambda+2\beta) |\tilde s|}}{c^2\c^2 }.
\]
These coordinates are built with the property that there exists a smooth real valued function $f$ defined in a neighborhood of $0 \in \R^{{2n-1}}$ so that if {$0<r<\frac{1}{128}{e^{\Lambda |\tilde s |}r_{\tilde s}}$}, 
\[ |y|<r_{\tilde s} \qquad \text{and} \qquad  d(\varphi_t(\rho(y)),\SigH)<r\;\;\text{
 for some}\; t\in \big[\tilde s-{\tau_1},\tilde s+{\tau_1}\big],\]
then 
\[|y_1-f(y_2,\dots y_{2n})|<{2(1+\c)ce^{\beta |\tilde s|} r} \qquad \text{and}\qquad 
|\partial_{y_j}f|<{\c\,}ce^{\beta |\tilde s|}e^{\Lambda|\tilde s|}
\]  
Assume {$\alpha_2<\tfrac{1}{128}$} so that  $r_2<{\frac{1}{128}e^{\Lambda |\tilde{s} |}r_{\tilde{s}}}$. Since  $\tilde{s}\in[s-\frac{{\tau_1}}{2},s+\frac{{\tau_1}}{2}]$, we may choose  $r:= r_2$ to get that, if $\rho=\rho(y)\in B(\SigH,r_0) $  satisfies $d(\rho,\rho_\alpha)<\frac{r_{\tilde s}}{2}$ and
\vspace{-0.2cm}
\begin{equation}\label{e:nonloop}
d\Big(\SigH, \bigcup_{t=s-\frac{{\tau_1}}{2}}^{s+\frac{{\tau_1}}{2}}\varphi_t(\rho)\Big)<  r_2, 
 \end{equation}
 \vspace{-0.2cm}
 then {with $\bar{y}=(y_n,\dots y_{2n})$}
\begin{align*}
|y_1-f(y_2,\dots, y_{n-1},0)|
&\leq |y_1-f(y_2,\dots, y_{n-1},{\bar{y}})|+ |\partial_{y_j}f(y_2,\dots, y_{n-1},0)| |{\bar{y}}| \\
&<{2(1+\c)}ce^{\beta |\tilde s|}  r_2+ {\c}c e^{\beta |\tilde s|}e^{\Lambda|\tilde s|}2r_0\\
&<C_0e^{\beta |\tilde s|}r_2.
\end{align*}
Here, we have used that the assumption $r_0 < \tfrac{1}{3}\, e^{-\Lambda T_0} r_2$ implies  {$e^{\Lambda|\tilde s|}2r_0< r_2$}, and we have written $C_0={(2+3\c) c}$. Also, we used that $|\bar y| \leq 2 d(\rho(y), \rho(y_2,\dots, y_{n-1},0))=2 d(\rho(y),\SigH){\leq 2 r_0}$.

\begin{figure}
    \centering
    \includegraphics[width=10cm]{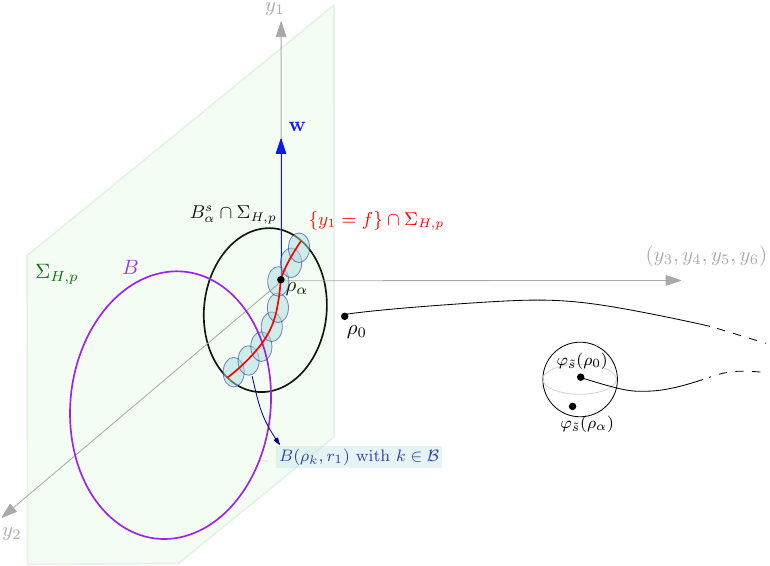}
    \caption{Illustration, when $n=3$, of  the covering balls that intersect $B_\alpha^s$ and loop back for times $\tilde s$ near $s$. }
    \label{fig:1lemma}
\end{figure}

Next, we let ${\bf R_s}=\frac{r_{\tilde s}}{8}$ {and use that $\alpha_2<\frac{1}{16 c^2\c^2}$} to obtain that since {$\rho_0\in B_\alpha^s$, for $\rho \in B_\alpha^s$,}
\begin{equation}\label{E:Rs}
d(\rho,\rho_\alpha )\leq d(\rho_0,\rho_\alpha)+d(\rho,\rho_0)<{r_0}+2{\bf R_s}<{\frac{r_{\tilde s}}{2}}.
\end{equation}
In particular,  \eqref{E:Rs} implies 
$$B_\alpha^s \subset \{\rho \in \TM :\; d(\rho,\rho_\alpha)<{\frac{r_{\tilde s}}{2}}\}.$$ 
Therefore, we have showed that if $\rho \in B_\alpha^s \cap B(\SigH,r_0)$ satisfies \eqref{e:nonloop}, then $\rho \in \mathcal U_{\rho_\alpha}^s \cap B(\SigH,r_0)$ where
$$
\mathcal U_{\rho_\alpha}^s=\left\{\rho:\; |y_1-f(y_2,\dots, y_{n-1},0)|<C_0e^{\beta |\tilde s|}r_2,\quad  d(\rho,\rho_\alpha)<\tfrac{r_{\tilde s}}{2}\right\}.
$$
This is illustrated in Figure \ref{fig:1lemma}.
Next, note that, the number of disjoint tubes in $\{ \Lambda_{\rho_j}^\tau( r_1 )\}_{j=1}^{N}$ that intersect $\mathcal U_{\rho_\alpha}^s \cap B(\SigH,r_0)$ is controlled by the number of disjoint balls in the collection $\{B(\rho_j, r_1)\}_{j=1}^N$ that intersect ${\mathcal U}_{\rho_\alpha}^s \cap \SigH$.  In addition, for each $j\in \{1, \dots, N\}$ the intersection $B(\rho_j, r_1)\cap \SigH$ is entirely contained in $\tilde{\mc{U}}_{\rho_\alpha}^s\cap \SigH$ where 
$$
\tilde{\mc{U}}_{\rho_\alpha}^s\!\!=\!\!\left\{\!\rho: |y_1-f(y_2,\dots, y_{n-1},0)|<C_0e^{\beta |\tilde s|}r_2\!+\!4r_1,\quad\, d(\rho,\rho_\alpha)<\frac{r_{\tilde s}}{2}\!+\!4r_1\right\}\!.
$$
 In particular, 
\begin{align*}
\vol  (\tilde{\mathcal U}_{\rho_\alpha}^s \cap \SigH)
 &\leq(C_0e^{\beta |\tilde s|}r_2+4r_1)\int_{B(0,\frac{r_{\tilde s}}{2}+4r_1)}\sqrt{1+|\nabla f|^2}\,dy_2\dots dy_{n-1}.
\end{align*}
Hence,  the number of disjoint balls in the collection $\{B(\rho_j, r_1)\}_{j=1}^N$ that intersect ${\mathcal U}_{\rho_\alpha}^s \cap \SigH$ is controlled by
 $$
 {2\sqrt{n-1}\, \c c}(C_0e^{\beta(|s|+{\tau_1})}r_2+4r_1)\, {e^{(\beta+\Lambda)(|s|+{\tau_1})}}\big(\frac{r_{\tilde s}}{2}+4r_1\big)^{n-2}r_1^{-(n-1)}.
 $$ 
 Here, we used the bound $|\partial_{y_j}f|<{\c} \,ce^{(\beta+\Lambda) |\tilde s|}$ and that $e^{\beta|\tilde s|} \leq e^{\beta(|s|+{\tau_1})}$.

Finally,  note that since $\alpha_2< \tfrac{1}{c^2 \c^2}$ and $\gamma \geq 3\Lambda +2\beta$, by choosing $\alpha_1<1$, we have $r_1<\min \{r_2,r_{\tilde s}\}$. 
Hence,  the number of disjoint balls in the collection $\{B(\rho_j, r_1)\}_{j=1}^N$ that intersect ${\mathcal U}_{\rho_\alpha}^s \cap \SigH$ is controlled by
 $
{e^{2\beta {\tau_1}}} e^{(2\beta+\Lambda)|s|}r_2 \tilde r_s^{n-2}r_1^{-(n-1)}
 $ 
 up to a constant that depends only on $(M,p,k,c, {\FR  })$.
In addition, note that in the collection $\{ \Lambda_{\rho_j}^\tau( r_1 )\}_{j=1}^{N}$ there are ${\mathfrak{D}}$ sets of disjoint tubes of radius $r_1$. Therefore, since there are $ I_s \leq C_{n} R^{n-1} {\bf R_s}^{-(n-1)} $ balls $B_\alpha^s$, for $s=s_\ell+t_j$ we can build $\mathfrak b^\ell_j$ so that
\[
\rho \notin \bigcup_{k \in \mathfrak b^\ell_j} \Lambda_{\rho_k}^\tau(r_1)\quad \Longrightarrow  \quad  d\Big(A, \bigcup_{t=s_\ell+t_j-\frac{{\tau_1}}{2}}^{s_\ell+t_j+\frac{{\tau_1}}{2}}\varphi_t(\rho)\Big)\geq r_2,
\]
and so that for some ${{\bf{C}_0}={\bf{C}_0}(M,p,k,c,\beta,\FR  )}>0$ 
$$
|\mathfrak b^\ell_j|
\leq {{\bf{C}_0}}{\mathfrak{D}}\frac{e^{(2\beta+\Lambda)|s|}r_2r_{\tilde s}^{n-2}R^{n-1}}{r_1^{n-1}\mathbf{R}_s^{n-1}}.
$$ 
{Here, we have used that $e^{2\beta {\tau_1}} \leq e^{2\beta}$ since ${\tau_1}\leq 1$.}
Using that $\frac{r_{\tilde s}^{n-2}}{\mathbf{R}_s^{n-1}}=\frac{8^{n-1}}{r_{\tilde s}}$ and adjusting ${{\bf{C}_0}}$, we obtain  \eqref{e:claim-grass0}.  This concludes the proofs of the  claims in \eqref{e:claim-grass0} and \eqref{e:claim-grass}.

\end{proof}

\section{Contraction of $\varphi_t$ and non-self looping sets}
\label{s:controlLooping anosov}
The proofs of  Theorems \ref{T:applications} and \ref{T:tangentSpace}  hinge on controlling how the geodesic flow changes the volume of sets contained in $\SNH$. {As in the previous section, we work with a general Hamiltonian $p$ such that $H$ is conormally transverse for $p$.}
Let
\begin{equation}\label{e:Jdef}
J_t:=d\varphi_t|_{T_\rho\SigH}:T_\rho \SigH\to d\varphi_t(T_\rho \SigH).
\end{equation}

{When the Hamiltonian flow is assumed to be Anosov, we have that for $A_0\subset \mc{S}_H\setminus \mc{M}_H$, we can split $A_0$ into pieces $A_{\pm,0}$ such that there is $C_0\geq 1$ satisfying
\begin{equation}
\label{e:forward}
\sup_{\rho \in {A_{\pm,0}} }|\det J_t|\leq C_0e^{-|t|/C_{_{0}}},\qquad \pm t \geq 0.
\end{equation}
The analysis in this section will be used in Section~\ref{s:Anosov} to prove Theorem~\ref{T:tangentSpace} and in particular, to handle $\mc{S}_H\setminus \mc{M}_H$. This, for instance, is the step which allows us to show that averages over subsets of horospheres have improvements.}

 { Note, however, that the condition in \eqref{e:forward} is very general and that it may hold in situations where the Hamiltonian flow is not Anosov. For example, such an estimate holds for the geodesic flow at the umbillic points of the triaxial ellipsoid (see e.g.~\cite{GT18a}). }This section is dedicated to study the structure of the set of looping tubes under the assumption that \eqref{e:forward} holds.

By~\eqref{e:derFlow}, there exists $C_{_{\!\varphi}}>0$ depending only on $(M,p)$, so that  {for all $\Lambda>\Lambda_{\text{max}}$}
\begin{equation}
\label{e:cphi}
 \|d\varphi_t\| \leq C_{_{\!\varphi}}e^{\Lambda |t|},\qquad t\in \R.
\end{equation}
Let $\Decay>1$ be so that 
{
\begin{equation}\label{e:D}
e^{-\Lambda \Decay}< \min\Big\{\frac{e^{-\Lambda{(1+\Tinj)}}}{C_{_{\!\varphi}}}, \frac{\alpha_1}{4}, \frac{1}{4} 
\Big\},
\end{equation}
{where $\alpha_1=\alpha_1(M,p)$} is the constant introduced in Proposition \ref{p:ballCover}. 
}
\begin{definition}\label{d:control}
Let $ A_0\subset\SigH$, $\e_0>0$, $\digamma>0$, {$\ti_0:[\e_0, \infty) \to [1, \infty)$}, and $T_0>1$ . If the following conditions are satisfied, we say that 
\[
A_0 \; \text{ can be}\;  {(\e_0, \ti_0, \digamma) \text{-controlled up to time}\; T_0.}
\]
Let $\ep\geq\e_0$, {$\Lambda>\Lambda_{\text{max}}$},
  \begin{equation*}\label{e:parameters}
0<R_0\leq \tfrac{1}{\digamma}e^{-\digamma\Lambda|T_0|},\qquad0<r_0<R_0,
\end{equation*} 
and balls $\{B_{0,i}\}_{i=1}^N\subset \SigH$ centered in $A_0$ with radii $\{R_{0,i}\}_{i=1}^N \subset [r_0, R_0]$. Then,  for {$0<\tau<\tfrac{1}{2}\Tinj$} and  all
$$
A_1\subset  \bigcup_{i=1}^NB_{0,i}\subset A_0
\qquad
\text{and} 
\qquad
0<r<\tfrac{1}{\digamma}e^{-\digamma\Lambda T_0}r_0,
$$ there are balls $\{\tilde B_{1,k}\}_{k}\subset \SigH$ with radii $\{R_{1,k}\}_k \subset [0, {\tfrac{1}{4}R_0}]$  so that  
\begin{enumerate}
\item $\Lambda^\tau_{A_1 \backslash \cup_k \tilde B_{1,k}}(r)$ is  non self-looping for times in $[\ti_0(\e), T_0]$, \\
\item $\sum_k R_{1,k}^{n-1}\leq \e \sum_i R_{0,i}^{n-1}$,\\
\item  ${\inf_k}R_{1,k} \geq  e^{-\Decay\Lambda T_0}\, {\inf_iR_{0,i}}$.
\end{enumerate}
{We observe that when we write $A_1 \backslash \cup_k \tilde B_{1,k}$ we mean $A_1 \cap (\SigH \backslash \cup_k \tilde B_{1,k})$.}
\end{definition}
{Note that Definition~\ref{d:control} is vacuous if $T_0\leq \ti_0(\e_0)$.}
\begin{lemma}\label{l:nonlooping}
{There exists $\digamma>0$ depending only on $(M,p, {K_{_{\!H}}})$ so that for every monotone decreasing function $\f:[0,\infty) \to [0,\infty)$ with $\f\in L^1([0,\infty))$ {and $\Lambda>\Lambda_{\max}$}, there exists a function $\ti_0:(0,\infty)\to[1,\infty)$ with the following properties.\\
If $A_0\subset\SigH$ is so that
\begin{equation}
\label{e:forward1}
\sup_{\rho \in {A_0} }|\det J_t|\leq \f(|t|)
\end{equation}
 for all $t \in (0, T_0)$ or for all $t \in (-T_0, 0)$, then, for all $\e_0>0$,
 \[
A_0 \; \text{can be}\; (\e_0, \ti_0, \digamma)\text{-controlled up to time}\; T_0
\]
 in the sense of Definition~\ref{d:control}. Furthermore,  in addition to conditions (1), (2) and (3) in Definition~\ref{d:control} being satisfied, either
\[\begin{gathered}
\bigcup_{t={\ti_0(\e)}}^{T_0} \varphi_t(\Lambda^\tau_{A_1\backslash\cup_k\tilde{B}_{1,k}}(r))\cap \Lambda_{\SigH \backslash\cup_k\tilde{B}_{1,k}}^\tau(r)=\emptyset,\\
\text{or}
\\
{\bigcup_{t=-T_0}^{-{\ti_0(\e)}} \varphi_t(\Lambda^\tau_{A_1\backslash\cup_k\tilde{B}_{1,k}}(r))\cap \Lambda_{\SigH \backslash\cup_k\tilde{B}_{1,k}}^\tau(r)=\emptyset.}
\end{gathered}
\]}
\end{lemma}
{Note that the last conclusion of Lemma~\ref{l:nonlooping} differs from condition (1) in Definition~\ref{d:control} since we insist that, after flowing, not only does $\Lambda^\tau_{A_1\backslash\cup_k\tilde{B}_{1,k}}(r)$ not self-intersect (as in (1) of Definition~\ref{d:control}, but it does not even intersect $\SigH\setminus \cup_k\tilde{B}_{1,k}.$}
\begin{proof}
We prove the case in which \eqref{e:forward1} holds for all $t \in (0, T_0)$ (the case in which it holds for all $t \in (-T_0, 0)$ is identical after sending $t\to -t$).  {Let $\Lambda>\Lambda_{\max}$ and
 $t_0$ be large enough so that  {$t_0>\Tinj+2$} and 
\begin{equation}\label{e:Tplus}
{C_{_{\!\varphi}}e^{\Lambda}e^{-\Decay\Lambda (t_0-\Tinj-1)}} \leq 1,
 \end{equation}
  where {$C_{_{\!\varphi}}$ is as in \eqref{e:D}}. {We will assume, without loss of generality, that $f(|t|)\geq \frac{1}{C_\varphi}e^{-\Lambda t}$.}
Define 
\[
\ti_0:(0,\infty) \to [1, \infty) \qquad  \ti_0(\e)=\inf\Bigg\{s\geq t_0:\,\, \int_s^\infty \!\! f(s)ds\leq \frac{\e\Tinj}{4\alpha}\Bigg\},
\]
where 
\[\alpha:={2^{3n-1}}\gamma^{n-1}\qquad \text{and} \qquad \gamma:=\tfrac{1}{4}C_{\varphi}e^{\Lambda}.\]  }
{Here, $\ti_0(\e)\geq 2$ since $t_0>\Tinj+2>2$.}

{Fix $\e_0>0$ and let $\e\geq \e_0$.} Let $0<\tau<\frac{1}{2}\Tinj$, $R_0>0$, $0<r_0<R_0$  and let $\{B_{0,i}\}_{i=1}^N\subset \SigH$ be a collection of balls centered in $A_0$ with radii $\{R_{0,i}\}_{i=1}^N \subset [r_0, R_0]$. Let $A_1\subset  \bigcup_{i=1}^NB_{0,i}$ and $0<r<1$.
   For each $i \in \{1, \dots, N\}$ let $\{I_{0,i,j}\}_{j=1}^{N_i}$ be a collection of {disjoint} intervals  $I_{0,i,j} \subset [{\ti_0(\e)}-2\tau-r, T_0+2\tau+r]$ {so that $\tfrac{\Tinj}{4}\leq |I_{0,i,j}|<\tfrac{\Tinj}{2}$ and}  
\begin{equation}\label{e:intervals}
\begin{gathered}
\{t \in [{\ti_0(\e)}-2\tau-r, T_0+2\tau+r]:\; \varphi_t(\Lambda^0_{B_{0,i}}(r)) \cap \Lambda^0_{_{\!\SigH}}(r) \neq \emptyset\big\} \subset \bigcup_{j=1}^{N_i} I_{0,i,j},\\
\text{and}\\
{\bigcup_{t\in I_{0,i,j}}\varphi_t(\Lambda_{B_{0,i}}^0(r){)}\cap \Lambda^0_{_{\!\SigH}}(r)\neq \emptyset.}
\end{gathered}
\end{equation}
 For $i \in \{1, \dots, N\}$ and $j \in \{1, \dots, N_i\}$ define
\begin{equation}\label{e:D def}
D_{0,i,j}:= \bigcup_{t \in I_{0,i,j}} \varphi_t(\Lambda^0_{B_{0,i}}(r)) \cap \Lambda^0_{_{\!\SigH}}(r).
\end{equation}
We claim that {for each pair $(i,j)$}
\begin{equation}\label{e:Disk}
D_{0,i,j}\subset \bigcup_{\ell=1}^{{L_{i,j}}} \Lambda_{B_{0,i,j,\ell}}^0(r)
\end{equation}
where $\{B_{0,i,j,\ell}\}_{\ell=1}^{{L_{i,j}}}$ are balls centered in $\SigH$ with radii  {$R_{0,i,j,\ell}:=\gamma e^{-\Decay\Lambda  t_{0,i,j}}R_{0,i}$} satisfying
\begin{equation}\label{e:radii}
 {L_{i,j}}R_{0,i,j,\ell}^{n-1}\leq \alpha {\f( t_{0,i,j})}R_{0,i}^{n-1}
\end{equation}
(see Figure~\ref{f:contract} for an illustration of this covering),
where $t_{0,i,j}:=\min\{t:\,t \in I_{0,i,j}\}$. 
Note that $t_{0,i,j}>1$ for all $(i,j)$ since $r<1$ and ${\ti_0(\e)\geq t_0}>\Tinj+2$, and so ${\ti_0(\e)}-2\tau-r >{\ti_0(\e)}-\Tinj-1>1$.  
\begin{figure}
\includegraphics[height=8cm]{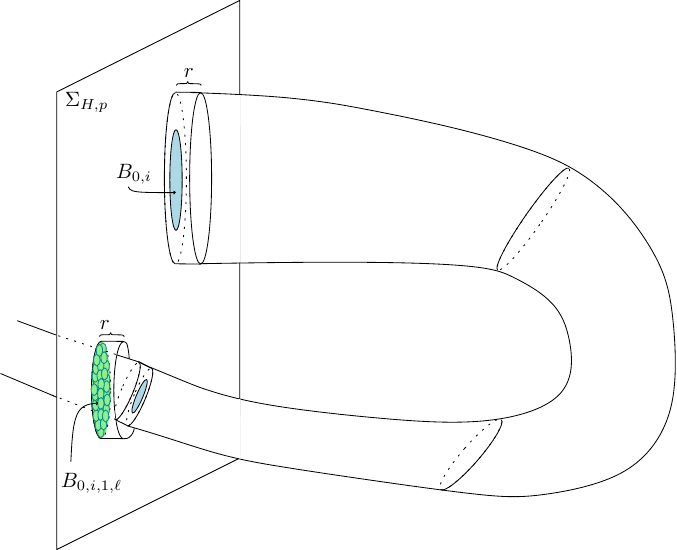}
\caption{\label{f:contract} Illustration of a contracting ball and the cover by much smaller balls for the proof of Lemma~\ref{l:nonlooping}.}
\end{figure}

Note that, since {we take $0<r<R_0<\digamma^{-1}{e^{-\digamma \Lambda T_0}}$, if we let $\digamma_{_{\!\!0}}=\digamma_{_{\!\!0}}(M,p,K_{_{\!H}})$ large enough and assume $\digamma\geq \digamma_{_{\!\!0}}$}, then $\SigH$ is almost flat as a submanifold {of $\TM$ at scale $R_0$}. In particular,  we have
\[
\mathcal B (\rho, \tfrac{1}{2} R)  \cap \Lambda^0_{_{\! \SigH}}(r) \; \subset\;  \Lambda_{ B(\rho, R)}^0(r),
\]
for all $\rho \in \SigH$ and $0\leq R\leq R_0$. Here we are using $\mathcal B$ to denote a ball in $\TM $ and $B$ to denote a ball in $\SigH$.
Therefore, it suffices to show that 
\begin{equation}\label{e:goalD}
D_{0,i,j}\subset \bigcup_{\ell=1}^{L_{i,j}} \mathcal B_{0,i,j,\ell}.
\end{equation}
where $\{\mathcal B_{0,i,j,\ell}\}_{\ell=1}^{L_{i,j}} \subset \TM $ are balls with radii 
$\mathcal R_{0,i,j,\ell}=\tfrac{1}{2}R_{0,i,j,\ell}$ with $R_{0,i,j,\ell}$ as in \eqref{e:radii}.

 {Let $\rho_{0,i}\in A_0$ be the center of $B_{0,i}$} and fix $j\in \{1, \dots, N_i\}$. To prove the claim in \eqref{e:goalD} fix $t_{\!_{\rho_{0,i}}} \in I_{0,i,j}$ so that $\varphi_{t_{\!_{\rho_{0,i}}}}(\rho_{0,i})\in \Lambda^0_{\SigH}(r)$. Observe that choosing coordinates near $\rho_{0,i}$ and $\varphi_{_{{t_{\!_{\rho_{0,i}}}}}}(\rho_{0,i})$, we have for $t$ near $t_{\rho_{0,i}}$ and $\rho$ near $\rho_{0,i}$, 
$$
\varphi_{t}(\rho)=\varphi_t(\rho_{0,i})+d\varphi_t(\rho-\rho_{0,i})+O(|\rho-\rho_{0,i}|^2{e^{2\Lambda|t|}}). 
$$
 {If $|\rho-\rho_{0,i}|\leq R_{0,i}$} and $\rho \in \SigH$, this gives 
$$
\varphi_{t}(\rho)=\varphi_t(\rho_{0,i})+J_t(\rho-\rho_{0,i})+O(R_{{0,i}}^2e^{2\Lambda|t|}). 
$$
Now, let $\{\lambda_i(t)\}_{i=1}^{n-1}$ be the {singular values} of $J_t$ ordered so that $\lambda_i(t)\leq \lambda_{i+1}(t)$. Then, modulo perturbations controlled by $R_0^2e^{2\Lambda |t|}$, the set $\varphi_t(B_{0,i})$ is an $n-1$ dimensional ellipsoid with axes of length $\lambda_i(t)R_{0,i}$. 
Also, observe that
$$
\frac{e^{-\Lambda t}}{C_{_{\!\varphi}}}\leq \lambda_1(t)\leq \lambda_{n-1}(t)\leq C_{_{\!\varphi}} e^{\Lambda t},
$$
{where $C_{_{\!\varphi}}$  is as in \eqref{e:cphi}.}
Since ${\ti_0(\e)}\geq 1$, we note that  $e^{-\Lambda {\ti_0(\e)} (\Decay -1)}<\frac{1}{C_{_{\!\varphi}}} $. This ensures that $e^{-\Decay \Lambda t} < \frac{e^{-\Lambda t}}{C_{_{\!\varphi}}}$ for all $t \geq {\ti_0(\e)}$.

Also, note that there exists a constant $\alpha_{_{\!M,p}}>0$ so that for all $i \in \{1, \dots, N\}$ and $\rho \in  \varphi_{_{t_{\rho_{0,i}}}}(\Lambda^0_{B_{0,i}}(r))$ we have $d(\rho, \varphi_{t_{\rho_{0,i}}}(B_{0,i}) )\leq \alpha_{_{\!M,p}} e^{\Lambda t_{\rho_{0,i}}}r$. {Define $\digamma$ by }
$$
\digamma:= \max\{8\alpha_{_{\!M,p}}\,,\,{\Decay+1}\,,\,\digamma_{_{\!\!0}}\},
$$ 
and {from now on work with $R_0\leq \tfrac{1}{\digamma}e^{-\digamma\Lambda|T_0|}$.}
Then, {if $0<r<\tfrac{1}{\digamma}e^{-\digamma\Lambda T_0}r_0$,} we have that  $r$ is small enough so that $\alpha_{_{\!M,p}}e^{\Lambda T_0}r \leq \frac{1}{8}e^{-\Decay \Lambda T_0}r_0$. In particular,  $\alpha_{_{\!M,p}} e^{\Lambda t_{\rho_{0,i}}}r<\frac{1}{8}e^{-\Decay \Lambda t_{0,i,j}}R_{0,i}$ {for all $i \in \{1, \dots, N\}$} and there are points $\{q_\ell\}_{\ell=1}^{{L_{i,j}}}\subset  \varphi_{_{t_{\rho_{0,i}}}}(B_{0,i})$ so that 
\begin{equation}\label{e:D union}
 \varphi_{_{t_{\rho_{0,i}}}}(\Lambda^0_{B_{0,i}}(r))\subset \bigcup_{\ell=1}^{{L_{i,j}}} \mc{B}(q_{\ell}, \tfrac{1}{8}e^{- \Decay \Lambda t_{0,i,j}}R_{0,i}),
\end{equation}
where the balls in the right hand side are balls in $\TM $.
Furthermore, 
\begin{align*}
\vol(\varphi_{t_{\rho_{0,i}}}(B_{0,i}))&\leq \vol(B_{0,i})  ( |\det (J_{t_{\rho_0,i}})|+ C_{_{\!M,p}}R_0^2e^{2\Lambda t_{\rho_{0,i}}}) \\
 &\leq C_{n} R_{0,i}^{n-1}( {\f(t_{\rho_{0,i}})}+ C_{_{\!M,p}}R_0^2e^{2\Lambda t_{\rho_0,i}})  
\end{align*}
 for some $C_{n}>0$ and $C_{_{\!M,p}}>0$.
 {Next, adjust $\digamma$ so that $\digamma^2>C_\varphi C_{_{\!M,p}}$. Then, since $f(|t|) \geq \frac{1}{C_\varphi} e^{-\Lambda t}$, 
 \begin{align*}
\vol(\varphi_{t_{\rho_{0,i}}}(B_{0,i}))\leq {2} C_{n} R_{0,i}^{n-1} {\f(t_{\rho_{0,i}})}.
\end{align*}}
{Observe that by~\eqref{e:D} and 
$t_{\rho_{0,i}}-\frac{\Tinj}{2}\leq t_{0,i,j}\leq t_{\rho_{0,i}}$, we have $e^{-\Decay \Lambda t_{0,i,j}}< \lambda_1(t_{\rho_0,i})$.}
{Therefore, using that $t_{0,i,j}\leq t_{\rho_{0,i}}$ again,} the points $\{q_\ell\}_{\ell=1}^{L_{i,j}}$ can be chosen so that  
{\begin{align}\label{e:volumebound}
 {L_{i,j}C_{n} ({\tfrac{1}{8}}e^{-\Decay \Lambda  t_{0,i,j}}R_{0,i})^{n-1}}
 &\leq 2  \vol \Big(\varphi_{t_{\rho_{0,i}}}(B_{0,i}) \;  \bigcap \; \cup_{\ell=1}^{{L_{i,j}}} \mc{B}(q_{\ell}, {\tfrac{1}{8}}e^{- \Decay \Lambda t_{0,i,j}}R_{0,i}) \Big) \notag\\
 & \leq  {4}{C_{n} R_{0,i}^{n-1}{\f(t_{0,i,j})}} .
 \end{align}}
 {Note that this yields $L_{i,j} ({\tfrac{1}{8}}e^{-\Decay \Lambda  t_{0,i,j}})^{n-1} \leq {4}\f(t_{0,i,j})$.}
 
 Since $|I_{0,i,j}|<1$,
 it follows that {for every choice of indices $\ell$, $(i,j)$ we have}
\begin{align}\label{e:diam}
\text{diam}\Big(\bigcup_{{t\in I_{0,i,j}}}\varphi_{_{t-t_{\rho_0,i}}}({\mc{B}}({q_{\ell},{\tfrac{1}{8}}e^{-\Decay\Lambda  t_{0,i,j}}R_{0,i}))}\cap \Lambda^0_{_{\!\SigH}}(r)\Big)
&\leq {\frac{1}{8}C_{_{\!\varphi}}e^{\Lambda}} e^{-\Decay\Lambda   t_{0,i,j} }R_{0,i}{\leq \frac{1}{8}R_{0,i} }
\end{align}
where in the last inequality, we use the definition of $\Decay$. Without loss of generality, we may assume that $C_{_{\!\varphi}}\geq 4$ (redefining $\Decay$ in the process) and hence that $\gamma=\frac{1}{4}C_{_{\!\varphi}}e^{\Lambda}\geq 1$ (see \eqref{e:radii}). This implies that we can find  a point $\rho_{0,i,j,\ell}\in \SigH$ so that the  ball $\mathcal B_{0,i,j,\ell}\subset \TM $ of center $\rho_{0,i,j,\ell}$ and radius $\mathcal R_{0,i,j,\ell}=\tfrac{1}{2}\gamma e^{-\Decay\Lambda t_{0,i,j }}R_{0,i}=\tfrac{1}{2}R_{0,i,j,\ell}$ contains the set {in \eqref{e:diam}} whose diameter is being bounded.  Thus, {by the definition \eqref{e:D def} of $D_{0,i,j}$ together with \eqref{e:D union}, we conclude that  \eqref{e:goalD} and   \eqref{e:Disk} hold}.  Also, {by {the definition of $R_{0,i,j,\ell}$, the definition of $\alpha$,} and \eqref{e:volumebound}, for each choice of $(i,j)$ }
$$
\sum_{\ell=1}^{L_{i,j}} R_{0,i,j,\ell}^{n-1}
= {L_{i,j}} \gamma^{n-1} (e^{-\Decay\Lambda t_{0,i,j }}R_{0,i})^{n-1}     \leq\alpha {\f(t_{0,i,j})}R_{0,i}^{n-1},
$$
and hence~\eqref{e:radii} holds.
Therefore, {from the definition of $\ti_0(\e)$ it follows that }
\begin{equation}\label{e:sum of radii}
\sum_{i,j, \ell} R_{0,i,j,\ell}^{n-1} \leq \alpha \sum_{i,j}{\f(t_{0,i,j})} R_{0, i}^{n-1} \leq   \frac{4 \alpha }{\Tinj}\,{\int_{{\ti_0(\e)}}^\infty\f(s)ds}  \sum_i R_{0, i}^{n-1}\leq {\e \sum_i R_{0, i}^{n-1}},
\end{equation}
where to get the second inequality we used that   $t_{0,i,j+1}-t_{0,i,j}\geq   \Tinj/4$ implies
\[
\sum_j  \tfrac{\Tinj}{4}   {f(t_{0,i,j})}  \leq {\int_{{\ti_0(\e)}}^\infty f(s)ds}  .
\] 

Let $k=k(i,j,\ell)$ be an index reassignment and write $\tilde B_{1,k}=B_{0,i,j,\ell}$ {and $R_{1,k}=R_{0,i,j,\ell}$. Note that by the definition of $R_{0,i,j,\ell}$ in \eqref{e:radii} and the first inequality in  \eqref{e:Tplus} we know $R_{1,k}\leq \tfrac{1}{4}R_0$.} {In addition, $\cup_{i,j}D_{0,i,j} \subset \cup_k \tilde B_{1,k}$}. {According to~\eqref{e:intervals} and~\eqref{e:D def} we proved that
\begin{equation}\label{e:a}
\bigcup_{t={\ti_0(\e)}-2\tau-r}^{T_0+2\tau+r} \varphi_t(\Lambda^0_{{A_1}\backslash\cup_k\tilde{B}_{1,k}}(r))\cap \Lambda_{{\SigH}\backslash\cup_k\tilde{B}_{1,k}}^0(r)=\emptyset.
\end{equation}
We claim that this implies
\begin{equation}\label{e:b}
\bigcup_{t={\ti_0(\e)}}^{T_0} \varphi_t(\Lambda^\tau_{{A_1}\backslash\cup_k\tilde{B}_{1,k}}(r))\cap \Lambda_{{\SigH}\backslash\cup_k\tilde{B}_{1,k}}^\tau(r)=\emptyset.
\end{equation}
Indeed, if $\rho$ belongs to the set in \eqref{e:b}, then there exist times $t\in [{\ti_0(\e)}-\tau-r,T_0+\tau+r]$, $s\in [-\tau-r, \tau+r]$,  and points $q_0,q_1{\in \mc{H}_{\Sigma}}$ {(see \eqref{e:Hsig})} with 
$$
{d(q_0, A_1\backslash\cup_k\tilde{B}_{1,k})<r},\qquad {d(q_1, {\SigH}\backslash\cup_k\tilde{B}_{1,k})<r}
$$ 
so that $\rho=\varphi_t(q_0)=\varphi_s(q_1)$. Let $\tau' \in [-\tau, \tau]$ be so that  $|s-\tau'|<r$. Then, $\varphi_{-\tau'}(\rho)=\varphi_{s-\tau'}(q_1)=\varphi_{t-\tau'}(q_0)$ belongs to the set in \eqref{e:a} since $|s-\tau'|<r$ and $t-\tau'\in [{\ti_0(\e)}-2\tau-r,T_0+2\tau +r]$. This means that if the set in \eqref{e:a} is empty, then so is the set in \eqref{e:b}.} Finally, \eqref{e:b} implies that
\[{\Lambda_{A_1}^\tau(r)} \backslash \bigcup_k\Lambda_{\tilde B_{1,k}}^\tau(r)\]
is non self looping for times in $[{\ti_0(\e)}, T_0]$. {Furthermore, \eqref{e:sum of radii} now reads
\[
\sum_{k} R_{1,k}^{n-1} \leq  \e \sum_i R_{0, i}^{n-1}.
\]}
\end{proof}

\begin{lemma}\label{l:nonlooping2}
Let $E\subset \SigH$ be a ball of radius $\delta>0$. Let {$\e_0>0$, $\ti_0:[\e_0, +\infty) \to [1, +\infty)$,} $T_0>0$, and $\digamma>0$, have the property {that $E$ can be $(\e_0, {\ti_0}, \digamma)$-controlled up to time $T_0$ in the sense of Definition \ref{d:control}.}
Let $0<m<\frac{\log T_0-\log {\ti_0(\e)}}{\log 2}$ be a positive integer,
\[
0\leq R_0\leq {\min}\Big\{{\tfrac{1}{\digamma}}e^{-{\digamma}\Lambda T_0}, {\tfrac{\delta}{10}}\Big\},
\qquad 
0<r_1<{\tfrac{1}{5\digamma}}e^{-({\digamma}+2\Decay)\Lambda T_0}R_0,
\]
and $E_0\subset E$ with $d(E_0, E^c)>R_0$. {{Let $0<\tau<\tfrac{1}{2}\Tinj$} and } suppose that $\Lambda_{_{\rho_j}}^\tau(r_1)$ is a $({\mathfrak{D}},\tau, r_1)$ good cover of $\SigH$ and set
$$
\mc{E}:=\{j \in \{1, \dots, {N_{r_1}}\}: \Lambda_{\rho_j}^\tau(r_1)\cap \Lambda^\tau_{E_0}(\tfrac{r_1}{5})\neq \emptyset\}.
$$ 

 Then, there exist $C_{_{\!M,p}}>0$ depending only on $(M,p)$  and  sets $\{\mc{G}_\ell\}_{\ell =0}^m\subset \{1,\dots N_{r_1}\}$, $\mc{B}\subset \{1,\dots N_{r_1}\}$ so that 
 \[\mc{E}\;\subset\;  \mc{B}\cup \displaystyle\bigcup_{\ell=0}^m\mc{G}_\ell,\]
\begin{align}
&\bullet \;\bigcup_{i\in \mc{G}_\ell}\Lambda_{\rho_i}^\tau(r_1)\text{\; is \;} [\ti_0,2^{-\ell}T_0]\text{\; non-self looping {for every $\ell \in\{0, \dots, m\}$},} \label{e:nsl} \\
&\bullet\;  |\mc{G}_\ell|\leq C_{_{\!M,p}}{\mathfrak{D}}\e_0^\ell {\delta^{n-1}} r_1^{1-n} \;\;\; {\text{for every}\;\; \ell \in\{0, \dots, m\}}, \label{e:count good}\\ \  \medskip
&\bullet\;  |\mc{B}|\leq C_{_{\!M,p}}{\mathfrak{D}} \e_0^{m+1}{\delta^{n-1}} r_1^{1-n}\Big.. \label{e:count bad}
 \end{align}
\end{lemma}

\begin{proof}
Choose balls $\{B_{0,i}\}_{i=1}^N$ centered in $E_0$ so that $E_0\subset  \bigcup_{i=1}^NB_{0,i}$ where  $B_{0,i}$ has radius $R_{0,i}=R_0$ built so that {$NR_0^{n-1}\leq C_{n}\delta^{n-1}$}. {This can be done since $R_0<\frac{\delta}{10}$.}  {Let $ r_0:=e^{-{2\Decay \Lambda T_0} }R_0$.} Since $E$ can be $(\e_0, \ti_0, \digamma)$-controlled up to time $T_0$,   for \[0< r<\tfrac{1}{\digamma}e^{-{\digamma}\Lambda T_0}r_0{=}\tfrac{1}{\digamma}e^{-(\digamma+2\Decay)\Lambda T_0}R_0\] there are balls $\{\tilde{B}_{1,k}\}_k\subset \SigH $ of radii {$\{R_{1,k}\}_k\subset[0, \tfrac{1}{4}R_0]$}, so that 
 \[
 {\inf_k}R_{1,k}\geq  e^{-\Decay\Lambda  T_0} R_0 \geq r_0,
 \qquad  \qquad
  \sum_{k}R_{1,k}^{n-1}\leq \e_0 N{R_{0}^{n-1}},
\qquad \qquad 
  \]
and with 
 $G_0:=   \Lambda^\tau_{E_0\backslash \tilde E_1}(r)$  non-self-looping for times in $[{\ti_0(\e)}, T_0],$
 where we have set $\tilde E_1=\cup_k \tilde B_{1,k}$.  {Note that we may assume that $E_0\cap \tilde B_{1,k}\neq \emptyset$  for all $k$.} Now, since $R_{1,k}\leq \frac{1}{4}R_0$, the ball $\tilde{B}_{1,k}$ is centered at a distance no more than $\frac{1}{4}R_0$ from $E_0$. So, letting $E_1:=\cup_k B_{1,k}$ with $B_{1,k}$ the ball of radius $2R_{1,k}$ with the same center as $\tilde{B}_{1,k}$, we have
 $$
 d(E_1,E^c)\geq {d(E_0,E^c)}-\tfrac{3}{4}R_0>(1-\tfrac{3}{4})R_0.
 $$ 
 {Next, we set $T_1:=2^{-1}T_0$ and use that $E_0$ can be {$(\e_0, \ti_0, \digamma)$}-controlled up to time $T_1$ (indeed up to time $2T_1$). By definition $E_1\subset   \bigcup_k  B_{1,k}$ and  $R_0\leq  {\digamma^{-1}}e^{-{\digamma}\Lambda T_0}\leq {\digamma^{-1}}e^{-{\digamma}\Lambda T_1}$.
Therefore, since $0<r<{\digamma^{-1}}e^{-{\digamma}\Lambda T_0}r_0<{\digamma}^{-1}e^{-{\digamma}\Lambda T_1}r_0$, there are balls $\{\tilde B_{2,k}\}_k \subset {\SigH}$ of radii {$0<R_{2,k}\leq \tfrac{1}{4^2}R_0$} with
\begin{equation}\label{e:inf}
 {\inf_k}R_{2,k}\geq  e^{-{\Decay}\Lambda  T_1} {\inf_i}R_{1,i}
\qquad \text{and}\qquad 
  \sum_{k}R_{2,k}^{n-1}\leq \e_0\sum_{k}R_{1,k}^{n-1} \leq \e_0^2 N{R_{0}^{n-1}},
  \end{equation}
 so that  $G_1:=   \Lambda^\tau_{E_1\backslash \tilde E_2}(r)$
is non-self-looping for times in $[{\ti_0(\e)}, T_1],$ where we have set $\tilde E_2=\cup_k \tilde B_{2,k}$. {Since we may assume that $E_1\cap \tilde{B}_{2,k}\neq \emptyset$ for all $k$, the balls $\tilde{B}_{2,k}$ are centered at a distance smaller than $\tfrac{1}{4^2}R_0$ from $E_1$. In particular, letting $E_2=\cup_{k}B_{2,k}$ where $B_{2,k}$ is the ball of radius $2R_{2,k}$ centered at the same point as $\tilde{R}_{2,k}$, we have} 
$$
d(E_2, E^c)\geq d(E_1,E^c)-\tfrac{3}{4^{2}}R_0>R_0\big(1-\tfrac{3}{4}-\tfrac{3}{4^2}\big).
$$
}
 Continuing this way we claim that one can construct a collection of sets $\{G_\ell\}_{\ell=1}^m \subset \Lambda_{E}^\tau(r)$ so that 
 \begin{enumerate}  
   \item[A)]  $G_\ell$ is  non-self-looping for times in $[{\ti_0(\e)}, T_\ell]$ with $T_\ell=2^{-\ell}T_0$.
   \item[B)] There are balls $B_{\ell, k}, \tilde B_{\ell, k} \subset \SigH$ {{centered at $\rho_{\ell,k}\in{E}$}} of radii $2R_{\ell,k}$, $ R_{\ell,k}$ respectively so that 
   \[G_\ell= \Lambda_{E_\ell\backslash\tilde E_{\ell+1}}^\tau(r),\]
where 
$E_\ell= \bigcup_k B_{\ell, k}$ and $\tilde E_{\ell} = \bigcup_k \tilde B_{\ell,k}$.
 \item[C)] For all $\ell \geq 1$, the radii satisfy  {$\sup_\ell R_{\ell,k}\leq \tfrac{1}{4^\ell}R_0$, } 
\begin{equation}\label{e:upperbound}
{\inf_k}R_{\ell, k} \geq  e^{-{2{\Decay} \Lambda T_0} }R_0=r_0 \qquad \text{and} \qquad \sum_k R_{\ell,k}^{n-1} \leq \e_0^\ell N R_0^{n-1}.
\end{equation}
\end{enumerate}
The claim in (A) {follows by construction of $G_{\ell}$. For the claim in (B), we only need to check that the balls $B_{\ell,k}$ are centered in $E$. For this, note} that since $R_{\ell,k} \leq \frac{1}{4^\ell}R_0$, by induction
$$
d(E_\ell, E^c)>d(E_{\ell-1}, E^c)-\tfrac{3}{4^\ell}R_0> R_0\Big(1-{\sum_{j=1}^\ell} \tfrac{3}{4^j}\Big)\geq \frac{1}{{4^\ell}}R_0.
$$
\begin{remark}
{Note that this actually gives $E_\ell \subset E$ and so all of $B_{\ell,k}$ is inside $E$ (not just its center).}
\end{remark}
We proceed to justify the first inequality in \eqref{e:upperbound}.
Note that the construction yields that 
$\inf_k R_{\ell,k}\geq  e^{-{\Decay}\Lambda T_\ell}\inf_{i}R_{\ell-1,i}$
for every $\ell$.
Therefore, {since $T_\ell=2^{-\ell}T_0$  and $\inf_k R_{\ell, k} \geq e^{-\Decay \Lambda T_\ell} \inf_i R_{\ell-1, i}$ (see \eqref{e:inf})}, we obtain
\[
\inf_k R_{\ell, k} \geq \prod_{{j=0}}^\ell e^{-{\Decay}\Lambda \frac{T_0}{{2^{j}}}} R_0 = e^{-{{\Decay} \Lambda T_0(2-{\frac{1}{2^{\ell}}})} }R_0\geq   e^{-{2{\Decay} \Lambda T_0} }R_0.
\]
The construction also yields that $\sum_{k} R_{\ell,k}^{n-1}\leq \e_0 \sum_k R_{\ell-1,k}^{n-1}$
for all $\ell$. Therefore, the upper bound \eqref{e:upperbound} on the sum of the radii follows by induction. Indeed, 
$$
\sum_k R_{\ell,k}^{n-1}\leq \e_0^\ell \sum_k R_{0,k}^{n-1}=\e_0^\ell N R_0^{n-1}.
$$
  {Set $r:=5r_1$ in the  above argument, and define }
$$
\mc{G}_\ell:=\{{i \in \mathcal E} : \Lambda_{\rho_i}^\tau(r_1)\subset G_\ell\},\qquad \mc{B}:=\mc{E}\setminus \bigcup_{{\ell=0}}^m\mc{G}_\ell.
$$
Then, since $G_\ell$ is $[{\ti_0(\e_0)},2^{-\ell}T_0]$ non-self looping,~\eqref{e:nsl} holds. Furthermore, $\mc{E}\;\subset\;  \mc{B}\cup \bigcup_{\ell=0}^m\mc{G}_\ell$  by construction.

 We proceed to prove~\eqref{e:count good}.
Since the cover by tubes can be decomposed into ${\mathfrak{D}}$ sets of disjoint tubes, 
$$
{|\mc{G}_\ell|\leq  {\mathfrak{D}}\frac{\vol(G_\ell \cap \Lambda^\tau_{E_0}(r_1))}{\min_{i}\vol(\Lambda_{\rho_i}^\tau(r_1))} }\leq C_{_{\!M,p}}{\mathfrak{D}}r_1^{1-n}\sum_k{R_{\ell,k}^{n-1}}\leq C_{_{\!M,p}}{\mathfrak{D}}r_1^{1-n}\e_0^\ell N R_0^{n-1},
$$
{for some $C_{_{\!M,p}}>0$ that depends only on $(M,p)$}. {Then, \eqref{e:count good} follows since $NR_0^{n-1}\leq C_n \delta^{n-1}$.}

{The rest of the proof is dedicated to obtaining ~\eqref{e:count bad}}. 
{For each $\ell$ note that  $E_\ell\subset (G_\ell\cup \tilde{E}_{\ell+1})$ and $\Lambda_{E_\ell}^\tau(\frac{r_1}{5})\subset \Lambda^\tau_{_{\!\SigH}}(\frac{r_1}{5}) \subset \cup_i\Lambda_{\rho_i}^\tau(r_1)$. {We claim that} for every pair of indices $(\ell,i)$ with $ \Lambda_{E_\ell}^\tau(\frac{r_1}{5}) \cap \Lambda_{\rho_i}^\tau(r_1)\neq \emptyset$, either 
$$
\Lambda_{\rho_i}^\tau(r_1)\subset \Lambda_{{E_\ell\setminus \tilde{E}_{\ell+1}}}^\tau ({5r_1})\,
 \qquad\text{ or }\qquad \Lambda_{\rho_i}^\tau(r_1)\cap { \Lambda_{{\tilde{E}_{\ell+1}}}^\tau(\tfrac{r_1}{5})}\neq \emptyset.
$$
Indeed, suppose that $\Lambda_{\rho_i}^\tau(r_1)\cap{ \Lambda_{{\tilde{E}_{\ell+1}}}^\tau(\tfrac{r_1}{5})}=\emptyset$. Then, there exists $q\in\mc{H}_{\Sigma}\cap \Lambda_{\rho_i}^\tau(r_1)$ so that $d(q,\rho_i)<r_1$, $d(q,E_\ell)<\frac{r_1}{5}$, $d(q,\tilde{E}_{\ell+1})\geq \frac{r_1}{5}$.  In particular, $d(q,E_\ell\setminus\tilde{E}_{\ell+1})<\frac{r_1}{5}$. Now, suppose that $q_1\in \mc{H}_{\Sigma}\cap \Lambda_{\rho_i}^\tau(r_1)$. Then, 
\[
d(q_1,E_\ell\setminus \tilde{E}_{\ell+1})\leq d(q_1,\rho_i)+d(\rho_i,q)+d(q,E_\ell\setminus {\tilde{E}_{\ell+1}})<\tfrac{11}{5}r_1<{5r_1}.
\]
In particular, $\Lambda_{\rho_i}^\tau(r_1)\subset \Lambda_{E_\ell \setminus \tilde{E}_{\ell+1}}^\tau({5r_1})$ as claimed.

{Now, suppose that $\Lambda_{\rho_i}^\tau(r_1)\cap {\Lambda_{\tilde{E}_{\ell+1}}^\tau}(\tfrac{r_1}{5})\neq \emptyset$. Then, since} $r_1<\frac{r_0}{5}$ {and $R_{\ell,k}\geq r_0$}, we have
  $$\Lambda_{\rho_i}^\tau(r_1)\cap\mc{H}_\Sigma \; \subset E'_{\ell+1}$$ 
  where $E'_{\ell+1}=\cup_{j}\frac{3}{2}\tilde{B}_{\ell+1,j}$.}
Observe then that {for all $\ell$}
{
\begin{equation}
\label{e:indStep}
\Lambda_{E_\ell}^\tau(\tfrac{r_1}{5})\cap \Big( \bigcup_{i\in \mc{G}_\ell}\Lambda_{\rho_i}^\tau(r_1)\Big)^c\;\;\subset \;\;\Lambda_{E'_{\ell+1}}^\tau(\tfrac{r_1}{5}).
\end{equation}}
{By induction {in $k\geq 2$} we assume  that
$
\Lambda_{E_0}^\tau(\tfrac{r_1}{5})\cap \Big( \bigcup_{\ell=0}^{{k}-1}\bigcup_{i\in \mc{G}_\ell}\Lambda_{\rho_i}^\tau(r_1)\Big)^c\subset \Lambda_{E'_{{k}}}^\tau(\tfrac{r_1}{5}).
$
{Note that the base case $k=1$ is covered by setting $\ell=0$ in \eqref{e:indStep}.}
Then, using \eqref{e:indStep} with $\ell=k$ together with the inclusion $\tilde{E}_{k}\subset E'_{k}\subset E_{k}$ (in fact the balls defining each set have the same center and radii given respectively by $R_{\ell,k}$, $\frac{3}{2}R_{l,k}$ and $2R_{l,k}$) we obtain 
$$
\Lambda_{E_0}^\tau(\tfrac{r_1}{5})\cap \Big( \bigcup_{\ell=0}^{{k}}\bigcup_{i\in \mc{G}_\ell}\Lambda_{\rho_i}^\tau(r_1)\Big)^c\;\;\subset \;\;\Lambda_{E'_{{k}+1}}^\tau(\tfrac{r_1}{5}).
$$}
In particular, if $i\in \mc{B}$, then 
$
{\Lambda_{E_0}^\tau(\tfrac{r_1}{5})\cap}\Lambda_{\rho_i}^\tau(r_1)\subset \Lambda_{{{E}}_{m+1}}^\tau({\tfrac{r_1}{5}}). 
$

Therefore,
$$
|\mc{B}|\leq C_{_{\!M,p}}{\mathfrak{D}}r_1^{1-n}\sum_i{R_{m+1,i}^{n-1}}\leq C_{_{\!M,p}}r_1^{1-n}\e_0^{m+1} N R_0^{n-1},
$$
for some $C_{_{\!M,p}}$ that depends only on $(M,p)$. {This proves \eqref{e:count bad} since $NR_0^{n-1}\leq C_n \delta^{n-1}$.}
\end{proof}


\renewcommand{\SigH}{\SNH}
\renewcommand{\LambdaH}{\Lambda^\tau_{\!\SNH}}

\addcontentsline{toc}{section}{\quad\;\,\bf{Construction of covers in concrete settings}}

\section{No Conjugate points: Proof of Theorems \ref{t:noConj2} and~\ref{t:noConj1}}
\label{s:dynNoConj}

{We dedicate this section to the proofs of  Theorems \ref{t:noConj2} and~\ref{t:noConj1}. We work with the Hamiltonian $p:\TM  \to \R$ given by $p(x,\xi)=|\xi|_{g(x,\xi)}-1$.  The Hamiltonian flow  $\varphi_t$ associated to it is the geodesic flow, and for any $H \subset M$ we have ${\Sigma_{_{\!H,p}}}=\SNH$.

{Let $\Lambda>\Lambda_{\max}$}, $t_0 \in \R$, $\ep>0$, and $x\in M$. The study of the behavior of the geodesic flow near $\SNH$ under the no conjugate points assumption hinges on the fact that if there are no more than $m$ conjugate points (counted with multiplicity) along $\varphi_t$ for $t\in (t_0-2\e,t_0+2\e)$,  then for every $\rho\in S^*_xM$ there is a subspace  $\mathbf{V_{\!\rho}} \subset T_\rho S^*_xM$   of dimension $n-1-m$  so that for all ${\bf v}\in\mathbf{V_{\!\rho}}$,
$$
 {\|{\bf v}\|}\leq {C\e^{-{1}}e^{\Lambda|t_0|}}\|(d\pi \circ d \varphi_t)_\rho {\bf v}\|,\qquad t\in (t_0-\e, t_0+\e).
$$
 In particular, this yields that the restriction   $(d\pi \circ d\varphi_t)_\rho: \mathbf{V_{\!\rho}}\to T_{\pi \varphi_t(\rho)}M$ is invertible onto its image with 
\begin{equation}\label{e:normbound}
\|(d\pi \circ d\varphi_t)_\rho^{-1}\|\leq {C \e^{-{1}}e^{\Lambda|t_0|}}.
\end{equation}
{The proof of this result {is included in} Section~\ref{s:prelim} as Proposition \ref{l:panda} and it holds as long as 
\begin{equation}\label{e:epsilonC}
    0<\e<e^{-C\Lambda|t_0|}/C
\end{equation}
for $C>0$, depending only on $(M,g)$ as defined in as in Proposition \ref{l:panda}.}

In what follows we continue to write  $F:\TM \to \re^{n+1}$ for the defining function of $\SNH$ satisfying \eqref{e:defFunction} and we continue to work with 
$$
\psi:\re\times \TM \to \re^{n+1}, \qquad \qquad \psi(t,\rho)=F\circ \varphi_t(\rho).
$$ 
The following lemma is dedicated to finding a suitable left inverse for $d\psi$.
}

\begin{lemma}
\label{l:prelimNoConj}
{Suppose $k>\frac{n+1}{2}$, $\Lambda>\Lambda_{\max}$. {There exists $c_{_H}>0$ depending only on $K_{_H}$ (as defined in~\eqref{e:KH}) such that the following holds}. Let $t_0 \in \R$ and $a > 0$ satisfy
{
$$
d(H, \mc{C}_{_{\!H}}^{2k-n-1,r_{t_0},t_0})>r_{t_0},
$$}
where $r_t=\tfrac{1}{a}e^{-a|t|}$.
Then, if $\rho_0\in \SNH$ and
$$
d(\SNH, \varphi_{t_0}(\rho_0))< \min(r_{t_0},{c_{_H}}),
$$
there exists $ \w_0 \in T_{\rho_0}\SNH$ so that the restriction}
$$
d\psi_{(t_0,\rho_0)}: \R \partial_t \times \re \w_0\to  T_{\psi(t_0,\rho_0)}\re^{n+1} 
$$ 
has left inverse $L_{(t_0,\rho_0)}$ with 
\[
\|L_{(t_0,\rho_0)}\|\leq  {C_{_{\!M,g}}(1+a) e^{C_{_{\!M,g}}(a+\Lambda)|t_0|}}
\]
 where $C_{_{\! M,g}}>0$ is a constant depending only on $(M,g)$.
\end{lemma}
{Note that the assumption $k>\frac{n+1}{2}$ is needed for $\mc{C}_{_{\!H}}^{2k-n-1,r_{t_0},t_0}$ to be defined. The reason why $2k-n-1$ appears in the exponent of $\mc{C}_{_{H}}$ is explained in Remark \ref{r:explainingNumbers}.}

\begin{proof}
{Let $\tilde{F}:=(f_1,\dots, f_k)\in C^\infty(M;\mathbb{R}^k)$ be a defining function for $H\subset M$ such  that $d\tilde F_y$ has {right} inverse ${R}_{_{\! \tilde F,y}}$ with $\|{R}_{_{\! \tilde F,y}}\|\leq 2$ for all $y$ such that $d(y,H)<c_{_H}$. Note that $c_{_H}$ can be chosen uniformly depending only on $K_{_H}$ as in~\eqref{e:KH}.}
 Next, define 
\[
\tilde \psi:\R \times \TM  \to \R^k, \qquad \qquad   \tilde \psi(t, \rho):=  \tilde F \circ \pi \circ \varphi_t (\rho).
\]

We claim that there exists $ \w_0 \in T_{\rho_0} \SNH$ so that 
$$
d\tilde \psi_{(t_0,\rho_0)}: \re\partial_t\times \re  \w_0  \to \re^{{k}}
$$
is injective and has a left inverse bounded by ${C_{_{\!M,g}}(1+a) e^{C_{_{\!M,g}}(a+\Lambda)|t_0|} }$. Note that this is sufficient as this produces a left inverse for $\psi$ itself.

Observe that for $s\in \re$, $\rho\in \SNH$, and $ \w \in T_\rho\SNH$,
\begin{equation}\label{e:split}
d\tilde \psi_{(t,\rho)}(s\partial_t, \w )
=d(\tilde F\circ \pi)_{\varphi_t(\rho)} \big( s\, H_p+ ({d\varphi_t})_\rho\, {\w}\big).
\end{equation}
Note also that since $H$ is conormally transverse for $p$, there exists a neighborhood $W\subset \TM $ of $\SNH$ and $c>0$ so that for $\tilde{\rho}\in W$, 
\begin{equation}\label{e:split1}
\|d(\tilde F\circ \pi)_{\varphi_t(\tilde \rho)} H_p\|\geq \frac{1}{2}.
\end{equation}
In particular, the restriction
$$
d\tilde \psi_{(t_0,\rho_0)}:\re\partial_t\to \re^k
$$
has a left inverse bounded by $2$.

We proceed to find $ \w_0 \in T_{\rho_0}\SNH$ as claimed. 

{Suppose} $d(H, \mc{C}_H^{2k-n-1,r_{t_0},t_0})>r_{t_0}$. {Then, by definition,} for all $x \in H$, and every unit speed geodesic $\gamma$ with $\gamma(0)=x$, there the number of conjugate points to $x$ (counted with multiplicity) along $\gamma(t_0-r_{t_0}, t_0+r_{t_0})$ is { smaller than or equal to $m:=2k-n-2$} whenever $d(\gamma(t_0),H)<r_{t_0}$. In particular, since $d(\varphi_{t_0}(\rho_0),\SNH)<r_{t_0}$, we have $d(\pi(\varphi_{t_0}(\rho_0)),H)<r_{t_0}$. Therefore,  by setting $\ep={\min(r_{t_0}/2, e^{-C\Lambda|t_0|}/C)}$ in \eqref{e:normbound} {with $C$ as in \eqref{e:epsilonC}}, 
we have that there is a ${n-1-m}$ dimensional subspace $\mathbf{V}_{\!\rho_0}\subset T_{\rho_0} S^*_{x_0}M$ so that $d\pi\circ d\varphi_{t_0}|_{\mathbf{V}_{\!\rho_0}}$ is invertible onto its image with
\begin{equation}\label{e:ub on norm}
\|(d\pi \circ d\varphi_{t_0}|_{\mathbf{V}_{\!\rho_0}})^{-1}\|\leq { C \e^{-1} e^{\Lambda|t_0|}}\leq  {C_{_{\!M,g}}(1+a) e^{C_{_{\!M,g}}(a+\Lambda) |t_0|}},
\end{equation}
for some $C_{_{\!M,g}}>0$ depending only on $(M,g)$, and where $x_0:=\pi( \rho_0)$.

Let 
\[
V
=d(\pi \circ \varphi)_{(t_0, \rho_0)}\big(\R\partial_t  \times (T_{\rho_0} (SN_{x_0}^*H)\cap{\mathbf{V}_{\!\rho_0}})\big).
\]
{Note that since $\dim \mathbf{V}_{\!\rho_0}={n-1-m}$, $\dim T_{\rho_0} SN_{x_0}^*H= k-1$, $\dim S^*_{x_0}M=n-1$,  we know that $\dim (T_{\rho_0} SN_{x_0}^*H\cap{\mathbf{V}_{\!\rho_0}}) \geq {k-1-m}$, and so  $\dim V\geq  {k-m}$. }Also,  the restriction
 \[
 d(\pi \circ \varphi)_{(t_0, \rho_0)}
 : \re\partial_t \times (T_{\rho_0} (SN_x^*H){\cap\mathbf{V}_{\!\rho_0}})\to V
 \]
  is invertible with inverse $\tilde L_{(t_0,\rho_0)}$ satisfying
$$
\| \tilde L_{(t_0,\rho_0)}\|\leq  {C_{_{\!M,g}}(1+a) e^{C_{_{\!M,g}}(a+\Lambda)|t_0|}}.
$$
Next, there exists a neighborhood  {$U\subset M$ of $H$ so that for $y\in U$}, $d\tilde F_y:T_{y}M\to \re^k$ is surjective with right inverse $R_y$. {By assumption, $R_y$ is bounded by $2$}. Furthermore, we may assume without loss of generality that for $\rho\in T^*U\cap W$, $d\pi_\rho H_p$ lies in the range of $R_{\pi(\rho)}$. 
Since $\dim (\ran R_{\pi(\varphi_{t_0}(\rho_0))})=k$, $\dim V\geq {k-m}$, and  both $V$ and $\ran R_{\pi(\varphi_{t_0}(\rho_0))}$ are contained in $T_{\pi(\varphi_{t_0}(\rho_0))}M$, we know that 
$$
\dim (\ran R_{\pi(\varphi_{t_0}(\rho_0))} \cap V)\geq {2k-m-n=2}.
$$

{Then, this guarantees the existence of }{$ \w_0 \in T_{\rho_0}(SN_{x_0}^*H)\cap \mathbf{V}_{\!\rho_0} \backslash \{0\}$}, so that 
$$ 
(d\pi \circ d\varphi_{t_0})_{\rho_0} \w_0 \in \ran R_{\pi(\varphi_{t_0}(\rho_0))}.
$$
\begin{remark}\label{r:explainingNumbers} 
{ Note that having $\dim (\ran R_{\pi(\varphi_{t_0}(\rho_0))} \cap V)\geq 1$ would  not have been sufficient as $\partial_t$ is a component we cannot ignore.} {It is here where we need that $2k-m-n=2$. In particular, this step explains why the assumption in the lemma is written for the space $\mc{C}_{_{H}}^{m+1,r_{t_0},t_0}$ with $m=2k-n-2$.} 
\end{remark}
Then, there exists $\mathbf{x}\in \re^k$ so that 
$$
(d\pi \circ d\varphi_{t_0})_{\rho_0} \w_0 =R_{\pi(\varphi_{t_0}(\rho_0))}\mathbf{x}.
$$
Since $\sup_{y\in U}\|R_y\|\leq 2$, 
$$
\|(d\pi \circ d\varphi_{t_0})_{\rho_0} \w_0 \|\leq 2\|\mathbf{x}\|
$$
and{ by \eqref{e:ub on norm}} we have 
$$
\| \w_0 \|\leq C_{_{\!M,g}}a e^{(a+\Lambda )|t_0|}\|\mathbf{x}\|.
$$
which implies the desired claim since $(d\tilde F\circ  d\pi \circ d\varphi_{t_0})_{\rho_0} \w_0 =\mathbf{x}$  and so 
\begin{equation}\label{e:split2}
   \|d(\tilde F\circ \pi)_{\varphi_{t_0}(\rho_0)} ( (d\varphi_{t_0})_{\rho_0} \w_0)\|\geq  (C_{_{\!M,g}}a)^{-1} e^{-(a+\Lambda )|t_0|}{\|\w_0\|}.
\end{equation}
Combining \eqref{e:split1} and \eqref{e:split2} with \eqref{e:split} gives the desired bound on the left inverse for $d\tilde \psi$ restricted to $\re\partial_t \times \re\w_0$ {provided we impose $C_{_{\!M,g}} \geq 2$}. 
\end{proof}

\noindent {\bf Proof of Theorem~\ref{t:noConj1}.}
Let $t_0>0$, $a>{\delta_F^{-1}}$ so that for $t\geq t_0$,
{\begin{equation}\label{e:toprove}
d\Big(H, \mc{C}_H^{2k-n-1,r_t,t}\Big)> r_t,
\end{equation}} 
where $r_t=\tfrac{1}{a}e^{-at}.$
By Lemma~\ref{l:prelimNoConj}, for $t\geq t_0$, if $\rho\in \SNH$ and $d(\varphi_t(\rho),\SNH)<{\min(\tfrac{1}{a}e^{-at},c_{_H})}$, then there exists a $\w=\w(t, \rho)\in T_{\rho}\SNH$ so that  $d\psi$ restricted to $\re\partial_t \times \re\w$ has left inverse $L_{(t,\rho)}$ with 
\[
\|L_{(t,\rho)}\|\leq { C_{_{\! M,g}} (1+a) e^{C_{_{\! M,g}}(a+\Lambda) |t|}},
\]
for some $C_{_{\! M,g}}>0$ and any $\Lambda>\Lambda_{\max}$. 
{For the purposes of the proof of Theorem~\ref{t:noConj1} fix  $\Lambda=2\Lambda_{\max}+1$. 
Let $c:= {(1+a) C_{_{\! M,g}}}$, $\beta:= {C_{_{\!M,g}}(a+\Lambda)}$, and  let $t_1=t_1(a,t_0)\geq t_0$ be so that 
\[
\|L_{(t,\rho)}\|\leq  c e^{\beta|t|} \qquad t\geq t_1.
\]
}In particular, we may cover $\SNH$ by finitely many balls $\{B_i\}_{i=1}^N$ of radius $R>0$ (independent of $h$) so that
$NR^{n-1}<C_{n}\vol(\SNH),$ and the hypotheses of Proposition~\ref{p:ballCover} hold for each $B_i$ choosing $\tilde c={a}^{-1}$. 

Let  {$\alpha_1=\alpha_1(M,g)$ and $\alpha_2=\alpha_2(M,g,a ,\delta_F)$} be as in Proposition~\ref{p:ballCover}.
Fix $0<\e<\frac{1}{4}$ and set
 \[r_0:=h^{2\e}, \qquad r_1:=h^\e, \qquad r_2:=\tfrac{2}{\alpha_1}h^\e.\]
Let {$$T_0(h)=b \log h^{-1}$$ with $b>0$ to be chosen later}. Then, the assumptions in Proposition~\ref{p:ballCover} hold provided
$$
h^{\e}< \min \Big\{{\tfrac{2}{3\alpha_1}} e^{-\Lambda T_0}\;,\;\tfrac{\alpha_1\alpha_2}{2} e^{- \gamma T_0}, {\tfrac{\alpha_1 R}{2}}\Big\}
$$
where {$\gamma = \max\{a, 3\Lambda+2\beta\}=5\Lambda+2a$}. In particular, {if we set $\alpha_3:=\min\{\tfrac{2}{3\alpha_1} ,\tfrac{\alpha_1\alpha_2}{2} \}$, the assumptions in Proposition~\ref{p:ballCover} hold provided {$h<\big(\frac{\alpha_1 R}{2}\big)^{\frac{1}{\e}}$} and
\begin{equation}
\label{e:t0Temp}
T_0(h)< \frac{ \e}{ \gamma} \log h^{-1} + \frac{\log \alpha_3}{ \gamma}.
\end{equation}}
We will choose $T_0$ satisfying~\eqref{e:t0Temp} later.

{Let $0<\tau_0<\Tinj$, ${R_0=R_0(n,k,g,K_{_{\!H}})}>0$ be as in  Theorem~\ref{t:coverToEstimate}.   Note that $\tau_0=\tau_0(M,g,\Tinj)$. Also let $h_0=h_0(M,g)>0$ be the constant given by  Theorem~\ref{t:coverToEstimate} and possibly shrink it so that $h_0<\big(\frac{\alpha_1 R}{2}\big)^{\frac{1}{\e}}$. }
Let {$\{\rho_j\}_j \subset \SNH$ be so that $\{\Lambda^\tau_{_{\rho_j}}(h^\e)\}_j$} is a $({\mathfrak{D}_n},\tau_0,h^\e)$ good cover of $\SNH$ ({existence of such a cover follows from~\cite[Proposition 3.3]{CG18d} - see Remark \ref{r:D depends on n}}).}
Then, for each $i \in \{1, \dots, K\}$ we apply Proposition~\ref{p:ballCover} to obtain a cover of  $\Lambda_{B_i}^{\tau_0}(h^{2\e})$  by tubes   $\{\Lambda_{\rho_{{i_j}}}^{\tau_0}(h^\e)\}_{j=1}^{N_i}$ with $\rho_{{i_j}}\in B_i$ and  so that $\{1, \dots, N_i\}=\mc{G}_i\cup\mc{B}_i$,
$$
\bigcup_{j\in \mc{G}_i}\Lambda_{\rho_j}^{\tau_0}(h^\e) \quad  \text{is }\;\; [t_0,T_0(h)]\;\; \text{ non-self looping,} 
$$
$$
h^{\e(n-1)}|\mc{B}_i|\leq {\bf{C}_0}\tfrac{2}{\alpha_1} \;h^\e \;R^{n-1}\;   T_0{e^{4({\Lambda+\beta})T_0}},
$$
where  {${\bf{C}}_0={\bf{C}}_0(M,g,k,a)>0$}. 
We choose {$b>0$} so that 
{$
b < \frac{\e}{12({\Lambda +\beta})}
$}
{and  \eqref{e:t0Temp} is satisfied for all $h<h_0$.} {Note that this implies that $b=b(M,g,a,\delta_F)$.}
In particular, there exists $h_0=h_0(\tau_0, {\bf{C}}_0)$, so that for all $0<h<h_0$, 
\begin{equation}\label{e:badsets}
h^{\e(n-1)}|\mc{B}_i|< h^{\frac{\e}{3}}R^{n-1}.
\end{equation}

We next apply Theorem~\ref{t:coverToEstimate} $\delta:=2\e$, and $R(h):=h^\e$ (not to be confused with $R$). If needed, we shrink $h_0$ so that $5h^{2\e}\leq R(h)<R_0 $ for all $0<h<h_0$. {We  let $\alpha<1-2\e$ and let $b$ be small enough so that $T_0(h) \leq 2\alpha T_e(h)$} for all $0<h<h_0$.  We also let  $\mc{B}=\cup_{i=1}^K\mc{B}_i$, and work with only one set of good indices $\mathcal G:=\mathcal I_h(w)\backslash \mathcal B$. We choose  {$t_\ell(h)=t_1$} and $T_\ell(h)=T_0(h)$. Note that \eqref{e:badsets} gives 
\[R(h)^{\frac{n-1}{2}}|\mc{B}|^{\frac{1}{2}}\leq  h^{\frac{\e}{6}} (K R^{n-1})^{\frac{1}{2}} \leq  h^{\frac{\e}{6}} { {C_{n}}^{\!\frac{1}{2}}} \vol(\SNH)^{\frac{1}{2}}.\]
Since in addition
\[|\mc{G}| \leq |\mathcal I_h(w)|\leq K (\max_{1\leq i \leq K}N_i) \leq \vol(\SNH)C_{n}{h^{-\e(n-1)}},\]
{Let $N>0$. Theorem~\ref{t:coverToEstimate}  yields the existence of constants $C_{n,k}>0$, $\tilde{C}=\tilde{C}({M,g,\tau_0,\e})>0$ and $C_{_{\!N}}>0$  so that for all $0<h<h_0$}
\begin{align}
&h^{\frac{k-1}{2}}\Big|\int_H w u\,d\sigma_H\Big| \notag\\
&\leq \frac{C_{n,k}{\vol(\SNH)^{\frac{1}{2}}}\|w\|_{_{\!\infty}}C_{_n}^{\tfrac{1}{2}}}{\tau_0^{\frac{1}{2}}}
\!\Bigg(\Bigg[h^{\frac{\e}{6}}+ \frac{t_1^{\frac{1}{2}}}{T_0^{\frac{1}{2}}(h)}\Bigg]\!\!\|u\|_{\LM}\!+\frac{T_0^{\frac{1}{2}}(h)t_1^{\frac{1}{2}}}{h}\|(-h^2\Delta_g-I)u\|_{_{{\sob{-2}}}}\Bigg)\notag \\
\label{e:absorbing} &+\frac{\tilde{C}}{h}{\|w\|_\infty}\|(-h^2\Delta_g-I)u\|_{\Hl}
\!+C_{_{\!N}}h^N\big(\|u\|_{\LM}\!+{\|(-h^2\Delta_g-I)u\|_{{\Hl}}}\big)\\
&\leq C{\|w\|_{_{\!\infty}}}\left(\frac{\|u\|_{\LM}}{\sqrt{\log h^{-1}}}+\frac{\sqrt{\log h^{-1}}}{h}\|(-h^2\Delta_g-I)u\|_{\Hl}\right)\label{e:submarine}
\end{align}
where $C=C(M,g,k,t_0,a,\delta_F, \vol(\SNH), \Tinj)>0$ is some positive constant {and $h_0=h_0(\delta, M,g,\tau_0,k,a,w,R_0)$ is chosen small enough so that the last term on the right of~\eqref{e:absorbing} can be absorbed}. Note that the $\e$ dependence of $C$ and $h_0$ is resolved by fixing any $\e<\tfrac{1}{4}$.
\qed

\ \\
\noindent {\bf Proof of Theorem~\ref{t:noConj2}.}  
{Note that if $H=\{x\}$ then $\SNH=S_x^*M$ and  $\vol(S^*_xM)=c_n$ for some $c_n>0$ that depends only on $n$. Next, note that $\Tinj(\{x\})$ and $\delta_F$ can be chosen uniform on $M$ and that $H_pr_H=2$. {Moreover, in this case, $w=1$ and $K_{_{\!H}}$ can be taken arbitrarily small so {$R_0=R_0(n,k,p,K_{_{\!H}} )$} can be taken to be uniform on $M$.}

Therefore, since the constant in \eqref{e:submarine} and $h_0$ depends only on $$M,\;g,\;k,\;t_0,\;a,\; \delta_F,\; \vol(\SNH),\; \Tinj,$$ all of the terms on the right hand side of~\eqref{e:submarine} are uniform for $x\in M$ completing the proof of Theorem~\ref{t:noConj2}.}
\qed


\section{No focal points or Anosov geodesic flow: Proof of Theorems \ref{T:applications} and \ref{T:tangentSpace}}
\label{s:Anosov}

Next we analyze the cases in which $(M,g)$ has no focal points or Anosov geodesic flow. {For $\rho \in \SNH$ we continue to write $N_{\pm}(\rho)=T_{\rho}(\SNH)\cap E_\pm(\rho)$} and define the  functions $m,m_{\pm}:\SNH\to \{0,\dots,n-1\}$ 
\begin{equation}
\label{e:dim}
\begin{gathered}
m(\rho):=\dim (N_+(\rho)+N_-(\rho)),\qquad m_{\pm}(\rho):=\dim N_{\pm}(\rho),
\end{gathered}
\end{equation}
and note that the continuity of $E_{\pm}(\rho)$ implies that $m,\,m_{\pm}$ are upper semicontinuous ({see e.g. \cite[Lemma 20]{CG17}}).
We will need extensions of $N_\pm(\rho)$, $m_\pm(\rho)$ to neighborhoods of $\SNH$ for our next lemma. To have this,  for each $\rho$ in a neighborhood of $\SNH$ define the set 
$$
\mc{F}_{\!\rho}:=\{q\in \TM :\; F(q)=F(\rho)\},
$$ 
where $F$ is the defining function for $\SNH$ introduced in \eqref{e:defFunction}.
Since for $\rho\in \SNH$, $\mc{F}_{\!\rho}=\SNH$,  $\mc{F}_{\!\rho}$ can be thought of as a family of `translates' of $\SNH$. We then define 
\[
\tilde{N}_{\pm}(\rho):=T_{\rho}\mc{F}_{\!\rho} \cap E_{\pm}(\rho)\qquad \text{and}\qquad \tilde{m}_{\pm}(\rho):=\dim \tilde{N}_{\pm}(\rho).
\]
Note that since $T_{\rho}\mc{F}_{\!\rho}$ is smooth in $\rho$ and agrees with $T_{\rho}(\SNH)$ for $\rho\in \SNH$, $\tilde{m}_{\pm}(\rho)$ is upper semicontinuous with $\tilde{m}_{\pm}|_{\SNH}=m_{\pm}.$ 
In what follows we continue to write  $\mc{S}_H= \{\rho\in \SNH:\;\, T_\rho (\SNH)=N_-(\rho)+N_+(\rho)\}$.

{The following lemma shows that if $\rho \in \SNH$ does not belong to $\mathcal S_H$ and $\varphi_t(\rho)$ is close enough to $\rho$ for $t$ sufficiently large, then $(d\varphi_t)_\rho\w$ leaves ${T_{{\varphi_{t}(\rho)}}}\mc{F}_{\!{\varphi_{t}}(\rho)}$ for some $\w \in T_\rho \SNH$. }

\begin{lemma}\label{P:1} 

Suppose $(M,g)$ has Anosov geodesic flow {or no focal points} and let $K \subset (\SNH \backslash {\mc{S}_H})$ be a compact set. Then there exist  positive constants $c_{_{\! K}},t_{_{\! K}}, \delta_{_{\! K}}>0$ so that if ${d(\rho, K)\leq \delta_{_{\!K}}}$, $|t| \geq t_{_{\! K}}$, and 
\[
{\varphi_{t}}(\rho) \in \; \overline{B(\rho, \delta_{_{\! K}})},
\]
then there is $\mathbf{w}=\mathbf{w}(t, \rho)\in T_{\rho}(\SNH){\setminus \{0\}}$ with
\begin{equation}
\label{e:noTangent1}
\inf\{\|d{\varphi_{t}} (\mathbf{w})+{\bf v}\|:\; {\bf v} \in  T_{\varphi_{t}(\rho)}\mc{F}_{\!{\varphi_{t}}(\rho)} {+} \R H_p\}\geq c_{_{\! K}} \|\mathbf{w}\|.
\end{equation}
\end{lemma}

\begin{proof}
First note that since $\tilde{m}_{\pm}$ are upper semi-continuous, $K$ is compact, and $K\cap \mc{S}_H$ is empty, there exists $\delta_{_{\! \tilde K}}>0$ so that $d(K,\mc{S}_H)>\delta_{_{\! \tilde K}}.$ Therefore, to prove the lemma we work with  the compact set $\tilde{K}:=\{\rho \in \SNH:\; d(\rho,K)\leq \frac{\delta_{_{\!\tilde K}}}{2}\}$ and insist that $\delta_K<\frac{\delta_{_{\! \tilde K}}}{2}$.

 Let $\rho\in \tilde{K}$. Since
 $T_{\rho}(\SNH)\neq N_+(\rho)+ N_-(\rho)$,
we may choose $ {\bf u}= {\bf u}(\rho)$ such that
 \[
 {\bf u} \in T_{\rho}(\SNH) \setminus ( N_+(\rho)+ N_-(\rho)),\qquad \|{\bf u}\|=1.
 \]
Now, let $\mathbf{u}_+\in E_+(\rho)$ and $\mathbf{u}_-\in E_-(\rho)$ be so that 
\[
\mathbf{u}=\mathbf{u}_++\mathbf{u}_-.
\]
{In particular, ${\bf u}_{\pm}\notin N_\pm(\rho)$.}

{When studying the case $t>t_K$, we will use that ${\bf u}_-$ grows along the positive time flow, while for $t<-t_K$ we will use that ${\bf u}_+$ grows along the negative time flow. Since the arguments are identical, except with time reversed and the roles of ${\bf u}_+$ and ${\bf u}_-$ switched, we only explicitly write that for $t>t_K$.}

{
We claim that there is $C_{_{\!K}}>0$ such that for all $\rho \in \tilde{K}$, we may in addition choose ${\bf u}= {\bf u}(\rho)$ such that 
\begin{equation}
\label{e:size}
\begin{gathered}
{\bf u}_-\in E_-(\rho)\cap (N_-(\rho))^\perp \cap (E_+(\rho)\cap E_-(\rho))^\perp,\\
C_{_{\!K}}^{-1}\|{\bf u}_+\|\leq \|{\bf u }_-\|\leq C_{_{\!K}}\|{\bf u}_+\|.
\end{gathered}
\end{equation}
For this, we set 
$$
\begin{gathered}\bar{N}_\pm(\rho):= N_\pm(\rho)\cap \Big(E_+(\rho)\cap E_-(\rho)\Big)^\perp, \\ U_{\pm}(\rho):=E_\pm(\rho)\cap (N_\pm(\rho))^\perp \cap \Big(E_+(\rho)\cap E_-(\rho)\Big)^\perp.
\end{gathered}
$$
We then observe that 
$$
(\mathbb{R}H_p(\rho))^\perp= U_+(\rho)\oplus \bar{N}_+(\rho)\oplus \Big(E_+(\rho)\cap E_-(\rho)\Big)\oplus \bar{N}_-(\rho)\oplus U_-(\rho)
$$
and decompose a vector ${\bf v}\in (\mathbb{R}H_p)^\perp$ correspondingly as
$$
{\bf v}={\bf v}_{U_+}+{\bf v}_{\bar{N}_+}+{\bf v}_0+{\bf v}_{\bar{N}_-}+{\bf v}_{U_-}.
$$
Suppose the claim in \eqref{e:size} fails. Then, for all $n\in\mathbb{N}$, there are $\rho_n\in \tilde{K}$ such that for all ${\bf v}\in T_{\rho_n}\SNH$, 
$$
n^{-1}\|{\bf v}_{U_+}+{\bf v}_{\bar{N}_+}+{\bf{v}}_0\|> \|{\bf v}_{U_-}+{\bf v}_{\bar{N}_-}\|,\qquad \text{ or }\qquad n\|{\bf v}_{U_+}+{\bf v}_{\bar{N}_+}+{\bf{v}_0}\|<\|{\bf v}_{U_-}+{\bf v}_{\bar{N}_-}\|.
$$
In particular, since ${\bf v}_{\bar{N}_-}\in T_{\rho_n} \SNH$, we have ${\bf{v}}-{\bf{v}}_{\bar{N}_-}\in T_{\rho_n}\SNH$, and hence, for all ${\bf v}\in T_{\rho_n}\SNH$,
$$
n^{-1}\|{\bf v}_{U_+}+{\bf v}_{\bar{N}_+}+{\bf{v}}_0\|> \|{\bf v}_{U_-}\|,\qquad \text{ or }\qquad n\|{{\bf v}_{U_+}}+{\bf v}_{\bar{N}_+}+{\bf{v}_0}\|<\|{\bf v}_{U_-}\|.
$$

Since $\tilde{K}$ is compact, we may assume $\rho_n\to \rho\in \tilde{K}$. Then, for all ${\bf v}\in T_\rho \SNH$, there are ${\bf v}_n\in T_{\rho_n}\SNH$ such that ${\bf v}_n\to {\bf v}$.  Let ${\bf v}\in T_{\rho}\SNH\setminus (N_+(\rho)+N_-(\rho))$ and ${\bf v}_n \to {\bf v}$ as above.

Then, 
$$
n^{-1}\|{\bf v}_{n,U_+}+{\bf v}_{n,\bar{N}_+}+{\bf{v}}_{n,0}\|> \|{\bf v}_{n,U_-}\|,\qquad \text{ or }\qquad n\|{\bf v}_{U_+}+{\bf v}_{n,\bar{N}_+}+{\bf{v}}_{n,0}\|< \|{\bf v}_{n,U_-}\|.
$$
Extracting a subsequence again, we may assume that one of these inequalities holds for all $n$. We consider first the case
$$
n^{-1}\|{\bf v}_{n,U_+}+{\bf v}_{n,\bar{N}_+}+{\bf{v}}_{n,0}\|> \|{\bf v}_{n,U_-}\|.
$$
 Now, since ${\bf v}_n\to {\bf v}$, and $E_+(\rho)$ is continuous,
$$
{\bf v}_{n,U_+}+{\bf v}_{n,\bar{N}_+}+{\bf{v}}_{n,0}\to \tilde{\bf v}_+\in E_+(\rho)
$$
In particular, this implies that ${\bf v}_{n,U_-}\to 0$ and hence ${\bf v}_{n,\bar{N}_-}\to {\bf v}-\tilde{\bf v}_+.$ Using that $\rho \mapsto T_\rho \SNH$ and $\rho\mapsto E_-(\rho)$ are continuous maps, and that ${\bf v}_{n,\bar{N}_-}\in E_-(\rho_n)\cap T_{\rho_n}\SNH$, we have ${\bf v}-\tilde{\bf v}_+\in N_-(\rho)$ and hence also ${\bf v}_+\in N_+(\rho)$. Therefore, ${\bf v}\in N_+(\rho)+N_-(\rho)$, a contradiction.

Next, we consider the other case:
$$n\|{\bf v}_{U_+}+{\bf v}_{n,\bar{N}_+}+{\bf{v}}_{n,0}\|< \|{\bf v}_{n,U_-}\|.$$
Then, since ${\bf v}_n\to {\bf v}$, ${\bf v}_{n,U_-}$ is bounded and hence ${\bf v}_{U_+}+{\bf v}_{n,\bar{N}_+}+{\bf{v}}_{n,0}\to 0$. In particular, ${\bf v}_{n,U_-}+{\bf v}_{n,\bar{N}_-}\to {\bf v}$, so ${\bf v}\in E_-(\rho)$ and hence ${\bf v }\in N_-(\rho)$, a contradiction. Since both cases lead to a contradiction, we have proved the claim~\eqref{e:size}.
}


{Since $d\varphi_t:E_-(\rho) \to  E_-(\varphi_t(\rho))$ and $d\varphi_t : E_+(\rho)\cap E_-(\rho) \to E_+(\varphi_t(\rho))\cap E_-(\varphi_t(\rho)) $     are  isomorphisms,}
 we have
\[
\dim \Span\begin{pmatrix} d\varphi_t(\mathbf{u}_-), &d\varphi_t(N_-(\rho))\end{pmatrix}=1+\dim N_-(\rho).
\]
Also, note that since $\tilde{m}_-$ is upper semicontinuous {and integer valued}, we may choose $\delta>0$ uniform in $\rho \in \SNH$ so that $\dim \tilde{N}_-(q)\leq \dim N_-(\rho)$ for all $q \in B(\rho, \delta)$.
For any $t$ {and $q\in B(\rho, \delta)$} we then have

\begin{equation}\label{E:dimension}
\dim \Span\begin{pmatrix} d\varphi_t(\mathbf{u}_-), &d\varphi_t(N_-(\rho))\end{pmatrix}\geq1+\dim{\tilde N_-(q)}. \medskip
\end{equation}
Next, note that $\Span\!\begin{pmatrix} d\varphi_t(\mathbf{u}_-), &d\varphi_t(N_-(\rho))\end{pmatrix} \subset E_-(\varphi_t(\rho))$.
Suppose now that $\varphi_t(\rho)\in B(\rho,\delta)$ {for some $t$} and note that if {$d\varphi_t(\w) \in E_-(\varphi_t(\rho)) \backslash \tilde{N}_-(\varphi_t(\rho))$, then $d\varphi_t(\w)  \notin T_{\varphi_t(\rho)}\mc{F}_{\!\varphi_t(\rho)}$}. 
In particular, relation \eqref{E:dimension} gives that there exists a linear combination
\[
{\bf w_t}= a_t \,\mathbf{u}_- +  {\bf e}_-(t) \;{\in E_-(\rho)},
\]
with ${\bf e}_-(t) \in N_-(\rho)$, so that {$d\varphi_t{\bf w_t}\in (\tilde{N}_-(\varphi_t(\rho)))^\perp$ with $\|d\varphi_t{\bf w_t}\|=1$.}
{If we had that ${\bf w_t}$ was a tangent vector in $T_{\rho}(\SNH)$  {and we had control on $\|{\bf w}_t\|$} we would be done proving \eqref{e:noTangent1}. {Note that to say this we are using that $d\varphi_t \w_t \in E_-(\varphi_t(\rho))$ and that $E_-(\varphi_t(\rho)){\perp} \re H_p $.} However, since $\mathbf{u}_-$ is not in {$T_{\rho}\SNH$} we have to modify ${\bf w_t}$.}
Consider the vector
\[
{\bf \tilde{w}_t}= a_t\, \mathbf{u} + {\bf e}_-(t),
\]
and note that ${\bf \tilde{w}_t} \in T_{\rho}(\SNH)$
and
\[
d\varphi_t ( {\bf \tilde{w}_t})=  d\varphi_t({\bf w_t})+a_t\,  d\varphi_t (\mathbf{u}_+).
\]

{Let $\delta_1>0$ be so that $1-\delta_1 \tilde\ca {C_{_{\!K}}}>\frac{1}{2}$. We claim that there is $t_{_{\! K}}>0$,  depending only on $(M,p,K)$,} so that for $t>t_{_{\! K}}$,
\begin{equation}\label{e:a_t}
\| {\bf w_t}\| \leq {\delta_1} \qquad \text{and} \qquad   |a_t|<\delta_1 {\|\mathbf{u}_-\|^{-1}}.
\end{equation}
Note that this yields that for $t$ large enough, $d\varphi_t ( {\bf \tilde{w}_t})$ approaches $d\varphi_t({\bf w_t}) \notin T_{\varphi_{t}(\rho)}\mc{F}_{\!{\varphi_{t}}(\rho)}$.  In particular, the $t$-flowout of the ${\bf \tilde{w}_t}$ direction  in $T_{\rho}(\SNH)$ approaches $E_-(\varphi_t(\rho))$ (see Figure \ref{f:rotation}).
We postpone the proof of \eqref{e:a_t} until the end, and show how to finish the proof assuming it holds.
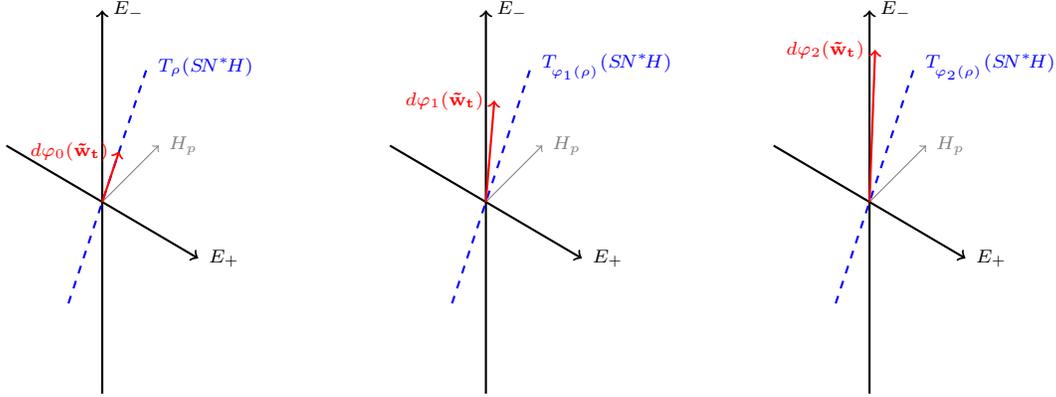
\begin{figure}[h]
\begin{tikzpicture}
\begin{scope}[scale =1.5]
\foreach \t in{0}{
\begin{scope}[shift= {(2.4*\t,0)}]
\draw[thick,->](\t,-1.7)--(\t,1.7)node[right]{\tiny{$E_-$}};
\draw[gray,->](\t,0)--(\t+.5,.5)node[right]{\color{gray}{\tiny{$H_p$}}};
\draw[thick,dashed,blue](\t-.3,-.9)--(\t+.4,1.2)node[right]{\tiny{\color{blue}$T_\rho (\SNH)$}};
\draw[thick,->](\t-.85,.5)--(\t+.85,-.5)node[right]{\tiny{$E_+$}};
\draw[thick,->,red] ({(\t-.75*.2/(\t+1)+\t+.75*.2/(\t+1))/2},{(-.75*.6*(\t+1)+.75*.6*(\t+1))/2}) --({\t+.75*.2/(\t+1)}, {.75*.6*(\t+1)})node[left]{\tiny{\color{red}$d\varphi_{\t}(\bf \tilde{w}_t)$}};
\end{scope}
}

\foreach \t in{1,2}{
\begin{scope}[shift= {(2.4*\t,0)}]
\draw[thick,->](\t,-1.7)--(\t,1.7)node[right]{\tiny{$E_-$}};
\draw[gray,->](\t,0)--(\t+.5,.5)node[right]{\color{gray}{\tiny{$H_p$}}};
\draw[thick,dashed,blue](\t-.3,-.9)--(\t+.4,1.2)node[right]{\tiny{\color{blue}$T_{_{\!\varphi_{\t}(\rho)}} (\SNH)$}};
\draw[thick,->](\t-.85,.5)--(\t+.85,-.5)node[right]{\tiny{$E_+$}};
\draw[thick,->,red] ({(\t-.75*.2/(\t+1)+\t+.75*.2/(\t+1))/2},{(-.75*.6*(\t+1)+.75*.6*(\t+1))/2}) --({\t+.75*.2/(\t+1)}, {.75*.6*(\t+1)})node[left]{\tiny{\color{red}$d\varphi_{\t}(\bf \tilde{w}_t)$}};
\end{scope}
}
\end{scope}
\end{tikzpicture}
\caption{\label{f:rotation} Schematic of the rotation of $\bf \tilde{w}_t$  under the geodesic flow.}
\end{figure}

We next observe that there exists $\tilde\ca>0$ so that  if $\w \in E_{\pm}(\rho)$ then $\|d\varphi_t \w\|\leq \tilde\ca\|\w\|$ as $t\to \pm\infty$. Indeed, in the Anosov case $\tilde\ca=\ca$, where $\ca$ is defined in \eqref{e:Bdef}, and in the no focal point case the existence of $\tilde\ca$ is guaranteed  by~\cite[Proposition 2.13, Corollary 2.14]{Eberlein73}. We can therefore conclude from \eqref{e:size}  and \eqref{e:a_t} that 
\[
\|  \pi_{t,\rho} (d\varphi_t  {\bf \tilde{w}_t}) \| \geq \|  \pi_{t,\rho} (d\varphi_t {\bf w_t})\|-\| a_t\,  \pi_{t,\rho} (d\varphi_t {\bf u}_+) \| >1-\delta_1 \tilde\ca {C_{_{\!K}}},
\]
and 
$$
\|\mathbf{\tilde{w}_t}\|=\|\w_t +a_t {\bf u}_+ \|\leq \|\mathbf{w}_t\|+|a_t|\|\mathbf{u}_+\|\leq \delta_1(1+C_{_{\!K}}),
$$
{where $\pi_{t,\rho}$ denotes orthogonal projection onto $E_-(\varphi_t(\rho))\cap (\tilde{N}_-(\varphi_t(\rho)))^\perp.$}
In particular,
\[
\|  \pi_{t,\rho} (d\varphi_t  {\bf \tilde{w}_t}) \| \geq   \frac{1-\delta_1 \tilde\ca {C_{_{\!K}}}}{\delta_1(1+C_{_{\!K}})} \|\mathbf{\tilde{w}_t}\|.
\]
Therefore,  there exist positive constants $c_{_{\! K}}$, $\delta_{_{\! K}}$ and $t_{_{\! K}}$ (uniform for $\rho\in K$) so that  if ${\varphi_{t}}(\rho)\in B(\rho, \delta_{_{\! K}})$ for some $t$ with $|t|>t_{_{\! K}}$, {then there is $\mathbf{w}=\tilde{\mathbf{w}}_{t}\in T_{\rho}(\SNH)$} so that 
\begin{equation}
\label{e:noTangent}
\| d{\varphi_{t}} (\mathbf{w})+\re H_p+T_{{\varphi_{t}}(\rho)}\mc{F}_{\!{\varphi_{t}}(\rho)}\|\geq c_{_{\! K}} \|\mathbf{w}\|.
\end{equation}
This would finish the proof assuming that  the claim in \eqref{e:a_t} holds. 

We proceed to prove \eqref{e:a_t}. We start with the Anosov case.
By the definition of Anosov geodesic flow, 
\[
\| (d\varphi_t|_{E_-})^{-1}\|\leq \ca e^{-t/\ca},\quad t\geq 0.
\]
{Thus, since ${\bf w_t}\in E_-(\rho)$ and  $\left \|   d\varphi_t {\bf w_t} \right\|=1$, we find {$\|{\bf w_t}\|\leq \ca e^{-t/\ca}$}. In particular, since $\mathbf{u}_-$ and $ {\bf e}_-(t) $ are orthogonal,  we have
\[
|a_t|\leq \ca e^{-t/\ca}\|\mathbf{u}_-\|^{-1},\qquad { t\geq 0}.
\]}
This proves the claim  \eqref{e:a_t} in the Anosov flow case after choosing $t_{_{\! K}}>0$ large enough so that  $\ca e^{-t/\ca}\leq \delta_1$.

We next consider the non-focal points case. Define {$\mathcal C_+^\alpha(\rho)\subset T_\rho (S^*\!M)$} to be the conic set of vectors forming an angle larger than or equal to $\alpha>0$ with $E_+(\rho)$. {Let $\alpha_{_{\! K}}>0$ be so that  ${\mathbf{w}_t}\in E_-(\rho)\cap \mathcal C_+^{\alpha_{_{\! K}}}(\rho)$ for all $\rho \in \tilde{K}$}. {By~\cite[Proposition 2.6]{Eberlein73} $(d\pi)_\rho:E_{\pm}(\rho)\oplus H_p(\rho)\to T_{\pi(\rho)}M$ is an isomorphism for each $\rho$. In particular, letting $V(\rho)\subset T_\rho (\SM)$ denote the vertical  vectors, we have that $E_{\pm}(\rho)\cap V(\rho)=\emptyset$ and $V(\rho)\oplus E_+(\rho)\oplus H_p(\rho)=T_{\pi(\rho)}S^*M$. In addition, since $(M,g)$ has no focal points, $\cup_{\rho\in S^*M} E_{\pm}(\rho)$ is closed~\cite[see right before Proposition 2.7]{Eberlein73} and hence there exists $c_{_{ \alpha_{_{\! K}}}}>0$ depending only on $\alpha_{_{\! K}}$ so that 
\[
\mathbf{w}_t={\bf e_+}+{\bf v}
\]
with
\[ 
c_{_{\alpha_{_{\! K}}}}\|{\bf e_+}\|\leq \|\mathbf{w}_t\| \leq  \frac{1}{c_{_{\alpha_{_{\! K}}}}} \|{\bf v}\|.
 \]}
 and ${\bf e_+}\in E_+(\rho)$, ${\bf v}\in {V}(\rho)$. By~\cite[Remark 2.10]{Eberlein73}, for all $R>0$ there exists $T(R)>0$ so that $\|Y(t)\| \geq R \|Y'(0)\|$ for all $t>T(R)$, where $Y(t)$ is any Jacobi field with $Y(0)=0$ and perpendicular to a unit speed geodesic $\gamma$  with $\gamma(0) \in \tilde{K}$.  Since ${\bf v}$ is a vertical vector, we may consider $Y(t)=d\pi \circ d\varphi_t ({\bf v})$, and this implies that $Y'(0)=\bf{K}{\bf v}^\sharp$ ({see Appendix~\ref{s:jacobi} for an explanation of  the connection map $\bf{K}$, and the $\sharp$ operator}). We therefore  have that $\|d\varphi_t {\bf v}\|\geq R\|{\bf v}\|$ for all $t>T(R)$. In particular, then 
$$
\|d\varphi_t\mathbf{w}_t\|=\|d\varphi_t\mathbf{v} + d\varphi_t\mathbf{e}_+ \|  \geq R\|{\bf v}\|-\tilde \ca\|{\bf e_+}\|\geq (Rc_{_{ \alpha_{_{\! K}}}}-c_{_{ \alpha_{_{\! K}}}}^{-1}\tilde \ca)\|\mathbf{w}_t\|.
$$
So, choosing $R(\alpha_{_{\! K}})=c_{_{ \alpha_{_{\! K}}}}^{-1}(\delta_1^{-1}+c_{_{ \alpha_{_{\! K}}}}^{-1}\tilde \ca)$, we have that for $t\geq t_{_{\! K}}:=T(R(\alpha_{_{\! K}}))$, 
$$
1=\|d\varphi_t\mathbf{w}_t\|\geq \delta_1^{-1}\|\mathbf{w}_t\|.
$$
 In particular, for $t\geq t_{_{\! K}}$,  since $\mathbf{u}_-$ is orthogonal to ${\bf e_-}(t)$, we obtain 
$
1=\|d\varphi_t\mathbf{w}_t\|\geq \delta_1^{-1}\|\mathbf{w}_t\|\geq \delta_1^{-1}|a_t|  {\|\mathbf{u}_-\|},
$
completing the proof of the lemma in the case of manifolds without focal points.\\
\end{proof}

{When $(M,g)$ has Anosov geodesic flow, we need to define a notion of angle between a vector and $E_{\pm}(\rho)$.} 
Let  $\pi_{\pm}:T_\rho S^*\!M\to E_{\pm}(\rho)$ be the projection onto $E_{\pm}(\rho)$ along $E_{\mp}(\rho)\oplus H_p(\rho)$ {i.e. if ${\bf{u}}={\bf{v}}_++{\bf{v}}_-+rH_p$ with $r\in \re$, ${\bf{v}}_{\pm}\in E_{\pm}(\rho)$, then $\pi_{\pm}({\bf{u}})={\bf{v}}_{\pm}$.}
For $\rho \in S^*\!M$, define ${\Theta}^{\pm}_\rho:{ (\re H_p(\rho))^\perp}\setminus \{0\}\to [0,\infty]$ by
\begin{equation}\label{e:Theta}
{{\Theta}^{\pm}_\rho}(\mathbf{u}):=\frac{\|\pi_\mp\mathbf{u}\|}{\|\pi_\pm\mathbf{u}\|}.
\end{equation}
Note that ${\Theta}^{\pm}_\rho$ should be thought of as measuring the tangent of the angle from $E_{\pm}(\rho)$, and that {given a compact subset $K$ of $\TM \backslash\{0\}$ there exists $C_{_{\!K}}>0$ so that for all $\rho \in K$,} $t\in \re$, and $\mathbf{u} \in T_\rho S^*\!M$, we have
\begin{equation}
\label{e:angleChange}
\frac{e^{\pm t/C_{_{\!K}}}}{C_{_{\!K}}}\,{\Theta}^{\pm}_\rho(\mathbf{u})\leq {\Theta}^{\pm}_\rho(d\varphi_t\mathbf{u})\leq C_{_{\!K}}e^{\pm C_{_{\!K}}t}\,{\Theta}^{\pm}_\rho(\mathbf{u}).
\end{equation}

{{In what follows we will use the fact that by~\cite[Proposition 3.3]{CG18d} there are {$\mathfrak{D}_n>0$ depending only on $n$}, $\tau_{_{\!\SNH}}>0$ depending only on $\tau_{_{\!\inj H}}$, {and} ${R_0>0}$ depending only on $(n,k,K_H)$ and finitely many derivatives of the curvature and second fundamental form of $H$, so that for ${0<}\tau<\tau_{_{\!\SigH}}$ and ${0<}r<R_0$, there is a $({\mathfrak{D}}_n,\tau,r)$ good cover of $\SNH$.}}
\begin{lemma}\label{l:anosov}
Let $(M,g)$ have Anosov geodesic flow and $H\subset M$ satisfy $\mc{A}_H=\emptyset$. Then, there exist{ $c=c(M,g,H)>0$, $C=C(M,g,H)>2$}, ${I}>0$, $t_0>1$, so that  for all  {$\Lambda>\Lambda_{\text{max}}$} the following holds. 

Let  $T_0\geq t_0$,\; $m=\big\lfloor \frac{\log T_0 - \log t_0} {\log 2}\big\rfloor$, \; $0<\tau_0<\tau_{_{\!\SigH}}$,\; $0<\tau\le \tau_0$, 
 \[0\leq r_1\leq \min\{e^{-C T_0},{R_0}\}, \]
  and $\{\Lambda_{\rho_j}^\tau(r_1)\}_{j=1}^{N_{r_1}}$ be a $({\mathfrak{D}_n},\tau,r_1)$ good cover of $\SNH$. 
  Then, for each  $i\in\{1,\dots, {I}\}$ there are sets of indices $\{\mc{G}_{i,\ell}\}_{\ell=0}^m\subset \{1,\dots, N_{r_1}\}$ and $\mc{B}\subset \{1,\dots, N_{r_1}\}$ so that 
$$\bigcup_{i=1}^{{I}}\bigcup_{\ell=0}^m\mc{G}_{i,\ell}\cup \mc{B}= \{1,\dots, N_{r_1}\},$$
and for every $i\in\{1,\dots, {{I}}\}$ and every $\ell\in \{0, \dots, m\}$
\begin{itemize}
\item $\bigcup_{j\in \mc{G}_{i,\ell}}\Lambda_{\rho_j}^\tau(r_1)$ is $[t_0,2^{-\ell}T_0]$ non-self looping, \\
\item $ |\mc{G}_{i,\ell}|\leq c \, 5^{-\ell}\, r_1^{1-n},$\\
\item $ |\mc{B}|\leq c \,{\Big(\frac{t_0}{T_0}\Big)^{\frac{\log 5}{\log 2}}}\, r_1^{1-n}.$
\end{itemize}
\end{lemma}

{We note that if $H_0 \subset M$ is an embedded submanifold, there exists a neighborhood $U$ of $H_0$ (in the $C^\infty$ topology) so that the constants $c=c(M,p,H)$ and $C=C(M,p,H)$ in Lemma \ref{l:anosov} are uniform for  $H\in U$.}


\begin{proof}
{Let
$0\leq r_0\leq \tfrac{1}{C}e^{-\Lambda T_0}r_1$. Then $\{\Lambda_{\rho_j}^\tau(r_1)\}_{j=1}^{N_{r_1}}$ covers $\Lambda_{_{\!\SNH}}^\tau(r_0)$ since $r_0 \leq \tfrac{1}{2}r_1$.} Throughout this proof we will repeatedly use that if  $F:\TM \to \re^{n+1}$ is the defining function for $\SNH$, then there exist $\delta_0,c_0>0$ so that for $q\in \TM $
\begin{equation}\label{e:F}
d(q, \SNH)\leq \delta_0 \quad \Longrightarrow \quad \|dF {\bf v}\|\geq c_0 \inf\big \{\|{\bf v}+{\bf u}\|:\; {\bf u} \in T_{q}\mathcal F_{q}\big\}\quad \forall {\bf v}\in T_{q}(\TM ).
\end{equation}
 In addition, {let $\nu>0$ be so that {$\rho \mapsto E_{\pm}(\rho)\in C^\nu$} and}  define $c_{_{\!H}}>0$ so that
\begin{equation}\label{e:angle}
\sup_{{q_1,q_2} \in \SNH} \Big(\| \tan^{-1}\circ\Theta^{{\pm}}_{{q_1}}\|_{L^\infty(T_{{q_1}}\SNH)}-\| \tan^{-1}\circ\Theta^{{\pm}}_{{q_2}}\|_{L^\infty(T_{{q_2}}\SNH)}\Big)\leq \frac{1}{c_{_{\!H}}}{d(q_1,q_2)^\nu}.
 \end{equation}
{This implies that} that for all $\e>0$, there exists $\delta_\e>0$ so that  for every  ball  $\tilde B \subset \SNH$ of radius $\delta_\e$ we have that 
\begin{equation}\label{e:Abounds}
\sup_{\rho_1,\rho_2\in \tilde B} \;\Big|   
\|\tan^{-1}{\Theta}^{\pm}_{\rho_1}\|_{L^\infty(T_{\rho_1}\SNH)} -
\| \tan^{-1}{\Theta}^{\pm}_{\rho_2}\|_{L^\infty(T_{\rho_2}\SNH)}
 \Big|<\e.
\end{equation}

Also,  since $\mc{A}_H=\emptyset$, we know that for every $\rho \in \mc{S}_H$ we must have that either $m_+(\rho)=0$ or $m_-(\rho)=0$, where we continue to write $m_{{\pm}}(\rho) =\dim N_\pm (\rho)$. 
Therefore, choosing 
\begin{equation}\label{e:epsilon}
\e=\e(M,p,H)<1
\end{equation}
small enough, depending only on $(M,g,H)$, and shrinking $\delta_\e$ if necessary, we may also assume  that
if $ \tilde B\cap \mc{S}_H\neq \emptyset$ then either
\begin{align}
m_-(\rho)&= 0 \;\; \text{and}  \;\; {\Theta}^+_\rho\leq \e \;\;\; \text{for all} \;\;\rho \in \tilde B, \notag\\
&\qquad \qquad \qquad\text{or}\label{e:dimensions}\\
m_+(\rho)&= 0 \;\; \text{and}  \;\;{\Theta}^-_\rho\leq \e \;\;\; \text{for all}\;\; \rho \in \tilde B.\notag
\end{align}
Furthermore, we assume that $\delta_\e \leq \tfrac{2}{9}\big[\e c_{_{\!H}}\big]^{\frac{1}{\nu}}$.

Next, let  $\{B_i\}_{i =1}^{N_\e} \subset \SNH$  be a cover of $\SNH$ with
\[
\SNH \subset \bigcup_{i =1}^{N_\e} B_i, \qquad  \qquad B_i \;\;\text{ball of radius}\; \tfrac{1}{2}\delta_\e.
\]
Let  $\mc{I}_{ \mc{S}_H}:=\{i \in \{1, \dots, N_\e\}: \; B_i \cap \mc{S}_H \neq \emptyset\}$, and define $K=K_\e$ by
\[
K:= \bigcup_{i\in \mc{I}_{\mc{S}_H} } (\SNH \backslash {B_i}).
\]
Since  $K\subset (\SNH \backslash  \mc{S}_H)$ is compact and the geodesic flow is Anosov, by Lemma~\ref{P:1}  there exist positive constants $c_{_{\!K}},t_{_{\!K}},\delta_{_{\!K}}$ so that {$d(K,\mc{S}_H)>\delta_{_{\! K}}$} and, if {$d(\rho ,K)\leq \delta_{_{\!K}}$} and $\varphi_t(\rho) \in \overline{ B(\rho, \delta_{_{\!K}})}$ for some $|t|>t_{_{\!K}}$, then there exists $\w=\w(t,\rho) \in T_{\rho}(\SNH)$ so that 
\begin{equation}\label{e:lb0}
\inf\{\|d\varphi_{t} (\mathbf{\w})+{\bf v}\|:\; {\bf v} \in  {T_{\varphi_{t}(\rho)}}\mc{F}_{\!\varphi_{t}(\rho)} + \R H_p\}\geq  c_{_{\!K}}\|\mathbf{\w}\|.
\end{equation}
We then introduce a cover   $\{D_i\}_{i \in I_K} \subset \SNH$   of $K$ by balls with
\[
K \subset \bigcup_{i \in I_K} D_i, \qquad \qquad D_i \;\;\text{ball of radius}\;{ \tfrac{1}{4}R,}
\]
where  
\[R:=\min\{\delta_{_{\!K}}, \delta_0, \tfrac{1}{2}\delta_\e,{\delta_F}\}\]
and $\delta_F$ is as in~\eqref{e:defFunction}.
Note that $R$ depends only on $(M,p,H,K)$.
It follows that, 
\begin{equation}\label{e:snh}
\SNH \subset  \left( \bigcup_{i \in \mc{I}_{ \mc{S}_H}} B_i \;\; \cup  \;\; \bigcup_{i \in \mc{I}_K} D_i \right)
\end{equation}
where each ball $B_i$ satisfies \eqref{e:Abounds} and \eqref{e:dimensions}, and each ball $D_i$ satisfies \eqref{e:lb0}. Also, 
\[ 
\mc{S}_H\cap D_i =\emptyset \;\;\; \forall  i \in \mc{I}_K \qquad  \text{and}\qquad   \mc{S}_H\cap  B_i \neq \emptyset \;\;\; \forall i \in \mc{I}_{ \mc{S}_H}. 
\]

Since $\SNH$ can be split as in \eqref{e:snh}, we present how to treat $D_i$ with $i \in \mc{S}_H$ and $B_i$ with $i \in \mc{I}_K$ separately.
\begin{center}
{\underline{\bf Treatment of $D\in \{D_i\}_{i \in \mc{I}_{K}}$.}}
\end{center}

  Let  $D\in \{D_i\}_{i \in \mc{I}_{K}}$. Note that since  $R\leq \min\{\delta_{_{\!K}}, \delta_0\}$, by \eqref{e:lb0} we know that if $\rho \in D$ and $|t| \geq t_{_{\!K}}$ are so that  $d(\varphi_t(\rho),\rho)< R$, then there exists $\w=\w(t, \rho)\in T_{\rho}(\SNH)$ so that for all $s\in \R$
\begin{align*}
\|dF(d\varphi_t\w+s H_p)\| 
&\geq  c_0 \inf\big \{\|d\varphi_t\w + s H_p+ {\bf u}\|:\;  {\bf u} \in T_{\varphi_t(\rho)}\mathcal F_{\varphi_t(\rho)} \big\}\\
&\geq  c_0 \inf\big \{\|d\varphi_t\w +  {\bf v}\|:\;  {\bf v} \in T_{\varphi_t(\rho)}\mathcal F_{\varphi_t(\rho)} + \R H_p\big\}\\
& \geq  c_0  c_{_{\!K}}\|\w\|,
\end{align*}
 where we used \eqref{e:F} to get the first inequality and \eqref{e:lb0} for the third one. This implies that if $|t|\geq t_{_{\!K}}$  and $\rho \in  D$ are so that $d(\varphi_t(\rho),\rho)< R$, then $d\psi(t,\rho):=d(F \circ \varphi_t)(t,\rho)$ has a left inverse $L_{(t, \rho)}$ when restricted to $\R\partial_t \oplus \R \w$ with $\|L_{(t, \rho)}\| \leq (c_0  c_{_{\!K}})^{-1}$.

Let $\alpha_1, \alpha_2$ be as in Proposition~\ref{p:ballCover}, and note that they only depend on $(M,g, H,{K})$. We aim to apply this proposition with  $A=D$, $B=D$, $\beta=0$, $c={(c_0  c_{_{\!K}})^{-1}}$, {$a=0$, $\tilde{c}=\frac{R}{4}$}. Let $t_1$ satisfy
\begin{equation}\label{e:t0def}
t_1\geq \max\{1,t_{_{\!K}}\}.
\end{equation}
 Note that  $t_1$ depends only on $(M,p,H,{K})$.
 
Next, let $T_0 \geq t_1$. By construction, if $(t,\rho) \in  [t_1, T_0] \times D $ are so that $d(\varphi_t(\rho),D) \leq  \tilde c $,  by \eqref{e:t0def} we have  
\[
d(\varphi_t(\rho),\rho) \leq  d(\varphi_t(\rho),D)+\diam (D) \leq \tilde c+ 2 (\tfrac{1}{4}R)< R.
\]
In this case there exists $\w=\w(t, \rho)\in T_{\rho}(\SNH)$ so that $d\psi(t,\rho)$ has a left inverse $L_{(t, \rho)}$ when restricted to $\R\partial_t \oplus \R \w$ with  $\|L_{(t, \rho)}\| \leq c_0  c_{_{\!K}}\leq c $.

Let $C>0$ be so that 
\begin{equation}\label{e:Cdef}
\frac{1}{C}<{\min}\{\tfrac{1}{2}, \tfrac{1}{3\alpha_1}\} \qquad \text{and} \qquad e^{-CT_0}\leq \min\{\tfrac{1}{8}\alpha_1R,\; \tfrac{1}{2}\alpha_1\alpha_2e^{-{3\Lambda} T_0}\}.
\end{equation} 
Set $r_2:=\frac{2}{\alpha_1}r_1$ and note that by construction, and the assumptions on the pair $(r_0,r_1)$, we have
 $$
 r_1< \alpha_1\, r_2, \qquad r_2 \leq \min\{{\tfrac{1}{4}R},\alpha_2\, e^{-{3\Lambda} T_0}\}, \qquad r_0 < \tfrac{1}{3}\, e^{-\Lambda T_0} r_2.
$$
Also, note that we work with $0<\tau<\tau_0<\tau_{_{\!\SigH}}$, and that by definition $\tau_{_{\!\SigH}}<\tfrac{1}{2}\Tinj$ as requested by Proposition~\ref{p:ballCover}.
We apply Proposition~\ref{p:ballCover} {to the cover $\{\Lambda_{\rho_j}^\tau(r_1)\}_{j \in \mathcal E_{_{\!D}}}$ of $\Lambda_{_{\!D}}^\tau(r_0)$ where}
\begin{equation}
\label{e:ei1}
\mc{E}_{_{\!D}}:=\{j: \Lambda_{\rho_j}^\tau(r_1)\cap \Lambda_{_{\!{D}}}^\tau(r_0)\neq \emptyset\}.
\end{equation}
Then, there is a partition  $\mc{E}_{_{\!D}}=\mc{G}_{_{\!D}}\cup \mc{B}_{_{\!D}}$ with 
\begin{equation}\label{e:badD}
|\mc{B}_{_{\!D}}|\leq {\bf{C}_0} \;\frac{{R}^{n-1}}{{r_1^{n-2}}}\;  T_0e^{{4\Lambda}T_0},
\end{equation}
where {${\bf{C}_0}={\bf{C}_0}(M,g,k,c_0,c_{_{K}})>0$},
and so that
\begin{equation}\label{e:goodD}
\bigcup_{j\in \mc{G}_{_{\!D}}}\Lambda^\tau_{\rho_j}(r_1)\qquad \text{is}\;\;\; [t_1,T_0]\text{ non-self looping}.
\end{equation}

\begin{center}
{\underline{\bf Treatment of $B\in \{B_i\}_{i \in \mc{I}_{\mc{S}_H}}$}}
\end{center}
 Let $B\in \{B_i\}_{i \in \mc{I}_{\mc{S}_H}}$. Since \eqref{e:dimensions} is satisfied for all $\rho \in B$, we shall focus on the case where $m_-(\rho)=0$ for all $\rho \in B$; the other being similar after sending $t\mapsto -t$ in the arguments below. 
 
 Suppose $B$ is the ball $B(\rho_{_{\!B}}, \tfrac{1}{2}\delta_\e)$ for some $\rho_{_{\!B}}\in \SNH$ and let 
 $$
 E:=B(\rho_{_{\!B}}, \tfrac{3}{4}\delta_\e)\subset \SNH, \qquad \tilde B:=B(\rho_{_{\!B}}, \delta_\e)\subset \SNH.
 $$
Note that $B  \subset E \subset \tilde B$, and that   ${\Theta}^+_\rho\leq \e$ for all $\rho \in \tilde B$ by \eqref{e:dimensions}.

 We claim that there exist a function $\ti_{2}:[\tfrac{1}{5}, +\infty) \to [1, +\infty)$ that depends only on $(M,p)$,  and a constant $\digamma>0$ {depending on $(M,p,K_{_{\!H}})$},  so that 
 \begin{equation}\label{e:Bcontrol}
 E  \;\; \text{can be}\;\; (\tfrac{1}{5}, \ti_{2}, \digamma)\text{-controlled up to time} \;T_0.
 \end{equation}
If the claim in \eqref{e:Bcontrol} holds, setting $R_0:= \min\{{\tfrac{1}{\digamma}}e^{-{\digamma}\Lambda T_0}, \tfrac{1}{8}\delta_\e\}$  and noting that $d(B, E^c)=\tfrac{1}{4}\delta_\e > R_0$,  we may apply Lemma~\ref{l:nonlooping2} to the ball $E$ with $E_0=B$ and $\e_0=\tfrac{1}{5}$.
Indeed, by possibly enlarging $C>0$ in \eqref{e:Cdef} so that 
\begin{equation}\label{e:Cdef3}
e^{-CT_0}<\tfrac{1}{5\digamma}e^{-({\digamma}+2\Decay)\Lambda T_0} R_0,
\end{equation}
by the assumption that $r_1\leq e^{-CT_0}$ we conclude $0<r_1<{\tfrac{1}{5\digamma}}e^{-({\digamma}+2\Decay)\Lambda T_0} R_0$. 
Therefore, letting
\begin{equation}
\label{e:ei2}
\mc{E}_{_{\!B}}:=\{j: \Lambda_{\rho_j}^\tau(r_1)\cap \Lambda_{_{\!B}}^\tau(r_0)\neq \emptyset\},
\end{equation}
 there exists $C_{_{M,g}}>0$ depending only on $(M,g)$, so that for every integer $0<m<\frac{\log T_0-\log t_0(\frac{1}{5})}{\log 2}$ there are sets $\{\mc{G}_{_{\!B,\ell}}\}_{\ell =0}^m\subset \{1,\dots N_{r_1}\}$, $\mc{B}_{_{\!B}}\subset \{1,\dots N_{r_1}\}$ satisfying
\begin{align}\label{e:Bballs}
 \mc{E}_{_{\!B}}\subset \mc{B}_{_{\!B}}\cup \displaystyle\bigcup_{\ell =0}^m\mc{G}_{_{\!B,\ell}},\qquad 
 \bigcup_{i\in \mc{G}_{_{\!B,\ell}}}\Lambda_{\rho_i}^\tau(r_1)\text{ is } [t_{2}(\tfrac{1}{5}),2^{-\ell}T_0]\text{ non-self looping}  \notag\\
 |\mc{G}_{_{\!B,\ell}}|\leq C_{_{M,p}}\frac{\delta_\e^{n-1}}{5^{\ell}} \frac{1}{r_1^{n-1} },\qquad \text{ and }
 \qquad|\mc{B}_{_{\!B}}|\leq C_{_{M,p}} \frac{\delta_\e^{n-1}}{5^{m+1}} \frac{1}{r_1^{n-1}},
 \end{align}
 for all $\ell \in \{0, \dots, m\}$.
We shall use this construction later in the proof, namely below the ``Constructing the complete cover" title, to build the complete cover.\\

We dedicate the rest of the argument to proving the claim in \eqref{e:Bcontrol}. 
Let $\digamma>0$ satisfy
\begin{equation}\label{e:digamma}
\frac{1}{\digamma} < \min \Big  \{ \frac{\alpha}{4} , \frac{\alpha^2}{4} , \frac{\alpha\,  }{{60\bf{C}_0}}, \frac{[\e c_{_{\!\!H}}]^{\frac{1}{\nu}}}{3} , {\frac{\e^{{\frac{1}{\nu}}}}{C^{{\frac{1}{\nu}}}_{_{\!\Theta}}}} , \frac{1}{11},{\frac{\nu}{2}}    \Big\},
\end{equation}
where $\alpha:=\min\{{\frac{1}{3}},\alpha_1, \alpha_2\}$, 
$c_{_{\!\!H}}$ is defined in \eqref{e:angle},   {${\bf{C}_0}$ } is the positive constant  introduced in Proposition~\ref{p:ballCover} (that depends only on $(M,g,H,\e)$ when the left inverse is bounded by  ${{\frac{2C_\varphi}{c_0\, \e }}}$), and {$C_{_{\!\Theta}}$ is so that {for all $\rho_1,\rho_2\in\SNH$}
\begin{equation}
\label{e:thetaChange}
\sup_{\substack{{\w_1}\in T_{{\rho_1}}\SNH\\ \Theta^+(\w_{{1}})\leq \e}}
\inf_{\substack{{\w_2}\in T_{{\rho_2}}\SNH\\ \Theta^+(\w_{{2}})\leq \e}}
|\Theta^+_{\varphi_t(\rho_1)}(d\varphi_t)_{\rho_1}\w_1 
-\Theta^+_{\varphi_t(\rho_2)}(d\varphi_t)_{\rho_2}\w_2\| 
\leq C_{_{\!\Theta}}d(\rho_1,\rho_2)^{\nu} e^{2\Lambda |t|}
\end{equation}
 for all $t \in \re$}.
Next, Let $0<\tau<\tau_0$,  $\e_1\geq \tfrac{1}{5}$,
\[
0< \tilde R_0\leq   \tfrac{1}{ \digamma}e^{- \digamma \Lambda T_0}\qquad  \text{and}\qquad 0<\tilde r_0<\tilde R_0.
\]
 Also, let $\{B_{0,i}\}_{i=1}^N \subset \SNH$ be a collection of balls with centers in $E$ and radii $R_{0,i}=\tilde R_0\geq 0$ so that 
\[
E \subset  \bigcup_{i=1}^NB_{0,i} \subset \tilde B.
\]

Using \eqref{e:angleChange} we let $L\geq 1$  be so that for all $q \in \SNH$ and all ${\bf u} \in T_\rho S^*\!M\backslash \{0\}$ we have 
${{\Theta}^+_{\varphi_s(q)}(d\varphi_s {\bf u} )}\geq \frac{1}{L}{\Theta}^+_q({\bf u} )$ provided  $ s\geq 0.$
Next, for each $i \in \{1, \dots, N\}$  let 
\[T_{_{\! B_{0,i}}}:= \inf_{\rho \in B_{0,i}}  T(\rho) \qquad \text{for}\qquad 
T(\rho):=\inf \big\{t\geq 0: \sup_{\w\in T_{\rho}\SNH} {\Theta}^+_\rho(d\varphi_t\w) > 5L\e\big\},\]
where $\e=\e(M,g,H)$ as defined in \eqref{e:epsilon}.
Note that since ${\Theta}^+_\rho\leq \e$ for  $\rho \in \tilde B$, then $T_{_{\! B_{0,i}}}>0$ for all $i \in \{1, \dots, N\}$. \\

{\noindent \bf Control of $ B_{0,i}$ before time $T_{_{\! B_{0,i}}}$.} 
We claim that  for all $\rho \in B_{0,i}$ and $\w \in T_\rho S^*\!M$
\begin{equation}\label{e:expDecay}
\|d\varphi_t \w\|\leq \ca(1+5L\e)e^{-t/\ca}\|\w\|\quad\qquad  0\leq t< T_{_{\! B_{0,i}}}.
\end{equation}
Indeed, suppose that $0\leq t< T(\rho)$ for some $\rho \in B_{0,i}$. Then, ${\Theta}^+_{\varphi_t(\rho)}(d\varphi_{t}\mathbf{w})\leq 5L\e$ for all $\w \in T_\rho\SNH$ and so, using that $\pi_\pm d\varphi_t=d\varphi_t\pi_{\pm},$
we have
\[
\|d\varphi_{t}{\mathbf{w}}\|\leq \|d\varphi_t\pi_+{\mathbf{w}}\|+\|d\varphi_t\pi_-{\mathbf{w}}\|
\leq (1+5L\e) \|d\varphi_{t}\pi_+{\mathbf{w}}\|
\leq (1+5L\e)\ca e^{- t/\ca}\|{\mathbf{w}}\|.
\]
From \eqref{e:expDecay} it follows that there exists $C_0>0$, depending only on $(M,g,H)$, so that  
$$
\sup_{\rho \in B_{0,i}}|\det J_t|\leq C_0\, e^{-|t|/C_0} \qquad \text{ for all}\;\; t \in (0, T_{_{\! B_{0,i}}}).
$$

Suppose that  $T_{_{\! B_{0,i}}} > 1$. By Lemma~\ref{l:nonlooping}, for all $\ep_0>0$ there exists $\digamma_{_{\!\!M,g,K_{_{\!H}}}}>0$ and a function $\ti_0:[\e_0, +\infty) \to [1, +\infty)$ depending only on $(M,g,H,\e_0, C_0)$ so that   the set $B_{0,i}$ can be $(\e_0, \ti_0, \digamma_{_{\!\!\!M,p}})$-controlled up to time $T_{_{\! B_{0,i}}}$ in the sense of Definition \ref{d:control}. In addition, by Lemma~\ref{l:nonlooping},  given  $\e_1 >0$  and any $0<r \leq \tfrac{1}{ \digamma} e^{- \digamma \Lambda T_0} \tilde r_0 $,   there exist  balls $\{\tilde B_{1,k}\}_{k}\subset \SNH$ with  radii $R_{1,k}\in [0, \tfrac{1}{4}\tilde R_0]$  so that 
 \begin{equation}\label{e:noloop}
 {\bigcup_{t=\ti_0(\tfrac{1}{5})}^{T_{_{\! B_{0,i}}}} \varphi_t(\Lambda^\tau_{B_{0,i} \backslash\cup_k\tilde{B}_{1,k}}(r))\;\bigcap \;\Lambda_{{\SNH}\backslash\cup_k\tilde{B}_{1,k}}^\tau(r)=\emptyset,}
 \end{equation}
 \begin{equation}\label{e:noloop-radius}
  \sum_k \tilde R_{1,k}^{n-1}\leq  \frac{\e_1}{2} \tilde R_0^{n-1} \qquad \text{and}\qquad {\inf_k}\tilde R_{1,k} \geq e^{-\Decay\Lambda T_0}\tilde R_{0}.
 \end{equation}
 
 In the case in which $T_{_{\! B_{0,i}}} \leq  1$ we will not attempt to control $B_{0,i}$ for times smaller than $T_{_{\! B_{0,i}}}$. Indeed, we will set $t_0=1$, interpret \eqref{e:noloop} and \eqref{e:noloop-radius} as empty statements, and define every ball $\tilde B_{1,k}$ as the empty set.
 
 We now set $\e_0=\frac{1}{10}$ so that $\e_1\geq \frac{1}{5}$.\\
 
 {\noindent \bf Control of $ B_{0,i}$ after time $T_{_{\! B_{0,i}}}$.} Set $A:=\bigcup_{i=1}^N B_{0,i}$. 
Next, suppose that  $\rho \in B_{0,i}$ and $t\geq T_{_{\! B_{0,i}}}$ are so that $d(\varphi_t(\rho),A)\leq \tilde c \,e^{-2\Lambda|t|}$ where
\[\tilde c:=\min\Big\{\tfrac{1}{3}\big[\e c_{_{\!H}}\big]^{\frac{1}{\nu}},  \delta_0, \delta_F \Big\},\]
 with $\delta_F$ defined in~\eqref{e:defFunction},  $\delta_0$ defined in \eqref{e:F}, and $c_{_{\!H}}$ defined in \eqref{e:angle}.     

Since by \eqref{e:digamma} the parameter $\digamma$ is chosen so that $\tfrac{1}{\digamma} \leq \min\{ {\frac{\e^{\frac{1}{\nu}}}{C_{_{\!\Theta}}^{\frac{1}{\nu}}}},\frac{1}{11}\}$ and $\tilde{R}_0<\frac{1}{\digamma}e^{-\digamma\Lambda T_0}$, we have  {$\tilde R_0 \leq  {\frac{\e^{{\frac{1}{\nu}}}}{C_{_{\!\Theta}}^{{\frac{1}{\nu}}}}}\, e^{-{\frac{2}{\nu}} \Lambda T_0}.$} Thus, {using~\eqref{e:thetaChange}}, {$L\geq {1}$}, and that $\rho \in B_{0,i}$, there exists $\w \in T_\rho \SNH$ for which 
\[
{\Theta}^+_{\varphi_{_{\!T_{_{\! B_{0,i}}}}}\!\!(\rho)}(d\varphi_{_{\!T_{_{\! B_{0,i}}}}} \!\!\w) \geq 4L\e.
\]
It then follows by the definition of $L$ that, if $t=T_{_{\! B_{0,i}}} +s$ for some $s> 0$, then
$
 {\Theta}^+_{\varphi_t(\rho)}(d\varphi_t \w)
 = {\Theta}^+_{\varphi_s( \varphi_{_{\!T_{_{\! B_{0,i}}}}}( \rho))}(d\varphi_s (d\varphi_{_{\!T_{_{\! B_{0,i}}}}} \!\!\w))
\geq  \tfrac{1}{L}{\Theta}^+_{\varphi_{_{\!T_{_{\! B_{0,i}}}}}\!\!(\rho)}(d\varphi_{_{\!T_{_{\! B_{0,i}}}}} \!\!\w) \geq  4\e.
$
In particular, 
\begin{equation}\label{e:theta1}
{\Theta}^+_{\varphi_t(\rho)}(d\varphi_t \w +r H_p) \geq 4\e \qquad \text{for all}\;\; r \in \R.
\end{equation}

In addition, we note that 
\begin{equation}\label{e:theta2}
{\Theta}^+_{\varphi_t(\rho)}({\bf v}) \leq 2\e\qquad \text{for all}\;\;{\bf v} \in T_{\varphi_t(\rho)} \mathcal F_{\varphi_t(\rho)}.
\end{equation}
Indeed, this follows from the estimate in \eqref{e:angle} together with the facts that  $\Theta^+_\rho\leq \e$,  $B_{0,i}$ is a ball with radius $\tilde R_0$ and center in $E$, and
\[
d(\varphi_t(\rho),\rho)\leq d(\varphi_t(\rho),A) + \text{diam}(E) +\tilde R_0  \leq  \tilde c \,e^{-2\Lambda|t|} + 2(\tfrac{3}{4}) \delta_\e +\tfrac{1}{\digamma}  \leq [\e c_{_{\!H}}]^{\frac{1}{\nu}}.
\]
We have also used that $\tilde c \leq \tfrac{1}{3} [\e c_{_{\!H}}]^{\frac{1}{\nu}}$, $\delta_\e  \leq \tfrac{2}{9}[\e c_{_{\!H}}]^{\frac{1}{\nu}}$,  and  $\tfrac{1}{\digamma} \leq \frac{1}{3}[\e c_{_{\!H}}]^{\frac{1}{\nu}}$ by \eqref{e:digamma}.

 \renewcommand{\cm}{c_{_{\!M,g}}}
 From \eqref{e:theta1} and \eqref{e:theta2} it follows that for all $r\in \R$ and $(\rho, t) \in B_{0,i}\times [T_{_{\! B_{0,i}}}, \infty)$ with $d(\varphi_t(\rho),A)\leq \tilde c \,e^{-2\Lambda|t|}$ we have
 \[
 \inf \{|{\Theta}^+_{\varphi_t(\rho)}(d\varphi_t \w+ r H_p) -{\Theta}^+_{\varphi_t(\rho)}({\bf v})| : {\bf v} \in T_{\varphi_t(\rho)} \mathcal F_{\varphi_t(\rho)}\} \geq 2\e{\|\w\|}.
 \]
 Moreover, we claim that {there is $c_{_{\!M,g}}>0$ depending only on $(M,g)$ so that}
 \begin{equation} 
 \label{e:aardvark}
 \|d\varphi_t \w+{\bf v}\|  \geq  {\frac{\e\cm}{2 C_{\varphi}}}\,e^{-\Lambda t}\|\w\|,
 \end{equation}
 for all ${\bf v} \in T_{\varphi_t(\rho)} \mathcal F_{\varphi_t(\rho)} \oplus \R H_p$.
 
 {To see this, first observe that by continuity of $E_{\pm}$ and the fact that $E_{+}\cap E_-=\{0\}$, there exists $c_{_{\!M,g}}>0$ depending only on $(M,g)$ so that for all ${\bf{v}}\in T\TM$ }
 {\begin{equation}\label{e:normCompare}
 \cm(\|\pi_+{\bf{v}}\|+\|\pi_-{\bf{v}}\|)\leq \|{\bf{v}}\|\leq \|\pi_+{\bf{v}}\|+\|\pi_-{\bf{v}}\|.
 \end{equation}} 
 Next, suppose that $\|\pi_+{\bf{v}}\|<\frac{3}{2}\|\pi_+d\varphi_t \w\|$. Then, by{~\eqref{e:theta1}, ~\eqref{e:theta2}}, and {~\eqref{e:normCompare}} 
 {\begin{align*}
 \|d\varphi_t \w+{\bf v}\|&\geq \cm( \|\pi_-d\varphi_t \w\|-\|\pi_-{\bf v}\|)\\
 &\geq \cm( 4\e \|\pi_+d\varphi_t \w\|-2\e\|\pi_+{\bf v}\|)\geq \cm\e\|\pi_+d\varphi_t{\bf{w}}\|.
 \end{align*}}
 On the other hand, assuming that $\e \leq \frac{1}{2}$ we have $\|\pi_+{\bf{v}}\|\geq \frac{3}{2}\|\pi_+d\varphi_t \w\|$, then $$
  \|d\varphi_t \w+{\bf v}\|\geq  \cm(\|\pi_+ {\bf v}\|-\|\pi_+d\varphi_t\w\|)\geq \cm{\tfrac{1}{2}}\|\pi_+d\varphi_t \w\|{\geq \cm {\e}\|\pi_+d\varphi_t \w\|}.
 $$
 Also, note that 
 $$
\|\pi_+d\varphi_t\w\|=\|d\varphi_t\pi_+ \w\|\geq {\tfrac{1}{C_\varphi}}e^{-\Lambda|t|}\|\pi_+\w\|,
 $$
 and
 {$$
 \|\w\|\leq \|\pi_+\w\|+\|\pi_-\w\| \leq (1+\Theta_\rho^+(\w))\|\pi_+\w\|\leq (1+\e)\|\pi_+\w\|.
 $$}
 The proof of \eqref{e:aardvark} follows from noticing that $\frac{\e}{1+\e}\geq \frac{\e}{2}$ since  $\e<1$.

 Since $d(\varphi_t(\rho),A)\leq\tilde c \,e^{-2\Lambda|t|}\leq   \delta_0$, we conclude by \eqref{e:F} and  \eqref{e:aardvark} that for all $s \in \R$
  \begin{align*}
 \|dF(d\varphi_t \w+sH_p)\|  
 &\geq  c_0 \inf\{ \|d\varphi_t \w+{\bf v}\|: {\bf v} \in  T_{\varphi_t(\rho)} \mathcal F_{\varphi_t(\rho)} \oplus \R H_p\} \\
 &\geq  { \frac{c_0\,\e\, \cm }{2C_\varphi}} e^{-\Lambda t} \|\w\|.
 \end{align*}
This means that if $\psi=F \circ \varphi_t$, then $d\psi(t,\rho)$ has a left inverse $L_{(t, \rho)}$ when restricted to $\R\partial_t \oplus \R \w$ with $\|L_{(t, \rho)}\| \leq {\frac{2C_\varphi }{c_0\, \e\, {\cm} }}e^{t\Lambda}$.

In particular, for any $t\geq T_{_{\! B_{0,i}}}$ so that $d(\varphi_t(\rho),A)\leq \tilde c\, e^{-2\Lambda|t|},$ the hypotheses of Proposition~\ref{p:ballCover} apply to the set $A$ with $t_0=T_{_{\! B_{0,i}}}$,  $B=B_{0,i}$, $R=\tilde R_0$, $\beta =\Lambda$, and $c=c_0\, C_{_{\!M,g}}\, \e^{{-1}}$, {$a=2\Lambda$}.
 Fix $0<\tilde r_0< \tilde R_0$ and $0<r \leq \tfrac{1}{ \digamma} e^{- \digamma \Lambda T_0} \tilde r_0 $. Let 
\[
\tilde r_2:= \max \Big\{ {6 e^{\Lambda T_0}r}, \; \tfrac{4}{\alpha_1} r,\; \tfrac{4}{\alpha_1}e^{-\Decay\Lambda T_0}\tilde R_0 \Big\},
\]
and note by the definition \eqref{e:digamma} of $\digamma$ we have
\[
\tilde r_2 < \min \Big\{ \tilde R_0,\; \alpha_2 e^{-{5}\Lambda T_0},\; \tfrac{1}{10 {\bf{C}_0} }e^{-10 \Lambda T_0}\Big\}.
\]
This can be done since $T_0>1$ and $e^{-\Decay \Lambda}<  \tfrac{\alpha_1}{4}$  by the definition \eqref{e:D} of $\Decay$.     

Setting $\tilde r_1:= \max\{2r, e^{-\Decay \Lambda T_0}\}$ we have
 \[
 r<\tilde r_1,\qquad 
 \tilde r_1< \alpha_1\, \tilde r_2, \qquad 
 \tilde r_2 \leq \min\{\tilde R_0,\,\alpha_2\, e^{-{{5} \Lambda} T_0}\},\qquad  
 r < {\tfrac{1}{3}} e^{-\Lambda T_0}\tilde  r_2.
 \]
Therefore, we may apply Proposition~\ref{p:ballCover} to the cover $\{\Lambda_{\rho_j}^\tau(\tilde  r_1)\}_{j \in \mathcal E_{_{\!B_{0,i}}}}$ of $\Lambda_{_{\!B_{0,i}}}^\tau(r)$ where
\begin{equation}
\label{e:ei3}
\mc{E}_{_{\!B_{0,i}}}:=\{j: \Lambda_{\rho_j}^\tau(\tilde r_1)\cap \Lambda_{_{\!B_{0,i}}}^\tau(r)\neq \emptyset\}.
\end{equation}
Then, there is a partition  $\mc{E}_{_{\! B_{0,i}}}=\mc{G}_{_{\!B_{0,i}}}\cup \mc{B}_{_{\!B_{0,i}}}$ with 
\begin{equation}\label{e:badballs0i}
|\mc{B}_{_{\!B_{0,i}}}|\leq {\bf{C}_0} \;\tilde r_2\frac{R_0^{n-1}}{\tilde r_1^{n-1}}\;  T_0e^{{8}\Lambda T_0},
\end{equation}
and so that
\begin{equation}\label{e:noloop2}
 \bigcup_{t=T_{_{\! B_{0,i}}}}^{T_0} \varphi_t\Big(\Lambda^\tau_{B_{0,i}}(r) \backslash \bigcup_{j\in \mc{B}_{_{\!B_{0,i}}}}\!\!\!\Lambda^\tau_{\rho_j}(\tilde  r_1) \Big)\;\bigcap \;\Lambda_{A}^\tau(r)=\emptyset.
\end{equation}
Here ${\bf{C}_{0}}$ coincides with the positive constant used in the definition \eqref{e:digamma} of $\digamma$.
Combining \eqref{e:noloop} with \eqref{e:noloop2}, and using that $E \subset A$ and $0<r< \frac{1}{\digamma} e^{-\digamma \Lambda T_0}\tilde r_0$, we obtain
\begin{equation}\label{e:noloop3}
 \bigcup_{t=t_0}^{T_0} \varphi_t\Big(\Lambda^\tau_{B_{0,i} \backslash\cup_k\tilde{B}_{1,k}}(r) \backslash \bigcup_{j\in \mc{B}_{_{\!B_{0,i}}}}\!\!\!\Lambda^\tau_{\rho_j}(\tilde  r_1) \Big)\;\bigcap \;\Lambda_{E\backslash\cup_k\tilde{B}_{1,k}}^\tau(r)=\emptyset,
\end{equation}
In particular, there are balls $\{\tilde{B}_{2,j}\}_j$ with radii $R_{2,j}=\tilde r_1$ so that 
$$
\bigcup_{t=t_0}^{T_0} \varphi_t(\Lambda^\tau_{B_{0,i}\backslash [\cup_{k,j}\tilde{B}_{1,k}\cup \tilde{B}_{2,j}]}(r))\cap \Lambda_{E\backslash\cup_k\tilde{B}_{1,k}}^\tau(r)=\emptyset.
$$
In addition, 
\begin{equation}\label{e:noloop-radius2}
\sum_j R_{2,j}^{n-1}\leq {\bf{C}_0} \tilde r_2R_0^{n-1}\;  T_0e^{{8}\Lambda T_0} \leq \frac{\e_1}{2}R_0^{n-1},
\end{equation}
 where the first inequality is due to \eqref{e:badballs0i} and the second one is a consequence of the fact that  $\tilde r_2<  \tfrac{1}{10 {\bf{C}_0}} e^{-9 \Lambda T_0}$ and $\frac{\e_1}{2} \geq \tfrac{1}{10}$.\\

Repeating this argument with $B_{0,i}$ for every $i\in\{1, \dots, N\}$ we conclude that there exist balls $\tilde B_{\ell}$ of radius $R_{\ell}$ centered in $E$ so that
 \begin{equation}\label{e:conclusion1}
 {\Lambda^\tau_{E\setminus \cup_\ell \tilde B_{\ell}}(r)\;\;\text{ is }\;\;[t_0(\tfrac{1}{5}),T_0]\text{ non-self looping}}.
 \end{equation}
{Note that $R_\ell= \tilde r_1  \in [0, \tfrac{1}{4}\tilde R_0]$ since  $\tilde r_1=\max\{2r, e^{-\Decay \Lambda T_0}\tilde R_0\}$ while $2r\leq \tfrac{2}{\digamma}\tilde r_0 \leq \tfrac{2}{11} \tilde r_0 \leq \tfrac{1}{4} \tilde R_0$ and   $e^{-\Decay \Lambda}< \tfrac{1}{4}$  by the definition \eqref{e:D} of $\Decay$. }
Also, by \eqref{e:noloop-radius} and \eqref{e:noloop-radius2},
 \begin{equation}\label{e:conclusion2}
 \sum_\ell R_\ell^{n-1}\leq \sum_{i=1}^N \Big(\sum_{k}R_{1,k}^{n-1}+\sum_{j}R_{2,j}^{n-1}  \Big)   \leq  \e_1 \sum_{i=1}^N R_0^{n-1}.
 \end{equation}

Finally, since $R_{1,k}\geq e^{-\Decay\Lambda T_0}R_0$ for all $k$ and {$R_{2,j}=\tilde r_1 \geq e^{-\Decay\Lambda T_0}\tilde R_0$} for all $j$,
 \begin{equation}\label{e:conclusion3}
 R_\ell \geq e^{-\Decay\Lambda T_0}R_0.
 \end{equation}
Relations \eqref{e:conclusion1},  \eqref{e:conclusion2} and  \eqref{e:conclusion3} show that $E$ can be $(\tfrac{1}{5}, \digamma)$-controlled up to time $T_0$ as claimed in \eqref{e:Bcontrol}.

\begin{center}
{\underline{\bf Constructing the complete cover}}
\end{center}

We now partition $\{\rho_j\}_{j=1}^{N_{r_1}}$. Let $t_0=\max\{t_1, \ti_2(\tfrac{1}{5})\}$ where $t_1$ is defined in \eqref{e:t0def} and $t_{2}$ is defined in~\eqref{e:Bcontrol}. By \eqref{e:badD} and \eqref{e:goodD}, for each $i\in \mc{I}_K$ we have constructed a partition $\mc{E}_{_{\!D_i}}=\mc{G}_{_{\!D_i}}\cup \mc{B}_{_{\!D_i}}$ of 
$\mc{E}_{_{\!D_i}}=\{j: \Lambda_{\rho_j}^\tau(r_1)\cap \Lambda_{_{\!{D_i}}}^\tau(r_0)\neq \emptyset\}$
 where 
\begin{equation}\label{e:bothD}
|\mc{B}_{_{\!D_i}}|\leq {{\bf{C}_0}}  \;\frac{{R}^{n-1}}{{r_1^{n-2}}}\;  T_0e^{{4\Lambda}T_0}
\quad\text{and}\quad
\bigcup_{j\in \mc{G}_{_{\!D_i}}}\Lambda^\tau_{\rho_j}(r_1)\;\text{is}\; [t_0,T_0]\text{ non-self looping}.
\end{equation}

Moreover, by \eqref{e:Bballs}, for each $i\in \mc{I}_{\mc{S}_H}$ and  $m>0$ integer we have constructed a partition  of 
$\mc{E}_{_{\!B_i}}=\{j: \Lambda_{\rho_j}^\tau(r_1)\cap \Lambda_{_{\!{B_i}}}^\tau(r_0)\neq \emptyset\}$
 by sets $\{\mc{G}_{_{\!B_i,\ell}}\}_{\ell =0}^m\subset \{1,\dots N_{r_1}\}$, $\mc{B}_{_{\!B_i}}\subset \{1,\dots N_{r_1}\}$ satisfying
\begin{align}\label{e:Bboth}
 \mc{E}_{_{\!B_i}}\subset \mc{B}_{_{\!B_i}}\cup \displaystyle\bigcup_{\ell =0}^m\mc{G}_{_{\!B_i,\ell}},\qquad 
 \bigcup_{{j}\in \mc{G}_{_{\!B_i,\ell}}}\Lambda_{\rho_{{j}}}^\tau(r_1)\text{ is } [t_0,2^{-\ell}T_0]\text{ non-self looping},  \notag\\
 |\mc{G}_{_{\!B_i,\ell}}|\leq C_{_{M,p}}\frac{\delta_\e^{n-1}}{5^{\ell}}\frac{1}{r_1^{n-1} }\qquad \text{ and }
 \qquad|\mc{B}_{_{\!B_i}}|\leq C_{_{M,p}} \frac{\delta_\e^{n-1}}{5^{m+1}}\frac{1}{r_1^{n-1} }. 
 \end{align}

Next, define 
\[
m:=\Big\lfloor \frac{\log T_0- \log {t_0}}{\log 2} \Big\rfloor \qquad  \text{and}\qquad \mc{B}:=\bigcup_{i \in I_{K}} \mc{B}_{_{\!D_i}} \cup \bigcup_{i \in \mc{I}_{\mc{S}_H}} \mc{B}_{_{\!B_i}}.
\]
 For each  $i\in \mc{I}_K$ set $\mc{G}_{{i,0}}:=\mc{G}_{_{\!D_i}}$ and {$\mc{G}_{i,\ell}:=\mc{G}_{_{\!B_{i,\ell-1}}}$} for $\ell \geq 1$. Then, there exists {$I<\infty$}, depending only on $(M,H,p)$, so that after relabelling  the indices $i \in I_{K} \cup I_{\mc{S}_H}$ there  are sets   $\{\mc{G}_{i,\ell}:\; 1\leq \ell \leq m, \; 1\leq i\leq I\}$ so that
\begin{gather*}
\bigcup_{i=1}^I\bigcup_{\ell=1}^m \mc{G}_{i,\ell}\cup \mc{B}= \{1,\dots N_{r_1}\},\qquad
 \bigcup_{j\in \mc{G}_{i,\ell}}\Lambda_{\rho_j}^\tau(r_1)\text{ is } [t_0,2^{-\ell}T_0]\text{ non-self looping}.  
 \end{gather*}
In addition, there exists $c>0$, which  may change from line to line, so that 
 \begin{align*}
  |\mc{B}|
  &\leq c\, r_1^{1-n} \Big( |I_{K}| \;r_1{{R}^{n-1}}\;  T_0e^{{4\Lambda}T_0}  +    |I_{\mc{S}_H}|\frac{\delta_\e^{n-1}}{5^{m+1}}\Big)  \\
  & \leq  c\, r_1^{1-n} \Big( r_1  T_0e^{{4\Lambda}T_0}  +  {\Big(\frac{t_0}{T_0}\Big)^{\frac{\log 5}{\log 2}}}\Big).
  \end{align*}
Here, we have used that $|\mc{I}_{K}| \leq  c\,R^{-(n-1)}$ and $|\mc{I}_{\mc{S}_H}| \leq  c\,\delta_\e^{-(n-1)}$. Since $r_1 \leq e^{-CT_0}$ and we may enlarge $C$ so that $C>4\Lambda +1 +\log 5$, we conclude that 
  \[
   |\mc{B}| \leq  c\, {\Big(\frac{t_0}{T_0}\Big)^{\frac{\log 5}{\log 2}}}r_1^{1-n},
   \]
    as claimed.  In addition, note that $|\mc{G}_{_{\!D_i}}|\leq |\mc{E}_{_{\!D_i}}|\leq  c\, R^{n-1}{r_1^{-(n-1)}}$ for each $i \in  \mc{I}_{K}$ . Therefore, since $R\leq 1$ and $\delta_\e\leq 1$, for all $\ell\in \{1, \dots, m\}$ and all $i \in \{1,\dots, L\}$
 \[|\mc{G}_{i,\ell}|\leq  c\, \frac{1}{5^{\ell}}\, r_1^{1-n}.\]
 
{ Finally, we note that by construction  the constants $c=c(M,g,H)$ and $C=C(M,g,H)$  are uniform for for $H$ varying in a small neighborhood of a fixed submanifold $H_0 \subset M$.}
\end{proof}

\begin{lemma}
\label{l:noFocalTubes}
Suppose that $(M,g)$ has no focal points and $\mc{S}_H=\emptyset.$
Then, the conclusions of Lemma \ref{l:anosov} hold.
\end{lemma}
\begin{proof}
Since  $\SNH$ is compact  by Lemma~\ref{P:1}  there exist positive constants $c_{_{\!K}},t_{_{\!K}},\delta_{_{\!K}}$ so that if $\rho \in K$ and $\varphi_t(\rho) \in \overline{ B(\rho, \delta_{_{\!K}})}$ for some $|t|>t_{_{\!K}}$, then there exists $\w=\w(t,\rho) \in T_{\rho}(\SNH)$ so that 
\begin{equation}
\inf\{\|d\varphi_{t} (\mathbf{\w})+{\bf v}\|:\; {\bf v} \in  {T_{\varphi_{t}(\rho)}}\mc{F}_{\!\varphi_{t}(\rho)} \oplus \R H_p\}\geq  c_{_{\!K}}\|\mathbf{\w}\|.
\end{equation}

Cover $\SNH$ with finitely many balls $\{D_i\}_{i \in I}\subset \SNH$ of radius equal to $\delta_{_{\!K}}$. The remainder of the proof of this lemma is identical to that in the Anosov case since $\mc{S}_H=\emptyset$ implies that $D_i\cap \mc{S}_H=\emptyset$ {for all $i$}.
\end{proof}

\subsection{Proof of Theorem~\ref{T:tangentSpace}}
We first apply Lemma \ref{l:anosov} when $(M,g)$ has Anosov geodesic flow, or Lemma \ref{l:noFocalTubes} when $(M,g)$ has no focal points. Let $c>0$, $C>2$, {$I>0$}, $t_0>1$ be the constants whose existence is given by the lemmas.  Then,  let $\Lambda>\Lambda_{\text{max}}$, $0<\tau_0<\tau_{_{\!\SigH}}$, $0<\tau<\tau_0$, 
\[
0<\e<\tfrac{1}{2}, \qquad 0<a<\tfrac{1-2\e}{\e}, \qquad \tilde c\geq \max\{C, \tfrac{\Lambda_{\text{max}}}{a}\},\qquad  \e \big(1+ \tfrac{\Lambda}{\tilde c} \big)<\delta<\tfrac{1}{2},
\]
\[
T_0(h)= \tfrac{\e}{\tilde c}\log h^{-1},\qquad r_1(h)=h^\e,  \qquad r_0(h)= h^\delta,
\]
 and let $\{\Lambda_{\rho_j}^\tau(h^\e)\}_{j=1}^{N_{h^\e}}$ be {a $({\mathfrak{D}_n},\tau, h^\e)$-good cover of $\SNH$. }
  Then, since $\tilde c \geq C$, Lemmas  \ref{l:anosov} and   \ref{l:noFocalTubes} give that  for each  $i\in\{1,\dots, I\}$, and 
  $$
  m:= \Big\lfloor \frac{\log T_0(h)- \log t_0}{\log 2} \Big\rfloor,
  $$
  there are sets of indices $\{\mc{G}_{i,\ell}\}_{\ell=0}^m\subset \{1,\dots, N_{h^\e}\}$ and $\mc{B}\subset \{1,\dots, N_{h^\e}\}$ so that 
$$
\bigcup_{i=1}^I\bigcup_{\ell=0}^m\mc{G}_{i,\ell}\cup \mc{B}= \{1,\dots, N_{h^\e}\},
$$
and for every $i\in\{1,\dots, I\}$ and every $\ell\in \{0, \dots, m\}$

\[
\bigcup_{j\in \mc{G}_{i,\ell}}\Lambda_{\rho_j}^\tau(h^\e)\;\;\text{is}\;\;[t_0,2^{-\ell}T_0(h)]\;\;\text{ non-self looping},
\] 
\[ 
|\mc{G}_{i,\ell}|\leq c \, 5^{-\ell}\, h^{\e(1-n)},\qquad \qquad  |\mc{B}|\leq c \,{\Big(\frac{\tilde{c}}{\e\log h^{-1}}\Big)^{\frac{\log 5}{\log 2}}}\,  h^{\e(1-n)}.
\]
Next, we apply Theorem~\ref{t:coverToEstimate} with $R(h)=h^\e$, $\alpha=a \e$, $t_\ell(h)=t_0$ for all $\ell$, $T_\ell(h)=2^{-\ell}T_0(h)$ for all $\ell$. Note that $R_0>R(h)\geq 5h^\delta$ for $h$ small enough since $\delta>\e$, and that $\alpha<1-2\e$ as needed. In addition,  $T_\ell(h) \leq 2\alpha T_e(h)$ since $\tilde c \geq  \tfrac{\Lambda_{\text{max}}}{a}$. It follows that {there exists $C>0$, and for all $N>0$ there exists $C_{_{\!N}}$ so that}
 \begin{align*}
&h^{\frac{k-1}{2}}\Big|\int_H w ud\sigma_H\Big|\\
&\qquad \leq C\|w\|_{_{\!\infty}}\Big(\Big[{\Big(\frac{\tilde{c}}{\e\log h^{-1}}\Big)^{\frac{\log 5}{2\log 2}}}+\tfrac{1}{\sqrt{\log h^{-1}}}\sum_\ell (\tfrac{2}{5})^{\frac{\ell}{2}}\Big]\|u\|_{\LM}+\tfrac{\sqrt{\log h^{-1}}}{h}\sum_\ell (\tfrac{1}{10})^{\frac{\ell}{2}}\|Pu\|_{\LM}\Big)\\
&\qquad \quad+Ch^{-1}\|w\|_\infty \|Pu\|_{{\Hs{\frac{k+1}{2}}}} +{C_{_{\!N}}h^N\big(\|u\|_{\LM}+\|Pu\|_{{\Hs{\frac{k+1}{2}}}}\big)}, 
\end{align*}
which gives the desired result after choosing $h_0$ to be small enough.
{We  note that if $H_0 \subset M$,  there is a neighborhood $U$ of $H_0$ (in the $C^\infty$ topology) so that the constants $C$, {$C_{_{\!N}}$} and $h_0$  are uniform over $H\in U$, $w$ taken in a bounded subset of $C_c^\infty$, {and $N$ bounded above}. }
\qed

\subsection{Proof of Theorem~\ref{T:applications}}
We have already proved {Theorem}~{\ref{a1}} in Theorem~\ref{t:noConj1}. For {Theorem}~{\ref{a3}},  {Theorem}~{\ref{a4}}, {Theorem}~{\ref{a5}} we refer the reader to~\cite[Section 5.4]{CG17} where it is shown that either $\mc{A}_H=\emptyset$ in  {Theorem}~{\ref{a3}}, $\mc{S}_H=\emptyset$ in {Theorem}~{\ref{a4}},  and $\mc{A}_H=\emptyset$ in {Theorem}~{\ref{a5}}.  Therefore, Theorem \ref{T:tangentSpace} can be applied to all these setups yielding the desired conclusions.

\begin{proof}[Proof of {Theorem}~{\ref{a2}}]
Let $H$ be a geodesic sphere. Then, $H= \pi(\varphi_s(S_x^*M))$ for some $x\in M$ and $s>0$. 
Next, we observe, using that $(M,g)$ has no conjugate points,  the proof of Theorem~\ref{t:noConj1} (when the submanifold is the point $\{x\}$) yields the existence of a cover  for $S_x^*M$, with some choices of $(R(h), t_\ell(h),T_\ell(h))$, so that Theorem~\ref{t:coverToEstimate} implies the outcome in Theorem~\ref{t:noConj1} (which coincides with that of Theorem~\ref{T:applications}). Then, since $\varphi_s(S_x^*M)=\SNH$, the result follows  from flowing out the cover for time $s$ to obtain a cover for $\SNH$. This cover will have the same desired properties as the original one, but  possibly with $R(h)$ replaced by $m_sR(h)$ for some $m_s>0$ independent of $h$.   The result follows from applying Theorem~\ref{t:coverToEstimate} to the new cover.
\end{proof}

\begin{remark}
This proof in fact shows that there is a certain invariance of estimates under fixed time geodesic flow. That is, if one uses Theorem~\ref{t:coverToEstimate} to conclude an estimate on $H$, then essentially the same estimate will hold on $\pi\varphi_s(\SNH)$ for any $s\in \re$ independent of $h$ provided that $\pi\varphi_s(\SNH)$ is a finite union of submanifolds of codimension $k$ for some $k$.
\end{remark}

\begin{proof}[Proof of {Theorem}~{\ref{a6}}] 
For this part we assume that $(M,g)$ has Anosov geodesic flow, non-positive curvature, and $H$ is a submanifold with codimension $k>1$. We will prove that $\mc{A_H}=\emptyset$, and by Theorem \ref{T:tangentSpace} this will imply the desired conclusion. {In what follows we write $\pi$ for both $\pi:TM \to M$ and $\pi:T^*M \to M$ since it should be clear from context which map is being used.}

We proceed by contradiction. Suppose there exists $\rho \in \mc{A}_H \subset \SNH$.
We  write $\rho^\sharp \in S\!N\!H$ and note
  \[T_{\rho^\sharp } NH=\{\w:  \;\;\exists \,  N\!:\!(-\e,\e)\to N\!H  \; \text{smooth field,\;} N(0)=\rho^\sharp, N'(0)=\w\}.\]    
   Moreover, for $v \in T_{\pi(\rho^\sharp)}H$ and $\w\in T_{\rho^\sharp } N\!H $ with  {$d\pi \w \in T_{\rho^\sharp}H \backslash\{ 0\}$} and $\w=N'(0)$ with $N$ as before,
$$
\langle \tilde{\nabla}_{d\pi \w}N\,,\, v\rangle_{_{\!g(\pi(\rho^\sharp))}} =-\langle \rho^\sharp \,,\, \Pi_H(d\pi \w,v)\rangle_{_{\!g(\pi(\rho^\sharp))}}. 
$$
Here, 
$\tilde{\nabla}$ denotes the Levi--Civita connection on $M$ and $\Pi_H:TH\times TH\to NH$ is the second fundamental form of $H$. 
The 
equality follows from the definition of the second fundamental form, together with the fact that $N$ is a normal vector field.


We will derive a contradiction from
the assumption that $T_{\rho}\SNH=N_+(\rho)\oplus N_-(\rho)$, by showing that the stable and unstable manifolds at $\rho^\sharp$ have signed second fundamental forms. In particular, note that $E_{\pm}^\sharp(\rho^\sharp)$ are given by $T\mc{W}_\pm(\rho^\sharp)$ where $\mc{W}_{\pm}(\rho^\sharp)$ are respectively the stable and unstable manifolds through $\rho^\sharp$. Furthermore, these manifolds are  $\mc{W}_{\pm}(\rho^\sharp)= N\!\mathcal{H}_{\pm}$ where $\mathcal{H}_{\pm}\subset M$ are smooth submanifolds given by the stable/unstable horospheres in $M$ so that $\rho^\sharp\in N\!\mathcal{H}_{\pm}$~\cite[Section 4.1]{Ruggiero}. The signed curvature of $\mathcal{H}_{\pm}$ implies that there is $c>0$ so that 
\begin{equation}
\label{e:signedCurve}
\pm \Pi_{\mathcal{H}_{\pm}}\geq c>0.
\end{equation}
We postpone the proof of this fact until the end of the lemma and first derive our contradiction. 

Since  $T_\rho \SNH=N_+(\rho)\oplus N_-(\rho)$, then 
$
T_{\rho^\sharp}S\!N\!H=N^\sharp_+(\rho)\oplus N^\sharp_{-}(\rho).
$
In addition, since $k>1$, {for any $u\in TH$, there exist $\w_1,\w_2\in T_{\rho^\sharp}SNH$ linearly independent with $d\pi \w_i=u$ for $i=1,2$. In particular, since $T_{\rho^\sharp}(S\!N\!H)=N_+^\sharp(\rho)\oplus N_-^\sharp(\rho)$, we have
$
\w_i=\w_{+,i}+\w_{-,i},
$
with $\w_{\pm,i}\in N_{\pm}^\sharp(\rho)$. Thus, $d\pi \w_+=d\pi \w_-$ where $\w_+=\w_{+,1}-\w_{+,2}\in N_+^\sharp(\rho)$ and $\w_-=\w_{-,2}-\w_{-,1}\in N_-^\sharp(\rho)$.
Since $d\pi:E^\sharp_{\pm}(\rho)\to T_{\pi(\rho)}M$ is injective where $\pi:TM\to M$ is the standard projection, $v:=d\pi \w_\pm\neq 0$.
}

Now, since $\w_{\pm}\in T_{\rho^\sharp}(S\!N\!\mathcal{H}_{\pm})$, using~\eqref{e:signedCurve},
$$
- \langle \tilde{\nabla}_{v}N\,,\, v\rangle_{_{\!g(\pi(\rho^\sharp))}}= - \langle \tilde{\nabla}_{d\pi {\bf w_-}}N\,,\, v\rangle_{_{\!g(\pi(\rho^\sharp))}}  = \langle \rho^\sharp,\Pi_{\mathcal{H}_{+}}(v,v)\rangle\geq c\|v\|^2,
$$
and
$$
 \langle \tilde{\nabla}_{v}N\,,\, v\rangle_{_{\!g(\pi(\rho^\sharp))}}= - \langle \tilde{\nabla}_{d\pi {\bf w_+}}N\,,\, v\rangle_{_{\!g(\pi(\rho^\sharp))}}  = \langle \rho^\sharp,\Pi_{\mathcal{H}_{-}}(v,v)\rangle\leq -c\|v\|^2.
$$
{This is a contradiction since $\|v\|>0$. }

We now prove~\eqref{e:signedCurve}. We have by~\cite[Theorem 1, part (6)]{Eberlein73b} that since $(M,g)$ has Anosov flow and non-positive curvature, there are $c,t_0>0$ so that for any perpendicular Jacobi field $Y(t)$ with $Y(0)=0$, and $t\geq t_0$, 
\begin{equation}
\label{e:curvedHorosphere}
\langle Y'(t),Y(t)\rangle \geq c\|Y(t)\|^2.
\end{equation}
By~\cite[Proof of Lemma 4.2]{Ruggiero} the second fundamental form to $\mathcal{H}_{\pm}$ at $\pi(\rho^\sharp)\in \mathcal{H}_{\pm}$ is given by 
$$
\pm \Pi_{\mathcal{H}_{\pm}}=\mp \lim_{r\to \mp\infty} U_r(0)
$$
where $U_r(t)=Y_r'(t)Y_r^{-1}(t)$ and $Y_r(t)$ is a matrix of perpendicular Jacobi fields along $t \mapsto \pi \varphi_t(\rho)$ satisfying
$Y_r(r)=0$ and $Y_r(0)=\Id.$
In particular, by~\eqref{e:curvedHorosphere}, applied to the Jacobi field $\tilde{Y}(t)=Y_r(r-t)$, at $t=r$ gives for $r\geq t_0$,
\begin{align*}
\langle U_r(0)x,x\rangle&=\langle Y_r'(0)x,Y_r(0)x\rangle=-\langle \tilde{Y}'(r)x,\tilde{Y}(r)x\rangle\leq -c\|Y_r(0)x\|^2=-c\|x\|^2.
\end{align*}
Similarly, for $r\leq -t_0$, we apply~\eqref{e:curvedHorosphere} to $\tilde{Y}(t)=Y_r(r+t)$ at $t=|r|$ to obtain
$$
\langle U_r(0)x,x\rangle=\langle\tilde{Y}'(|r|)x,\tilde{Y}(|r|)x\rangle\geq c\|x\|^2
$$
This yields that $\pm \Pi_{\mathcal{H}_{\pm}}=\mp\lim_{r\to \pm \infty}U_r(0)\geq c>0$ as claimed.
\end{proof}


\subsection{Proof of Theorem~\ref{t:surfaces}}
For {Theorem}~{\ref{a3}} we refer the reader to~\cite[Section 5.4]{CG17} where it is shown that $\mc{A}_H=\emptyset$.  Therefore, Theorem \ref{T:tangentSpace} can be applied to this setup yielding the desired conclusions.\smallskip

We proceed to prove Theorem \ref{a7}.   
Fix a geodesic $H\subset {M}$.
We prove that  Theorem {\ref{a7}} holds under the following curvature assumption.
Suppose there exist $T>0$, and $c_1, c_2, c_3>0$ so that for all  $\rho_0, \rho_1 \in \SNH$ with $d(\rho_0, \rho_1)=s\leq c_3$, and all $t_0, t_1\geq T$ with $\varphi_{t_0}(\rho_0),\varphi_{t_1}(\rho_1) \in \SNH$, we have
\begin{equation}\label{e:modified curvature assumption}
-\int_{Q_s}K dv_{\tilde g} \geq c_1e^{-c_2/{\sqrt{s}}},
\end{equation}
where $Q_s$ is the quadrilateral domain in the universal cover, $(\tilde M, \tilde g)$, whose sides are the geodesics that join the points,   ${\pi(\rho_0)}, {\pi(\rho_1)}, \pi(\varphi_{t_0}(\rho_0)), \pi(\varphi_{t_1}(\rho_1))$.
At the end of the proof we shall show that the integrated curvature assumption \eqref{e:intCurve} implies the assumption in \eqref{e:modified curvature assumption}.

The first step in the proof is to show  that there exist $r_0>0$ and $c_4>0$ so that the following holds. If  $0<r\leq r_0$ and $\rho_0,\rho_1 \in \SNH$ are such that there are $t_0,t_1  \geq T$ with $|t_0-t_1|<\frac{\Tinj}{2}$ and
$$
d(\varphi_{t_0}(\rho_0),\SNH)<r,\qquad d(\varphi_{t_1}(\rho_1),\SNH)<r,
$$ 
then either 
\begin{equation}\label{e:distance claim}
d(\rho_0,\rho_1)<c_2^2\ln\big(\tfrac{c_4}{r}\big)^{-{2}} \qquad \text{or}\qquad  d(\rho_0,\rho_1)>c_3.
\end{equation}
To prove the claim in \eqref{e:distance claim} suppose that there is $\rho_0\in \SNH$ with $d(\varphi_{t_0}(\rho_0),\SNH)<r$ for some $r>0$. Then, there exists $C=C(M,g,H)\geq 1$ so that  by changing $t_0$ to $\tilde{t_0}$ with $|t_0-\tilde{t}_0|\leq C r$ and $r>0$ small enough, we may assume that $\pi(\varphi_{{\tilde t}_0}(\rho_0))\in H$ and $d(\varphi_{\tilde t_0}(\rho_0),\SNH)<2Cr$. Now, let $\rho_s\in \SNH$, with $d(\rho_0,\rho_s)=s$ and suppose there is $t_s$ with $|t_0-t_s|<\tfrac{\Tinj}{2}$ and $d(\varphi_{t_s}(\rho_s),\SNH)<r$. As before, we can adjust ${t_s}$ to $\tilde t_s$, with $|t_s-\tilde t_s|\leq Cr$, in order to have $\pi(\varphi_{\tilde{t}_s}(\rho_s))\in H$ and $d(\varphi_{\tilde t_s}(\rho_s),\SNH)<2Cr$.  Let
\[\gamma_0(t):=\pi(\varphi_t(\rho_0)), \qquad  \gamma_s(t):=\pi(\varphi_t(\rho_s)).\]

Note that, in the universal cover of $M$, $\tilde{M}$, $\gamma_s$ does not intersect $\gamma_0$ unless $\rho_0=\rho_s$. Indeed, suppose they did intersect at an angle $\beta$. Then, by the Gauss--Bonnet theorem, we would have
$$
0\geq\int_{\Delta_s} K\,dv_{\tilde g}=\beta \geq 0,
$$
where $\Delta_s$ is the triangular region enclosed by $\gamma_0$, $\gamma_s$ and $H$.
In particular, this would give $\beta=0$ and hence $\gamma_s=\gamma_0$ and  $s=0$.

Next, suppose that $\gamma_0$ and $\gamma_s$ do not cross in the universal cover.  Let $\alpha_s$ denote the angle between $\dot{\gamma}_s(\tilde t_s)$ and $H$, and let $\alpha_0$ denote the angle between $\dot{\gamma_0}(\tilde t_0)$ and $H$. This can be done since   $\pi(\varphi_{{\tilde t}_0}(\rho_0))\in H$ and  $\pi(\varphi_{\tilde{t}_s}(\rho_s))\in H$.  Then, by the Gauss--Bonnet theorem, 
$$
\pi-\alpha_0-\alpha_s = -\int_{Q_s} K\, dv_{\tilde g}
$$
where $Q_s$ is the quadrilateral formed by   $\gamma_0$, $\gamma_s$, the copy of $H$ in $\tilde M$ that contains ${\pi(\rho_0)}, {\pi(\rho_s)}$, and the copy of $H$ that contains $\pi(\varphi_{{\tilde t}_0}(\rho_0)), \pi(\varphi_{\tilde{t}_s}(\rho_s))$.
Since   $d(\varphi_{\tilde t_0}(\rho_0),\SNH)\leq 2Cr$,   we have $0<\frac{\pi}{2}-\alpha_0\leq 2Cr$. Hence,
$$
\frac{\pi}{2}-\alpha_s\geq-\int_{Q_s} K\,dv_{\tilde g}-2Cr.
$$

In particular, by the curvature assumption \eqref{e:modified curvature assumption} we have that if $s\leq c_3$, 
$$
\frac{\pi}{2}-\alpha_s\geq c_1e^{-c_2/ {\sqrt{s}}} -2Cr. 
$$
Let $\tilde C=\tilde C(H, M, g)>0$ be so that if $\frac{\pi}{2}-\alpha_s \geq 2 \tilde C r$, then $d(\varphi_{\tilde{t}_s}(\rho_s),\SNH)>2Cr$.
Then, for $ c_2^2\ln(c_4r^{-1})^{-{2}}<s \leq c_3$, with $c_4=c_1/{2}(C+\tilde C)$, we have
$$
\frac{\pi}{2}-\alpha_s >  2\tilde{C}r.
$$
This implies that $d(\varphi_{\tilde{t}_s}(\rho_s),\SNH)>2Cr$, and hence proves \eqref{e:distance claim}.

Let $\tau_0$ be the positive constant given in Theorem \ref{t:coverToEstimate} and $0<r\leq r_0$.
Next, we prove that there exists $C>0$ so that if {$0<r_1< r$},  then for every $0<\tau \leq \tau_0$, $T_0>T$,  and  every $({\mathfrak{D}_n},\tau, r_1)$-{good} cover of $\SNH$ by tubes $\{\Lambda^\tau_{\rho_j}(r_1)\}_{j=1}^{N_{r_1}}$, there is a partition $\{1, \dots, N_{r_1}\}=\mc{B} \cup \mc{G}$   so that  \begin{equation}\label{e: artichoke}
\bigcup_{j\in \mc{G}} \Lambda^\tau_{\rho_j}(r_1) \;\;\text{is}\;\; (T,T_0) \; \text{non-self looping \;\;and\;\;} |\mc{B}|\leq C\frac{T_0}{T}\ln\big(\tfrac{c_4}{r}\big)^{-{2}}r_1^{-1}.
\end{equation}
Note that by splitting $[T,T_0]$ into intervals of length $\tau$ the claim in \eqref{e: artichoke} {is implied by} showing that for each 
 $\tilde t \in [T, T_0]$
\begin{equation}\label{e: artichoke leaf}
\#\Bigg\{\rho_j:\;  \bigcup_{|t-\tilde{t}|<\tfrac{\tau}{2}}\varphi_{t}(\Lambda_{\rho_j}^\tau(r_1))\cap \LambdaH(r_1)\neq \emptyset\Bigg\}\leq C\ln\big(\tfrac{c_4}{r}\big)^{-{2}}r_1^{-1}.
\end{equation}
To prove \eqref{e: artichoke leaf} we start by covering $\SNH$ by balls $\{B_\ell\}_{\ell=1}^L$ of radius $\tfrac{c_3}{2}$.  Fix  $\tilde t\geq T+\frac{\tau}{2}$. It follows from \eqref{e:distance claim} that for each $\ell \in \{1, \dots, L\}$, if
$$
N_\ell:= B_\ell\cap\{\rho:\;  \exists\,{t}\in (\tilde t-\tfrac{\tau}{2}, \tilde t+\tfrac{\tau}{2}),\quad d(\SNH,\varphi_{{t}}(\rho))<r\},
$$
then {there is $\rho_\ell\in N_\ell$ such that 
$$
N_\ell\subset \{\rho \in \SNH:\; d(\rho,\rho_\ell)<c_{_2}^2(\ln (c_4 r^{-1}))^{-2}\}.
$$}
In particular, {since $\{\Lambda_{\rho_j}^\tau(r_1)\}_{j=1}^{N_{r_1}}$ is a $(\mathfrak{D}_n,\tau,r_1)$ good cover for $\SNH$ and $r_1<r$} there exists $C_n>0$ so that for each $\ell \in \{1, \dots, L\}$,
\begin{equation*}
 \#\Big\{\rho_j:\; \Lambda_{\rho_j}^\tau(r_1)\cap B_\ell \neq \emptyset,\;\;\bigcup_{|t-\tilde{t}|<\tfrac{\tau}{2}}\varphi_{{t}}(\Lambda_{\rho_j}^\tau(r_1))\cap \LambdaH(r_1)\neq \emptyset\Big\}
 \leq C_n{c_{_2}^2} \ln(\tfrac{c_4}{r})^{-{2}} r_1^{-1}.
\end{equation*}
The claim in \eqref{e: artichoke leaf} follows from taking the union in $\ell$ over all the balls $B_\ell$.

Finally, let  $\e>0$ and $\delta>0$ with  $\e<\delta$.   Also, set $r=h^\e$, $r_1=8h^\delta$ and 
\[T_0= \gamma \log h^{-1}-\beta,  
\qquad 
0<\gamma<\tfrac{\delta -\e}{\Lambda_{\max}}, \qquad\beta<-\tfrac{\log C}{\Lambda_{\max}}.
\]
We have obtained a splitting of $\{1,\dots, N_h\}$ into $\mc{B}\cup \mc{G}$ with the tubes in $\mc{G}$ being $[T, T_0]$ non-self looping and such that
$$
|\mc{B}|\leq C\frac{T_0}{T} (\e\ln c_4{h^{-1}})^{-{2}}h^{-\delta}.
$$
Using this cover in Theorem~\ref{t:coverToEstimate} completes the proof of Theorem~\ref{T:applications} part~\ref{a6} since $\tfrac{T_0}{T}\leq \log h^{-1}$ and hence $h^\delta|\mc{B}|\leq \frac{C}{\log h^{-1}}$ for some $C>0$ and  $h$ small enough.

To see that~\eqref{e:modified curvature assumption} holds, 
let $s\mapsto\rho_s=(x(s), \xi(s)){\in \SNH}$ {be a smooth map}, where $x(s)$ parametrizes $H$ with $|\dot x(s)|_{{g}}=1$ and $\langle \dot \xi(s), \xi(s)\rangle=0$ for all $s$. Next, let $\Gamma(s,t)=\pi(\varphi_t(\rho_s))$ so that $t\mapsto \Gamma(s,t)$ is a geodesic  with $\langle \partial_t\Gamma(s,t),\dot{x}(s)\rangle_g=0$ and $\Gamma(s,0)=x(s)$.

{In particular, if we let 
$$
Y(t)=\partial_s\Gamma(s,t)|_{s=0},
$$
then $Y(t)$ is a {Jacobi} field along $\gamma_0$ with $Y(0)=\dot{x}(0)$ and 
$$\tfrac{D}{dt}Y(0)=\tfrac{D}{ds}\partial_t\Gamma(s,t)\Big|_{(0,0)}{=0}.$$}
{Indeed,} observe that the angle between $\partial_t\Gamma(s,t)|_{t=0}$ and $\dot{x}(s)$ is constant and $|\partial_t\Gamma(s,t)|_g=1$. Therefore, since $x(s)$ is a unit speed geodesic, $\tfrac{D}{ds}\partial_t\Gamma(s,t)|_{t=0}=0$ and hence $\tfrac{D}{dt} Y(0)=0$.

{Now,} let $\gamma_0^\perp(t)$ be a {parallel} vector field along $\gamma_0(t)$ with $\langle \dot{\gamma_0}(t),\gamma_0^\perp(t)\rangle_g=0$ and $|\gamma_0^{\perp}(t)|_g=1$, we then have
$Y(t)=J(t)\gamma_0^{\perp}(t)$ with $J(0)=1$, $J'(0)=0$, and 
$$
J''(t)+R(t)J(t)=0.
$$
Since,  $R(t)\leq 0$ and $J''(t)\geq 0$,
$$
J(t)\geq 1.
$$
In particular,
{
$$\partial_s(\pi\circ\varphi_t(\rho_s))|_{s=0}=d(\pi\circ \varphi_t)|_{\rho_0}\partial_s \rho_s|_{s=0}=Y(t),$$
and hence
$$
d(\pi\circ\varphi_t(\rho_s),\exp_{\pi\circ\varphi_t(\rho_0)}({s}Y(t))\leq C_1e^{2\Lambda t}s^2.
$$
}
{Therefore,} for $t\in [0,4T]$,
$$
d(\gamma_s(t),\exp_{\gamma_0(t)}(sY(t)))\leq C_1e^{8\Lambda T}s^2.
$$
{Since $J(t)\geq 1$,} it follows  that $Q_s$ contains $\Omega_{\tilde{\gamma}}(\tfrac{s}{4})$ for $s<\tfrac{1}{8C_1}e^{-{8\Lambda T}}$ where $\tilde{\gamma}:=\{\gamma_{_{\frac{s}{2}}}(t):t\in [T,2T]\}.$
{Therefore, 
\[
-\int_{Q_s}K dv_{\tilde g} \geq -\int_{\Omega_{\tilde{\gamma}}(\tfrac{s}{4})}K dv_{\tilde g}\geq  c_1e^{-c_2/{\sqrt{s}}},
\]
as claimed.
}
\qed

\ \\

\begin{remark}
We note that the proof of Theorem~\ref{a7} essentially shows that, while horospheres on $M$ may not be positively curved everywhere, their curvature can only vanish at a fixed exponential rate.
\end{remark}

{\begin{remark}\label{r:SXZ16}
This remark explains how Theorem \ref{a7} implies the results of~\cite{SXZ}. Note that the condition in~\cite{SXZ} is that there are $c_1>0$, and $N>0$ such that for every ball $B_s$ in $M$ of radius $s<1$  one has
$
\int_{B_s}K\leq -c_1s^N.
$
This remains true if we replace $M$ by its universal cover, $\tilde{M}$, and implies that $\tilde{M}$ has non-positive curvature. To see that this condition implies those in Theorem \ref{a7}, one needs to check that there  is $c>0$ such that 
$
\int_{\Omega_{\gamma}(s)}K\leq -ce^{-\frac{1}{c\sqrt{s}}}
$
where
$\Omega_{\gamma}(s):=\{ x\in \tilde{M}\mid d(x,\gamma)\leq s\}$. Now, observe that $\Omega_{\gamma(s)}$ contains at least one ball, $B_s$ of radius $s$ and hence, since $\tilde{M}$ has non-positive curvature,
$$
\int_{\Omega_{\gamma}(s)}K\leq \int_{B_s}K\leq -c_1s^N\ll -ce^{-\frac{1}{c\sqrt{s}}},
$$
for some $c>0$.
\end{remark}}


\section{On vanishing of Jacobi of fields}
\label{s:prelim}
\label{s:jacobi}

This section is dedicated to the proof of Proposition~\ref{l:panda} below. The proof of this proposition hinges on showing that given a geodesic $\gamma(t)$, if there is an $r$-dimensional vector space of perpendicular Jacobi fields along the geodesic that vanish at $\gamma(0)$ and that nearly vanish at $\gamma(t_0)$, then there must be $r$ conjugate points to $\gamma(0)$ (counted with multiplicity) near $\gamma(t_0)$.  See Lemma \ref{l:youVanishIVanish} for a precise statement of the required degree of vanishing. There, each $A(t)u_j$ denotes a Jacobi field.

In what follows $\pi:T^*M \to M$ is the natural projection {and  $\varphi_t$ denotes the geodesic flow on $\SM$}.

\begin{proposition}
\label{l:panda}
{Let  $\Lambda>\Lambda_{\max}$.}
There exists $C>0$ so that for any $t_0\in \mathbb{R}$, ${\rho\in \SM}$, and $0<\e<\tfrac{1}{C}e^{-C\Lambda |t_0|}$, the following holds. If there are {no more than $m$} conjugate points to $\pi(\rho)$ (counted with multiplicity) along the geodesic $t\mapsto \pi( \varphi_t(\rho))$ for $t\in (t_0-2\e,t_0+2\e)$,  then there is a subspace  $\mathbf{V_{\!\rho}} \subset T_\rho S^*_xM$   of dimension $n-1-m$  so that for all ${\bf v}\in\mathbf{V_{\!\rho}}$,
$$
{\|{\bf v}\|}\leq C\e^{-1}e^{\Lambda|t_0|}\|d\pi \circ d \varphi_t {\bf v}\|,\qquad t\in (t_0-\e, t_0+\e).
$$
In particular,  $d\pi \circ d\varphi_t: \mathbf{V_{\!\rho}}\to T_{\pi \varphi_t(\rho)}M$ is invertible onto its image with 
$$
\|(d\pi \circ d\varphi_t)^{-1}\|\leq  C\e^{-{1}}e^{\Lambda|t_0|},
$$
 for all $t\in (t_0-\e, t_0+\e)$.
\end{proposition}
 The proof of Proposition~\ref{l:panda} can be found at the end of this section.



\subsection{Preliminaries on the Jacobi equation}
 The argument relies on the fact that given  $v\in T_\rho S_xM$ the vector field  $Y(t)=d\pi \circ d T_t (v)$  is  a Jacobi vector field along the geodesic $\gamma(t)$ in $M$ whose initial conditions are given by $\rho$. Here, $T_t$ denotes the geodesic flow on $TM$. Note that \cite[Proposition 1.7]{Eberlein73} gives   $\|d T_t v\|^2=\|Y(t)\|^2+\|Y'(t)\|^2$ where $'$ denotes the covariant derivative of $Y$ along $\gamma$. 

Let $\{E_1(t), \dots, E_{n-1}(t)\}$ be a parallel orthonormal frame along a geodesic $\gamma$ spanning the orthogonal complement of $E_n(t):=\gamma'(t)$.
 Then for $Y(t)=\sum_{i=1}^{n-1}y_i(t)E_i(t)$ a perpendicular vector field along $\gamma$, we identify $Y$ with $t\mapsto (y_1(t),\dots ,y_{n-1}(t))$. The covariant derivative of $Y$ is then given by $t\mapsto (y_1'(t),\dots y_{n-1}'(t))$. Conversely, for each such curve in $\R^{n-1}$, there is a perpendicular vector field along $\gamma$. Now, for $t\in \R$, we define a symmetric $(n-1)\times (n-1)$ matrix $R(t)=(R_{ij}(t))$ where 
\begin{equation}\label{e:Req}
   R_{ij}(t)=\langle {R({E_n(t),E_{i}(t)})E_n(t),E_j(t)}\rangle_{g(\gamma(t))} 
\end{equation}
and $R(X,Y)$ denotes the curvature tensor. Then we consider the Jacobi equation
\begin{equation}
\label{e:matJacobieqn}
Y''(t)+R(t)Y(t)=0.
\end{equation}
 Let  $A(t)\in \mathbb{M}_{n-1\times n-1}$ solve~\eqref{e:matJacobieqn} with
\begin{equation}
\label{e:conjSol}
A(0)=0,\qquad  A'(0)=\Id.
 \end{equation}
 Then, the perpendicular Jacobi fields on $\gamma$ with $Y(0)=0$ and $\|Y'(0)\|=1$, are given by 
 $$Y(t)= A(t)v,$$ with $\|v\|=1$. In particular, $A(t)$ is nonsingular if and only if $\gamma(0)$ is not conjugate to $\gamma(t)$ along $\gamma$ ({at time $t$}). 

Before proceeding further, we relate $d\varphi_t$ to $A(t)$. To do this, we introduce the horizontal and vertical decomposition of $TM$. Let $\pi:TM\to M$ be projection to the base. Then $d\pi:TTM\to TM$ has kernel equal to the \emph{vertical subspace} of $TTM$. We define the \emph{connection map} $${\bf{K}}:TTM\to TM$$ by the following procedure. Let ${V\in TM}$ and  $v\in T_V(TM)$, let $Z:(-\e,\e)\to TM$ be a smooth curve with initial velocity $v$ and position $V$. Let $\alpha=\pi \circ Z:(-\e,\e)\to M$ and define ${\bf{K}}(v)=Z'(0)$ where $Z'(0)$ denotes the covariant derivative of $Z(t)$ along $\alpha$ evaluated at $t=0$. The  kernel of ${\bf{K}}$ is called the \emph{horizontal subspace}. The \emph{Sasaki metric}, $g_s$, on $TM$ is defined for $v,w\in T_VTM$ by
$$
\langle v,w\rangle_{g_s(V)} :=\langle d\pi v,d\pi w\rangle_{g(\pi(V))}+\langle {\bf{K}}v,{\bf{K}}w\rangle_{g(\pi(V))}.
$$
Under the Sasaki metric, $TTM$ decomposes into the orthogonal sum of the horizontal and vertical subspaces. 

Define the map $^\sharp:\TM \to TM$ and its inverse $^\flat:TM\to \TM $ by
$$
g(\rho^\sharp,W)=\rho(W),\qquad V^\flat(W)=g(V,W).
$$
Next, we define a map $^\sharp: T\TM \to TTM$ and its inverse $^\flat:TTM\to T\TM $ as follows. Let $\rho(t):(-\e,\e)\to \TM $ be a smooth curve with initial velocity ${\bf v}\in T_\rho \TM $. Then, 
$$
{\bf v}^\sharp = \frac{d}{dt}\Big|_{t=0}\rho^\sharp(t).
$$
Similarly, let $V(t):(-\e,\e)\to TM$ be a smooth curve with initial velocity $v\in T_{q}TM$. Then, 
$$
v^\flat = \frac{d}{dt}\Big|_{t=0}V^\flat(t).
$$
Using these identifications, we define the \emph{Sasaki metric on $\TM $}, $g_s^*$, by 
$$
\langle {\bf v},{\bf w}\rangle_{g_s^*}=\langle {\bf v}^\sharp,{\bf w}^\sharp\rangle_{g_s}.
$$
Note also that 
$$
d\pi V^\flat=d\pi V.
$$

The {geodesic flow on $TM$, $T_t:TM\to TM$,} is given by 
$$
T_t V :=( \varphi_t V^\flat)^\sharp.
$$ 
Now, if $v\in T_VTM$, then by~\cite[Proposition 1.7]{Eberlein73}
$$
Y_v(t)=d\pi \circ dT_t(v),\qquad Y_v'(t)={\bf{K}} \circ dT_t(v)
$$
where $Y_v(t)$ is the unique solution to~\eqref{e:matJacobieqn} with $Y_v(0)=d\pi v$ and $Y_v'(0)={\bf{K}}v$. In particular, 
$$
|dT_tv|_{g_s}^2=|Y_v(t)|^2+|Y_v'(t)|^2.
$$
Finally, this implies that for ${\bf v}\in T\TM $,
\begin{equation}\label{e:sum}
   | d\varphi_t {\bf v}|_{g_s^*}^2=|Y_{{\bf v}^\sharp}(t)|^2+|Y_{{\bf v}^\sharp}'(t)|^2.  
\end{equation}

\begin{lemma}
\label{l:tanCotan}
For all $x\in M$ and $\rho \in S_x^*M$ the map $^\sharp$ is an isomorphism from $T_\rho S^*_xM$ to the subspace of $T_{\rho^\sharp}SM$ consisting of vertical vectors $v$ such that ${\bf K}v$ is perpendicular to $\gamma'(0)$ where $\gamma(t)=\pi \circ \varphi_t(\rho)$.
\end{lemma}

\begin{proof}
Let ${\bf v}\in T_\rho S^*_xM$. Then $d\pi {\bf v}=0$ and in particular ${\bf v}^\sharp$ is vertical. Let $\rho(s):(-\e,\e)\to S^*_xM$ with velocity equal to ${\bf v}$ at $0$ and $\rho(0)=\rho$. Then, using geodesic normal coordinates with $x=0$, and $\rho=dx^1$, we have
 $$\rho(t)=\sum_{i=1}^n\rho_i(t)dx^i$$
 with  $\rho_1(0)=1$, and
 $
 \sum_{i=1}^n|\rho_i(t)|^2=1.
 $
 Therefore, 
 $
 \sum_{i=1}^n 2\rho_i(0)\rho_i'(0)=0,
 $
 and hence, since $\rho_i(0)=0$ for $i=2,\dots n$ and $\rho_1(0)=1$, we have $\rho_1'(0)=0$. 
 Next, since $\pi \circ \rho(s)=x$, we have in geodesic normal coordinates at $x$ that
 $
 \rho(t)^\sharp =\sum_{i=1}^n\rho_i(t)\partial_{x_i}.
 $ 
 In particular, since $\gamma(t)=(t,0,\dots, 0)$,
$$
\langle {\bf{K}}{\bf v}^\sharp, \gamma'(0)\rangle_{g(x)} =\partial_t\langle \rho^\sharp(t), \gamma'(0)\rangle_{g(x)}\big|_{t=0}=\partial_t \rho_1(s)|_{t=0}=\rho_1'(0)=0.
$$
Therefore, ${\bf K}{\bf v}^\sharp$ is perpendicular to $\gamma'(0)$. 

Since $\dim T_\rho S^*_xM=n-1$, the set of vectors in $T_xM$ orthogonal to $\gamma'(0)$ has dimension $n-1$, and $^\sharp$ is an isomorphism, this completes the proof of the lemma. 
\end{proof}

Now, fix $\rho\in S^*M$, and let $\gamma(t):=\pi(\varphi_t(\rho))$. Observe that by Lemma~\ref{l:tanCotan} for ${\bf v}\in T_{\rho}S^*_xM$, $d\pi {\bf v}^\sharp=0$ and $\mathbf{K} {\bf v}^\sharp$ is perpendicular to $\gamma'(0)$. Therefore, 
\begin{equation}
\label{e:flowAndJacobi}
d\pi (d\varphi_t {\bf v})^\sharp= A(t)\mathbf{K}{\bf v}^\sharp,\qquad \qquad \mathbf{K}(d\varphi_t{\bf v})^\sharp= A'(t)\mathbf{K}{\bf v}^\sharp.
\end{equation}

The next lemma shows that if $A(t)v$ is small, then $A'(t)v$ cannot be very small.

\begin{lemma}\label{l:boundA'}
Let $\Lambda>\Lambda_{\max}$. Then there is $c>0$ such that for all $\gamma$ geodesic, $A(t)$ solving \eqref{e:conjSol}, $t_0\in \mathbb{R}$ and ${v}\in (\gamma'(0))^{\perp}$ such that 
$
\|A(t_0)v\|\leq \tfrac{c}{2}e^{-\Lambda|t_0|}\|{v}\|,
$
we have
$$
\|A'(t_0)v\|\geq \tfrac{c}{2}e^{-\Lambda|t_0|}\|{v}\|.
$$
\end{lemma}

\begin{proof}
Let $x=\gamma(0)$, $\rho=(x,\gamma'(0))^\flat$, and ${v}\in (\gamma'(0))^\perp$. Then, by Lemma~\ref{l:tanCotan}, there exists ${\bf v}\in T_\rho S^*_xM$ such that $d\pi {\bf v}^\sharp=0$, and $\mathbf{K}{\bf v}^\sharp=v$. 
In particular, by~\eqref{e:flowAndJacobi}
$$
d\pi (d\varphi_t {\bf v})^\sharp= A(t)v,\qquad \qquad\mathbf{K}(d\varphi_t{\bf v})^\sharp= A'(t)v.
$$

Since there exists $C>0$ such that  $\|(d\varphi_{t})^{-1}\| \leq Ce^{\Lambda |t|}$ for all $t$, the maps $\flat$, $\sharp$ are isomorphisms, {and \eqref{e:sum} holds,} there exists $c>0$ such that
\begin{equation}\label{e:LBA'}
\|A(t) v\| + \|A'(t)v \| \geq ce^{-\Lambda|t|}\|v\|.
\end{equation}
In particular, if 
$
\|A(t)v\|\leq \tfrac{c}{2}e^{-\Lambda|t_0|}\|v\|,
$
 the conclusion holds.
\end{proof}

\subsection{Finding conjugate points}
The goal of this section is to prove that if there is a vector space ${\mc{V}}$ of dimension $r$ such that $\|A(t_0)|_{{\mc{V}}}\|$ is small, then there are at least $r$ conjugate points to $\gamma(0)$ (counted with multiplicity) near the point $\gamma(t_0)$. That is, we show that if there is an $r$-dimensional vector space consisting of perpendicular Jacobi fields along $\gamma(t)$ that vanish at $\gamma(0)$ and nearly vanish at $\gamma(t_0)$, then there are $r$ conjugate points to $\gamma(0)$ (counted with multiplicity) near the point $\gamma(t_0)$.

\begin{lemma}
\label{l:youVanishIVanish}
Let $1\leq r\leq n-1$. There are $c,C>0$ such that the following holds. Let $\gamma$ be a geodesic and $A(t)$ solve \eqref{e:conjSol} and suppose there are $t_0 \in \re$, $\{u_j\}_{j=1}^r \subset (\gamma'(0))^{\perp}$ orthonormal and $\beta_0>0$ such that $$\|A(t_0)u_j\| \leq \beta_0,\qquad \qquad \beta_0 \leq c e^{-{(r+2)}\Lambda|t_0|}.$$
Then, there exist  $t_1, \dots, t_{r} \in \re\backslash\{0\}$ such that 
\begin{equation}
\label{e:lConclusion}\sum_{j=1}^r\dim \ker A(t_j)\geq r  \qquad \text{and}\qquad \max_j|t_j-t_0|< C \beta_0 e^{ \Lambda |t_0|}.\end{equation}
\end{lemma}

To ease notation, for any $t$ such that $A^{-1}(t)$ exists,  we introduce the matrix  
\begin{equation}\label{e:Z}
    U(t):=A'(t)A^{-1}(t),
\end{equation}
and note that $U(t)$ is symmetric for all such $t$ \cite{Eberlein73}. This matrix was also used by Green~\cite{Green} and Eberlein~\cite{Eberlein73,Eberlein73b} in the case of no conjugate points, for which it exists for all $t\neq 0$ and solves a certain Ricatti equation.

Recall that in the Newton iteration algorithm for finding zeros of a function, $f$, one starts with $x_0$ where $f(x_0)$ is small, and searches for the zero by defining the sequence 
$x_{n+1}=x_n-\frac{f'(x_n)}{f(x_n)}$. Under appropriate conditions $x_n\to x_*$ and $f(x_*)=0$.

In this section, we implement a Newton-type algorithm for finding non-zero solutions, $(t_*,v_*)$, of the equation $A(t_*)v_*=0$. The sequence $\{x_n\}_n$ is defined so that the linearization of $f$ at $x_n$ will be zero at $x_{n+1}$. In the same spirit, we start at some time $t_0$ where $\|A(t_0)|_{\mc{V}}\|\ll 1$ for some vector space ${\mc{V}}$ and  look for solutions to
\begin{equation}
\label{e:linearization}
A(t_0)v-\lambda_0 A'(t_0)v=0
\end{equation}
such that $|\lambda_0|\ll 1$ and $v\in {\mc{V}}$. 
Since we can rephrase the problem as solving $(\Id-\lambda U(t))A(t)v=0$, the matrix $U(t)$ will be used to do this. In particular, finding solutions to~\eqref{e:linearization} will amount to finding eigenvalues and eigenvectors for $U$. It is here that the self-adjointness of $U$ plays a crucial role. After this step, we put $t_{1}=t_0-\lambda_0$ and repeat the process as in Newton iteration.

In the next lemma, we show that if $\|A(t_0)|_{{\mc{V}}}\|\ll 1$, for some $r$-dimensional vector space ${\mc{V}}$, then we can find $r$ large eigenvalues of the matrix $U(t_0)$.

\begin{lemma}\label{l:eigenvectorsCanBeFound}
There is $C>0$ such that the following holds. Let $t_0\in \mathbb{R}$ and $\beta>0$ such that 
 $A(t_0)^{-1}$ exists, $C\beta e^{\Lambda|t_0|}<1,$ and there are $\{u_j\}_{j=1}^r \subset (\gamma'(0))^{\perp} \backslash \{0\}$ orthogonal with
$$
\max_{j}\frac{\|A(t_0)u_j\|}{\|u_j\|}\leq \beta.
$$
Then, there exist  eigenvalues $\{\lambda_j^{-1}\}_{j=1}^r$  of $U(t_0)$ with $\max_j|\lambda_j|\leq C\beta e^{\Lambda|\tilde{t}|}$ for all $|\tilde{t}-t_0|\leq 1$.
\end{lemma}
\begin{proof}
First, we  check that 
$A(t_0)u$ is small for all $u\in \operatorname{span}\{u_1,\dots,u_r\}$. This follows since {there exists $C_n>0$ depending only on $n$ such that}
$$
\Big\|A(t_0)\sum_{j=1}^rb_j u_j\Big\|\leq \beta\sum_{j=1}^r|b_j|\leq \beta C_n \Big\|\sum_{j=1}^rb_ju_j\Big\|.
$$
In particular, provided $\beta C_n\leq \frac{c}{2}e^{-\Lambda|t_0|}$, by Lemma~\ref{l:boundA'}, we have
$$
\frac{\|U(t_0)A(t_0)u\|}{\|A(t_0)u\|}=\frac{\|A'(t_0)u\|}{\|A(t_0)u\|}\geq \beta^{-1}C_n^{-1}ce^{-\Lambda|t_0|}.
$$

We now apply the max-min principle to $U(t_0)$ using the fact that $A(t_0)$ applied to $\operatorname{span}\{u_1,\dots,u_r\}$ is an $r$ dimensional vector space. That is, observe that if we order the eigenvalues of $U(t_0)$ as $|\lambda_1|^{-1}\geq |\lambda_2|^{-1}\geq \dots \geq |\lambda_{n-1}|^{-1}$, then,
$$
|\lambda_{k}|^{-2}=\max_{\mc{V}}\Big\{\min\big\{ \tfrac{\|Av\|^2}{\|v\|^2}\,:\, v\in {\mc{V}}\big\}\,:\, \dim {\mc{V}}=k\Big\},
$$
{where the maximum is taken over all subspaces $\mc{V}$ of dimension $k$.}
Taking ${\mc{V}}_r=\operatorname{span}\{ A(t_0)u_1,\dots,  A(t_0)u_r\}$, $\dim {\mc{V}}_r=r$, and 
$$
\min\big\{ \tfrac{\|U(t_0)v\|^2}{\|v\|^2}\,:\, v\in {\mc{V}}_r\big\}\geq \beta^{-2}C_n^{-2}c^2e^{-2\Lambda|t_0|}.
$$
In particular, 
$$
|\lambda_{j}|^{-1}\geq \beta^{-1}C_n^{-1}ce^{-\Lambda|t_0|},\qquad j=1,\dots,r.
$$
The bound can be rewritten as a bound in terms of $\tilde t$ by modifying the constant $C$.
\end{proof}

The next lemma will be used to make steps in the Newton iteration. In particular, starting from time $t_0$, where $U(t_0)$ has large eigenvalues, we find a new time, $t_0-s$, where $U(t_0-s)$ has substantially larger eigenvalues.

\begin{lemma}\label{l:lambdasCanBeFound}
There {are $c,C>0$} such that the following holds. 
Suppose that $A(t_0)^{-1}$ exists, $U(t_0)$ has eigenvalues $\{1/\lambda_j\}_{j=1}^{r}$ with $|\lambda_1|={\max_j |\lambda_j|}$ and orthonormal eigenvectors $\{e_j\}_{j=1}^r$. Let $B\geq 0$ and  $|s|\leq 2|\lambda_1|$ such that 
$$
\max_{j}|s-\lambda_j|\leq B|\lambda_1|^{3} \quad\text{and}\quad C(1+B)|\lambda_1|^3\leq \tfrac{c}{2}e^{-{2}\Lambda|t_0|}.
$$
Then, for $v\in \operatorname{span}\{ A(t_0)^{-1}e_j\}_{j=1}^r$,
$$
\|A(t_0-s)v\|\leq C(1+B)|\lambda_1|^3e^{\Lambda|t_0|}\|v\|.
$$
Moreover, if $A(t_0-s)^{-1}$ exists, $U(t_0-s)$ has eigenvalues $\{1/{\lambda}_j(s)\}_{j=1}^r$ satisfying 
$$|{\lambda}_j(s)|\leq C(1+B)|\lambda_1|^3e^{2\Lambda|\tilde{t}|},\qquad \text{for }\;\; |\tilde{t}-t|\leq 1.$$
\end{lemma}

\begin{proof}
We claim that for all $w\in \operatorname{span}\{e_1,\dots e_r\}$  we have 
\begin{equation}\label{e:max-min}
\frac{\|U(t_0-s)A(t_0-s)A^{-1}(t_0)w\|}{\|A(t_0-s)A^{-1}(t_0)w\|}\geq C^{-1}(1+ B)^{-1}|\lambda_1|^{-3}e^{-2\Lambda|t_0|}.
\end{equation}
This would complete the proof after an application of the max-min principle since $A(t_0-s)A^{-1}(t_0)$ applied to $\operatorname{span}\{e_1,\dots,e_r\}$ is an $r$ dimensional vector space.
Note that \eqref{e:max-min} yields a bound on $|\lambda_j(s)|$ in terms of $t_0$. This can be rewritten as a bound in terms of $\tilde t$ by modifying the constant $C$.

Note that 
$
U(t_0-s)A(t_0-s)A^{-1}(t_0)w=A'(t_0-s)A^{-1}(t_0)w
$ 
for $w\in \operatorname{span}\{ e_1, \dots, e_r\}$.
Therefore, by Lemma~\ref{l:boundA'}, proving \eqref{e:max-min} amounts to  finding an upper bound on its denominator.

Given any $t\in \re$, a Taylor expansion near $s=0$  {combined with \eqref{e:matJacobieqn}} yield that for all $v \in (\gamma'(0))^{\perp}$
\begin{equation}\label{e:taylor}
A(t-s)v= A(t)v - s A'(t) v - \tfrac{s^2}{2} R(t)A(t) v + Q(t,s,v),
\end{equation}
with $\|Q(t,s, v)\|\leq C s^3 e^{\Lambda|t|}\|v\|$ for some $C>0$ depending only on $R$ (c.f. \eqref{e:Req}).

Let $w=\sum_{j=1}^rb_je_j$ for some $\{b_j\}_{j=1}^r \subset\re$ and set $v=A^{-1}(t_0)w$. Then, by \eqref{e:taylor}
\begin{align*}
&A(t_0-s)v
=\sum_{j=1}^r \frac{\lambda_j-s}{\lambda_j}b_je_j-\tfrac{1}{2}s^2R(t_0)w+Q(t_0,s,v).
\end{align*}
Next, using that {$|\lambda_1|=\max_{1\leq j\leq r}|\lambda_j|$} and orthogonality, {there is $C>0$ such that}
$$
\frac{1}{|\lambda_1|}\|w\|\leq \Big\|\sum_{j=1}^r\frac{b_j}{\lambda_j }e_j\Big\|
=\big\|U(t_0)w\big\|\leq \|A'(t_0)\|\big\|A^{-1}(t_0)w\big\|\leq {C}e^{\Lambda|t_0|}\|v\|.
$$
Then,
$$
\Big\|\sum_{j=1}^r \frac{\lambda_j-s}{\lambda_j}b_je_j\Big\|^2
{=} \sum_{j=1}^r \frac{|\lambda_j-s|^2}{\lambda_j^2}b^2_j
\leq B^2|\lambda_1|^6\sum_{j=1}^r\frac{b_j^2}{\lambda_j^2}\leq C^2e^{2\Lambda|t_0|}B^2|\lambda_1|^6\|v\|^2.
$$
In particular, together these imply
$
\|A(t_0-s)v\|\leq (C+B)|\lambda_1|^3e^{\Lambda|t_0|}\|v\|.
$
Thus, using Lemma~\ref{l:boundA'}, provided that $C(1+B)|\lambda_1|^3e^{\Lambda|t_0|}\leq\frac{{c}}{2}e^{-\Lambda|t_0|}$,  the claim in \eqref{e:max-min} holds.

\end{proof}

The first step in proving Lemma~\ref{l:youVanishIVanish} is to show that given $u_0$ such that $\|A(t_0)u_0\|\ll \|u_0\|$, we can find $t$ near $t_0$ such that $\ker A(t)\neq \{0\}$. This lemma uses the simplest version of our Newton iteration scheme where we do not keep track of multiplicities.

\begin{lemma}\label{l:algorithm}
There are $c,C>0$ such that the following holds.
  Suppose that there are $t_0\in \re$ and  $u_0   \in (\gamma'(0))^{\perp}\setminus\{0\}$ such that   
\begin{equation}
\label{e:beta}
\|A( t_0)u_0\|\leq \beta\|u_0\|,\qquad 0\leq\beta \leq ce^{-3\Lambda|t_0|}.
\end{equation} 
Then, there exist $t \in \re$ such that 
$$ |t- t_0| \leq C \beta e^{ \Lambda |t_0|}\qquad \text{and}\qquad \dim \ker A(t)\geq 1.$$
\end{lemma}

\begin{proof} We assume by contradiction that 
 $A(s)^{-1}$ exists if $|s-t_0|\leq {C_1}\beta e^{\Lambda|t_0|}$. Then, by Lemma~\ref{l:eigenvectorsCanBeFound}, there is an eigenvector $v_0$ of $U(t_0)$ with eigenvalue $\lambda_0^{-1}$ satisfying
\begin{equation}\label{e:lambda_1}
|\lambda_0| \leq   C\beta  e^{\Lambda|t_0|}.
\end{equation}

Let $t_1:={t_0-\lambda_0}$, $\lambda_{-1}:=\beta^{1/3}e^{-\Lambda|t_0|/3}$, and assume we have found $(t_{k+1},\lambda_k,v_k)$ for  $k=0,\dots, m$ such that $\|v_k\|= 1$,
\begin{equation}
\label{e:inductionStep}
t_{k+1}=t_{k}-\lambda_{k},\,\qquad U(t_{k})v_{k}=\lambda_{k}^{-1} v_{k},\qquad |\lambda_{k}|\leq Ce^{2\Lambda|t_0|}|\lambda_{k-1}|^3.
\end{equation}
By induction, one checks that 
\begin{equation}\label{e:lambdaK}
    |\lambda_k| \leq  \Big(C e^{2\Lambda|t_0|}\Big)^{\sum_{\ell=0}^{k-1}3^{\ell}} \Big(C\beta e^{\Lambda|t_0|} \Big)^{3^{k}},\qquad k=1,\dots,m.
\end{equation}
In particular,
$$
|t_{m+1}-t_0|\leq \sum_{k=0}^m |t_{k+1}-t_k|\leq {2}C\beta {e^{\Lambda|t_0|}}\leq 1.
$$

Next, by Lemma~\ref{l:lambdasCanBeFound} with $t_0={t_{m}}$, $s=\lambda_m$ and $B=0$, there are $(v_{m+1},\lambda_{m+1})$ such that $\|v_{m+1}\|=1$, $U(t_{m+1})v_{m+1}=\lambda_{m+1}^{-1}v_{m+1}$, and
$$
|\lambda_{m+1}|\leq Ce^{2\Lambda|t_0|}|\lambda_m|^3.
$$
Finally, letting $t_{m+2}=t_{m+1}-\lambda_{m+1}$ completes the inductive step.

Therefore, for all $k\geq 0$ there are $(t_k,\lambda_k,v_k)$ satisfying~\eqref{e:inductionStep}. In particular, 
$$
|t_k-t_0|\leq C\beta e^{2\Lambda|t_0|}.
$$
Hence, there exists $t \in \re$ such that $t_k\to t$ and
$
|t-t_0|\leq C\beta e^{2\Lambda|t_0|}.
$ 
Next, note that 
$$
|\lambda_k|^{-1}=\|U(t_k)v_k\|=\|A'(t_k)A^{-1}(t_k)v_k\|\leq {C}e^{\Lambda|t_k|}\|A^{-1}(t_k)v_k\|\leq  Ce^{\Lambda|t_0|}\|A^{-1}(t_k)v_k\|.
$$
In particular, since $|\lambda_k| \to 0$, we conclude $\|A(t_k)^{-1}v_k\|\to \infty$. On the other hand, by assumption $A(t)$ is invertible and hence, there exists $C>0$ and an open interval $I$ around $t$ such that 
$$
\|A(s)^{-1}\|\leq C<\infty,\qquad s\in I,
$$
which gives a contradiction {if we choose $C_1$ large enough}.

\end{proof}

In the proof of Lemma~\ref{l:youVanishIVanish}, we will induct on the number of times at which $A(t)$ is not invertible in a small neighborhood of the time $t_0$ where $\|A(t_0)|_{{\mc{V}}}\|\ll 1$. To begin, we implement Newton iteration to handle the case when we apriori have at most one such time and control the multiplicity of the conjugate point at that time in terms of $\dim {\mc{V}}$.

\begin{lemma}\label{l:Kernel}
There {are $c,C>0$} such that the following holds.
Let $\beta_0>0$ and $t_0\in \mathbb{R}$ with
$
\beta_0\leq ce^{-3\Lambda|t_0|}.
$
Suppose there exists $t_*$ so that  
$$A(t)\; \text{is invertible for }\;t\neq t_* \;\;\text{with}\;\;  |t-t_0|\leq 2C\beta_0e^{\Lambda|t_0|},$$
and that 
 there are $\{u_j\}_{j=1}^r$ orthogonal such that 
$
\|A(t_0)u_j\|\leq \beta_0 \|u_j\|
$
for $j=1, \dots, r$.
Then, {$|t-t_*|\leq C\beta_0 e^{\Lambda|t_0|}$} and $$\dim \ker(A(t_*))\geq r.$$
\end{lemma}

\begin{proof} 

{We first show that, by increasing $C$ and decreasing $c$ slightly, we may assume $t_*\neq t_0$. Indeed, suppose the lemma holds {for some  $t_*\neq t_0$}. 

Let $c,C>0$ be the constants found for {the $t_* \neq t_0$ case} and suppose that $\beta_0\leq \tfrac{c}{2}e^{-3\Lambda|t_0|}$, $A(t)$ is invertible for $t\neq t_0$ with $|t-t_0|\leq 3C\beta_0e^{\Lambda|t_0|}$, and there are $\{u_j\}_{j=1}^r$ orthogonal such that $\|A(t_0)u_j\|\leq \beta_0\|u_j\|$ for $j=1,\dots,r$.

Then, let $s_0\in \mathbb{R}$ and $0<c_0<C$ such that 
$$
0<|s_0-t_0|<c_0\beta_0e^{-\Lambda|t_0|}.
$$
Note that $A(s)$ is invertible, and, since
$$
|t-t_0|<|t-s_0|+|s_0-t_0|<|t-s_0|+c_0\beta_0e^{-\Lambda|t_0|},
$$
$A(t)$ is invertible for $t\neq t_*$ with $|s_0-t|<2C\beta_0e^{\Lambda|t_0|}$. Moreover, since $\|A'(t)\|\leq Ce^{\Lambda|t|}$, we have 
$$
\|A(s_0)u_j\|\leq \|A(t_0)u_j\|+Ce^{\Lambda|t_0|}\|u_j\|\leq \beta_0(1+c_0C)\|u_j\|\leq \tfrac{3}{2}\beta_0\|u_j\|
$$
provided we choose $c_0>0$ small enough. Finally, observe that $\tfrac{3}{2}\beta_0\leq \tfrac{3}{4}ce^{-3\Lambda|t_0|}\leq ce^{-3\Lambda|s_0|}$, again provided we choose $c_0>0$ small enough.
To finish the argument for the $t_0=t_*$ case, apply the lemma with $t_0:=s_0$ and $t_*:=t_0$.

We now prove the lemma assuming that $t_0\neq t_*$.
}

{Let $c,C$ be respectively the minimum and maximum of the constants found in Lemmas~\ref{l:eigenvectorsCanBeFound},~\ref{l:lambdasCanBeFound}, and~\ref{l:algorithm}. By Lemma~\ref{l:algorithm}, since $\beta\leq ce^{-3\Lambda|t_0|}$, $|t_*-t_0|\leq C\beta e^{\Lambda|t_0|}$.}

By Lemma~\ref{l:eigenvectorsCanBeFound}, since $C\beta_0e^{\Lambda|t_0|}<1$, $U(t_0)$ has eigenvalues $\{\lambda_{0,j}^{-1}\}_{j=1}^r$ such that 
$$
|\lambda_{0,j}|\leq C\beta_0e^{\Lambda|t_0|}.
$$

 Let $\{e_{0,j}\}_{j=1}^r$ be the eigenvectors of $U(t_0)$ with eigenvalues $\{1/\lambda_{0,j}\}_{j=1}^r$. Here, we set $\lambda_{0,j}=\lambda_{j}$ for all $j=1, \dots, r$. Note that, by {Lemma~\ref{l:lambdasCanBeFound}}, for all $j=1, \dots, r$
$$
\|A(t_0-\lambda_{0,j}) A^{-1}(t_0)e_{0,j}\|\leq C^4\beta_0^3e^{4\Lambda|t_0|}\|A^{-1}(t_0)e_{0,j}\|.
$$
Then, by Lemma~\ref{l:algorithm} there are $t\in\mathbb{R}$ and $w\in (\gamma'(0))^{\perp}\setminus \{0\}$ such that $A(t)w=0$ and
$
 \max_j|t-t_0+\lambda_{0,j}|\leq C^5\beta_0^3e^{5\Lambda|t_0|}.
$
In particular, {since $|t-t_0|\leq 2\beta_0$}, we have must have $t=t_*$ and so
$$
|t_0-t_*|\leq C\beta_0e^{\Lambda|t_0|}+C^5\beta_0^3e^{5\Lambda|t_0|}.
$$

Set $\beta_{-1}:=(\beta_0({C}^3(1+2C^2e^{2\Lambda|t_0|}))^{-1}e^{-{4}\Lambda|t_0|})^{1/3}$. Let  $m\geq 0$ and  for $0\leq k\leq m$ suppose that we have found  $(t_k,\{\lambda_{k,j}\}_{j=1}^r,\beta_k)$  such that
\begin{enumerate}
    \item $U(t_k)$ has eigenvalues $\{\lambda_{k,j}^{-1}\}_{j=1}^r$ with $\max_j|\lambda_{k,j}|\leq C\beta_k e^{\Lambda|t_0|},$
    \item $A(t)$ is invertible on ${I(t_k,\beta_k)}\setminus \{t_*\}$,
    \item $0<|t_{k}-t_*|\leq C\beta_ke^{\Lambda|t_0|}+C^5\beta_k^3e^{5\Lambda|t_0|},$  
    \item $\beta_{k}\leq C^{
    {3}}(1+2C^2e^{2\Lambda|t_0|})\beta_{k-1}^3e^{{4}\Lambda|t_0|}$,
\end{enumerate}
{where $$I(t_k,\beta_k):=(t_k-2C\beta_k e^{\Lambda|t_0|},t_k+2{C}\beta_ke^{\Lambda|t_0|}).$$}
Then, for each $0\leq k\leq m$  let $\{e_{k,j}\}_{j=1}^r$ be the eigenvectors of $U(t_k)$ with eigenvalues $\{1/\lambda_{k,j}\}_{j=1}^r$. Note that, by {Lemma~\ref{l:lambdasCanBeFound} with $B=0$,}
$$
\|A(t_k-\lambda_{k,j}) A^{-1}(t_k)e_{k,j}\|\leq C|\lambda_{k,j}|^3e^{\Lambda|t_0|}\|A^{-1}(t_k)e_{k,j}\|.
$$
Thus, by Lemma~\ref{l:algorithm} there are $t\in\mathbb{R}$ and $w\in (\gamma'(0))^{\perp}\setminus \{0\}$ such that $A(t)w=0$ and
$
 |t-t_k+\lambda_{k,j}|\leq C^2|\lambda_{k,j}|^3e^{2\Lambda|t_0|}
$
for $j=1, \dots, r.$
In particular, 
{since $t \in I(t_k.\beta_k)$,} we must have $t=t_*$ and so
\begin{equation}
\label{e:lambdasClose}
\max_j|t_*-t_k+\lambda_{k,j}|\leq C^2e^{2\Lambda|t_0|}|\lambda_{k,j}|^3
\qquad \text{and}\qquad 
\max_{j,\ell}|\lambda_{k,j}-\lambda_{k,\ell}|\leq 2C^5\beta_k^3e^{5\Lambda|t_0|}.
\end{equation}
Next, we define
 $t_{k+1}\in \mathbb{R}$ such that 
\begin{equation}\label{e:s_0t_k}
0<|t_*-t_{k+1}|\leq C^{{2}}|\lambda_{k,1}|^3e^{{2}\Lambda|t_0|},
\end{equation}
where $\lambda_{k,1}$ is chosen so that $\max_j|\lambda_{k,j}|=|\lambda_{k,1}|$.
 Then, with ${s_k}=t_k-t_{k+1}$,
$$
\max_j|{s_k}-\lambda_{k,j}|= \max_j|t_*-t_{k+1}+t_k-t_*-\lambda_{k,j}|\leq 2C^{{2}}e^{{2}\Lambda|t_0|}|\lambda_{k,1}|^3.
$$
Thus, we may apply Lemma~\ref{l:lambdasCanBeFound} {with $B=2C^{{2}}e^{{2}\Lambda|t_0|}$, $s=s_k$, and $t_0=t_k$} to obtain that $U(t_{k+1})$ has eigenvalues $\{1/\lambda_{k+1,j}\}_{j=1}^r$ satisfying
$$
|\lambda_{k+1,j}|\leq {C}(1+2C^2e^{2\Lambda|t_0|})|\lambda_{k,1}|^3e^{{2}\Lambda|t_0|}\leq C\beta_{k+1}e^{\Lambda|t_0|},
$$
{where we set $\beta_{k+1}:=C^{3}(1+2C^2e^{2\Lambda|t_0|})\beta_k^3e^{{4}\Lambda|t_0|}$.}

{Next, we claim that $A$ is invertible on {$I(t_{k+1},\beta_{k+1})\setminus \{t_*\}$}.} Indeed, for $t\in {I(t_{k+1},\beta_{k+1})}$, assumptions (3) and (4) in the induction hypotheses and \eqref{e:s_0t_k} yield, {since $|t-t_{k}| \leq |t-t_{k+1}|+|t_*-t_k|+|t_*-t_{k+1}|$,}
$$
\begin{aligned}
|t-t_{k}|
&< 2{C}\beta_{k+1}{e^{\Lambda|t_0|}}+C\beta_ke^{\Lambda|t_0|}+2C^5\beta_k^3e^{5\Lambda|t_0|}
\leq 2{C}\beta_ke^{{\Lambda|t_0|}}.
\end{aligned}
$$
Therefore, {$I(t_{k+1}.\beta_{k+1})\subset I(t_k.\beta_k)$}
and hence $A$ is invertible on {$I(t_{k+1}.\beta_{k+1})\setminus\{t_*\}$}.

Thus, by induction, there are $(t_k,\{\lambda_{k,j}\}_{j=1}^r,\beta_k)$ such that (1)-({4}) above hold. In particular, $\beta_k\to 0$, $t_k\to t_*$, and, {by~\eqref{e:lambdasClose}}, we may choose $\tilde{t}_k {\in I(t_k.\beta_k)}$ such that $A(\tilde{t}_k)$ is invertible and 
$$
\max_j|\tilde{t}_k-t_{k}+\lambda_{k,j}|\leq {2C^2e^{2\Lambda|t_0|}|\lambda_{k,1}|^3}.
$$
Note that  $\tilde{t}_k\to t_*$ and by Lemma~\ref{l:lambdasCanBeFound} {(with $t_0=t_k$, $s= t_k-\tilde t_k$, and $B=2C^2e^{2\Lambda|t_0|}$)}, for $v\in \mc{V}_k:= \operatorname{span}\{A(t_k)^{-1}e_{k,j}\}_{j=1}^r$,
\begin{equation}\label{e:AtoZero}
\begin{aligned}\|A(\tilde{t}_k)v\|&\leq C(1+2C^2e^{2\Lambda|t_0|}){|\lambda_{k,1}|^3}e^{\Lambda|t_0|}\|v\|\\
&\leq {C^4(1+2C^2e^{2\Lambda|t_0|}){\beta_{k}^3}e^{4\Lambda|t_0|}}\|v\|.\\
\end{aligned}
\end{equation}
Choosing any orthonormal basis $\{v_{k,1},\dots, v_{k,r}\}$ for $\mc{V}_k$ we may extract a convergent subsequence $\{v_{k_\ell,j}\}_\ell$ such that 
$$
\lim_{\ell \to \infty}v_{k_\ell,j}= v_j 
$$
for all $j=1, \dots, r$, and where  $\{v_j\}_{j=1}^r \subset (\gamma'(0))^{\perp}$ are orthonormal vectors. Since the map $t\mapsto A(t)$ is continuous, and by \eqref{e:AtoZero}  $\lim_{\ell \to \infty}\|{A(\tilde{t}_{k_\ell})}v_{k_\ell,j}\|=0$ {for all $j=1, \dots, r$}, we conclude
$$
A(t_*)v_j=0,\qquad j=1,\dots,r,
$$
and hence $\dim \ker A(t_*)\geq r$.

\end{proof}

We now prove Lemma~\ref{l:youVanishIVanish}. {We need to address} the fact that Lemma~\ref{l:Kernel} only applies when there is a single time, $t_*$, in an interval proportional to the smallness of $\beta:=\|A(t_0)|_{{\mc{V}}}\|$ such that $A(t_*)$ is not invertible. To explain how to handle this, {we will reduce to the case that there are at most $r-1$} times $t_i$ in a small interval around $t_0$ such that $A(t_i)$ is not invertible. {We will then show that these times can be grouped together into clusters around times $t_{i_\alpha}^\infty$ with corresponding vector spaces $\mc{V}_\alpha^\infty$  such that $\|A(t_{i_\alpha}^\infty)|_{\mc{V}_\alpha^\infty}\|\leq \frac{1}{2}\beta$ and {$\sum_\alpha \dim \mc{V}_{\alpha}^\infty\geq r$}. In other words, by grouping the times appropriately, we can effectively decrease $\beta$. After an induction on the number of times, we will then be able to complete the proof.}


\begin{proof}[{\bf Proof of Lemma~\ref{l:youVanishIVanish}}]
{Let $C$  be the maximum of the constants $C$ found in Lemmas \ref{l:eigenvectorsCanBeFound}, \ref{l:lambdasCanBeFound}, \ref{l:algorithm}, and \ref{l:Kernel}. Similarly, we let $c$ be the minimum of all the constants $c$ given by the same lemmas. To ease the presentation, for $t\in \mathbb{R}$ and $\beta>0$ we again write 
$$
I(t, \beta):=(t-2C\beta e^{\Lambda|t|},t+2C \beta e^{\Lambda|t|}).
$$

 We first reduce to the case that
 there is $0\leq k\leq r-1$ such that 
 \begin{equation}
 \text{$A(t)$ is invertible on 
$
I(t_0,2^k\beta_0)\setminus I(t_0,2^{k-1}\beta_0). 
$}\label{e:PCR}\end{equation}
Suppose there is no such $k$. 
Since $\{I(t_0,2^k\beta_0)\setminus I(t_0,2^{k-1}\beta_0)\}_{k=0}^{r-1}$ are disjoint,
this implies there are $s_0,\dots, s_{r-1}$ distinct such that $s_i\in I(t_0,2^{r-1}\beta_0)$ and $\dim \ker A(s_i)\geq 1$. This implies there exist $\{s_i\}_{i=0}^{r-1}$ with
$$
\sum_{i=0}^{r-1} \dim \ker A(s_i)\geq r,\qquad \max_{i}|t_0-s_i|\leq C 2^r\beta_0 e^{\Lambda|t_0|}
$$
and hence the lemma holds.

 Let $0\leq k\leq r-1$ such that~\eqref{e:PCR} holds and let $\{s_i\}_{i=1}^N\subset I(t_0,2^{k-1}\beta_0)$ distinct such that $A(s_i)$ is not invertible and $A(t)$ is invertible on $I(t_0,2^k\beta_0)\setminus \{s_i\}_{i=1}^N$.
If $N\geq r$, then the proof is complete since $\dim \ker A(s_i)\geq 1$. Therefore, we may assume $N\leq r-1$. {In particular,  with $\beta_1=2^k\beta_0$,  there are $\{t_i\}_{i=0}^{r}\subset I(t_0,\frac{1}{2}\beta_1)$ such that $A(t)$ is invertible on $I(t_0,\beta_1)\setminus \{t_i\}_{i=0}^{r}$. }

By the discussion above it is enough to show that there is $c_r>0$ such that for all $t$, all ${0<\beta}<c_r{e^{-(r+2)\Lambda|t|}}$, and $\{t_i\}_{i=1}^{r-1}\subset I(t,\frac{1}{2}\beta)$, if $A(t)$ is invertible on $I(t,\beta)\setminus\{t_i\}_{i=1}^{r-1}$ and there are $\{u_j\}_{j=1}^r \in \{\gamma'(0)\}^\perp$ orthonormal with $\max_{1\leq j\leq r}\|A(t)u_j\|\leq \beta$, then $\sum_{i=1}^{r-1}\dim \ker A(t_i)\geq r$.

\ \medskip

For $t \in \re$, $\beta>0$, $l\in \mathbb{N}$, $r_0\in \mathbb{N}$,  $\{t_i\}_{i=1}^l\subset \re$ we introduce the following statements:
\begin{itemize}[leftmargin=1em]
\item  $\mc{P}(t,\beta,l,r_0,\{t_i\}_{i=1}^l)$ is the statement: {If $A$ is invertible on $I(t,\beta) \setminus \{ t_i\}_{i=1}^l$ and there are $\{u_j\}_{j=1}^{r_0}$ orthonormal with
${\underset{1\leq j\leq r_0}{\max}}\|A(t)u_j\|\leq \beta,$} then 
$\sum_{i=1}^l \dim \ker A(t_i)\geq {r_0}.$
\item $\mc{P}(t,\beta,l,r_0)$ is the statement:  $\mc{P}(t,\beta,l,r_0,\{t_i\}_{i=1}^l)$ holds for all collections $\{t_i\}_{i=1}^l\subset I(t,\tfrac{1}{2}\beta)$ with $t_i$ distinct.
\end{itemize}

\medskip

The goal is to prove that for all $1\leq l\leq r-1$ {there is $c_l>0$ such that}  {$\mc{P}(t, \beta,l,r)$ holds for all $t$, and $0<\beta<{c_l}e^{-(l+2)\Lambda|t|}$}. We split the proof in two steps. 
\\

}

\noindent\emph{Step 1.} 
Suppose that $k\geq 0$, $0<\beta< ce^{-(k+2)\Lambda|t|}$, {$r_0>0$,} and {$\{t_i\}_{i=1}^k\subset I(t,\tfrac{1}{2}\beta)$} are {distinct times} such that the hypothesis of $\mc{P}(t,\beta,k,{r_0},\{t_i\}_{i=1}^k)$ hold.

We claim that there exist a collection of indices $\{i_{\alpha}\}_{\alpha=1}^m\subset\{1,\dots k\}$,  indices $\{\ell^{\infty}_{i_{\alpha}}\}_{\alpha=1}^m\subset\{1,\dots, k\}$,  times $\{t_{i_{\alpha}}^{\infty}\}_\alpha\subset I(t,\frac{1}{2}\beta)$,  numbers $\{\beta_{\ell^{\infty}_{i_{\alpha}}}\}_\alpha\subset (0,\frac{1}{2}\beta)$,  non-empty disjoint sets $\mc{I}^\infty_{i_\alpha}\subset\{1,\dots k\}$, sets  $\mc{J}^\infty_{i_\alpha}\subset\{1,\dots, r_0\}$,  and intervals
\begin{equation}\label{e:intersectionEmpty}
U_{i_\alpha}^\infty=I(t_{i_{\alpha}}^{\infty},\beta_{\ell^{\infty}_{i_{\alpha}}}),\qquad \qquad  U_{i_{\alpha_1}}^\infty\cap U_{i_{\alpha_2}}^\infty=\emptyset\quad \;\; \alpha_1\neq \alpha_2,
\end{equation}
satisfying
\begin{equation}\label{e:setsAreAll}
 \{1,\dots, k\}=\bigcup_{\alpha=1}^m {\mc{I}_{i_{\alpha}}^\infty},\qquad  \quad {\{1,\dots, r_0\}=\bigcup_{\alpha=1}^m {\mc{J}_{i_{\alpha}}^\infty}},
\end{equation}
and such that $\{t_i\}_{i\in \mc{I}_{\alpha}^\infty}\subset I(t_{i_{\alpha}}^{\infty},\tfrac{1}{2}\beta_{\ell^{\infty}_{i_{\alpha}}})$,  $A$ is invertible on 
$
U_{i_\alpha}^\infty \setminus \{t_i\}_{i\in \mc{I}_{\alpha}^\infty},
$
and {there is a $|\mc{J}_{i_\alpha}^\infty|$-dimensional subspace $\mc{V}_\alpha^\infty$ such that for all $v\in \mc{V}_{\alpha}^\infty$
$$
\|A(t_{i_{\alpha}}^{\infty})v\|\leq \beta_{\ell^{\infty}_{i_{\alpha}}}\|v\|.
$$
In particular, the assumptions of $\mc{P}(t_{i_\alpha}^\infty,\beta_{\ell^{\infty}_{i_{\alpha}}},|\mc{I}_{i_\alpha}^\infty|, |\mc{J}_{i_\alpha}^\infty|,\{t_{i}\}_{i\in \mc{I}_{i_\alpha}^\infty})$
hold. Thus, by partitioning the times appropriately, we are able to reduce $\beta$ by at least half. 
} 



{
Let $s\in \mathbb{R}$ and $c_0>0$ such that 
\begin{equation}\label{e:s_0-t_0}
|s-t|<c_0\beta {e^{-\Lambda|t|}}, \qquad \qquad A(s)^{-1} \;\;\text{exists}.
\end{equation}
Then, {since $\max_{1\leq j\leq r_0}\|A(t)u_j\|\leq \beta$ and $\|A'(t)\|\leq Ce^{\Lambda|t|}$,} for $c_0$ small enough,
\begin{equation}
\label{e:shiftMeNow}
\max_{1\leq j\leq r_0}\|A(s)u_j\|\leq \beta +{C}|s-t|e^{\Lambda|t|}\leq {\tfrac{3}{2}\beta}.
\end{equation}
In particular, by Lemma~\ref{l:eigenvectorsCanBeFound}, $U(s)$ has eigenvalues $\{\lambda_{j}^{-1}\}_{j=1}^{r_0}$ with orthonormal eigenvectors $\{e_{j}\}_{j=1}^{r_0}$ such that 
\begin{equation}\label{e:lambdaBeta}
{|\lambda_1|=}\max_{1\leq j\leq r_0}|\lambda_{j}|\leq {\tfrac{3}{2}}C\beta e^{\Lambda|t|}.
\end{equation}

Observe that for all $j=1, \dots, r_0$, by Lemma~\ref{l:lambdasCanBeFound} (with $t_0=s$, $B=0$, $s=\lambda_{j}$, {and $\tilde{t}=t$}),
$$
\|A(s-\lambda_{j})A^{-1}(s)e_{j}\|\leq C|\lambda_{1}|^3{e^{\Lambda|{t}|}}\|A^{-1}({s})e_{j}\|.
$$
Then, we apply Lemma~\ref{l:algorithm} {(with $t_0=s-\lambda_j$)} to obtain $\{\tilde t_{j}\}_{j=1}^{r_0}$ such that $A(\tilde t_{j})$ is not invertible and
\begin{equation}\label{e:elephant}
\max_j|\tilde t_{j}-s+\lambda_{j}|\leq C^2 |\lambda_{1}|^3{e^{2\Lambda|t|}}.
\end{equation}

Next, defining
\begin{equation}\label{e:beta_1}
\beta_{1}:={C}(1+C^2e^{{2\Lambda|t|}})|\lambda_{1}|^3e^{\Lambda|t|},
\end{equation} 
for $i_0\in \{1,\dots, k\}$ let
$
{U}_{i_0}^{\,1}:=I(t_{i_0}, 3\beta_{1}),
$ 
\begin{equation}\label{e:IandJ}
\mc{I}_{i_0}^{\,1}:=\{ i\mid t_i\in 
{U}_{i_0}^{\,1}\},\qquad \mc{J}_{i_0}^{\,1}:=\{j\mid \min_{i\in \mc{I}_{i_0}^{\,1}}|s-\lambda_{j}-t_i|\leq C^2|\lambda_{1}|^3e^{2\Lambda|t|}\}.
\end{equation}

If $\mc{I}_{i_0}^{\,1}=\{i_0\}$, then
 $A$ is invertible on $I(t_{i_0}, \beta_{1})\setminus \{t_{i_0}\}$ and for all $v\in \operatorname{span}\{A^{-1}(s)e_{j}\}_{j\in \mc{J}_{i_0}^1}$, by Lemma~\ref{l:lambdasCanBeFound} (with $t_0=s$, $B=C^2e^{2\Lambda |t|}$, and  $s-t_{i_0}$ in place of $s$)
 $$
 \|A(t_{i_0})v\|\leq \beta_1\|v\|.
 $$
 We then define $\ell^{\infty}_{i_0}=1$, $t^\infty_{i_0}= t_{i_0}$,
$$
U_{i_0}^\infty:=I(t_{i_0}, \beta_{\ell^{\infty}_{i_0}}), \qquad \mc{J}_{i_0}^{\infty}:=\mc{J}_{i_0}^{\,1}, \qquad  \mc{I}_{i_0}^{\infty}:=\mc{I}_{i_0}^{\,1}.$$

If $\{i_0\} \subsetneq \mc{I}_{i_0}^{\,1}$, let $\bar{t}_{i_0}^{\,1}:=\frac{1}{|\mc{I}_{i_0}^{\,1}|}\sum_{i\in \mc{I}_{i_0}^{\,1}}t_i\in I(t,\frac{1}{2}\beta).$
Then, {by Lemma~\ref{l:lambdasCanBeFound} (with $t_0=s$,  $B=C^2e^{2\Lambda|t|}+12Ce^{\Lambda|t|}\frac{\beta_{1}}{|\lambda_{1}|^3}$}, {and  $s-\bar{t}_{i_0}^{\,1}$ in place of $s$}), for all $v\in \operatorname{span}\{A^{-1}(s)e_{j}\}_{j\in \mc{J}_{i_0}^{\,1}},$
$$
\|A(\bar{t}_{i_0}^{\,1})v\|\leq  \beta_{2} \|v\|,
\qquad 
\beta_{2}:=C\Big(1+C^2e^{2\Lambda|t|}+12{Ce^{\Lambda|t|}}\frac{\beta_{1}}{|\lambda_{1}|^3} \Big)|\lambda_{1}|^3e^{\Lambda|t|}.
$$
Next, let 
$
U_{i_0}^{\,2}:=I(\bar{t}_{i_0}^{\,1}, 3\beta_{{2}})
$
and define $\mc{J}_{i_0}^{\,2},\mc{I}_{i_0}^{\,2}$ as in \eqref{e:IandJ}.
Note that $\mc{I}_{i_0}^{\,1} \subset \mc{I}_{i_0}^{\,2}$.

If $\mc{I}_{i_0}^{\,2}=\mc{I}_{i_0}^{\,1}$, then, since $\beta_2\geq 6\beta_1$, $A$ is invertible on $I(\bar{t}_{i_0}^1,\beta_2)\setminus \{t_i\}_{i\in \mc{I}_{i_0}^2}$, and $\{t_i\}_{i\in \mc{I}_{i_0}^2}\subset I(\bar{t}_{i_0}^1,\frac{\beta_2}{2})$.
We let
$\ell^{\infty}_{i_0}=2$, $t^\infty_{i_0}= \bar{t}_{i_0}^{\,1}$,
$$U_{i_0}^\infty:=I(\bar{t}_{i_0}^{\,1}, \beta_{2}), \qquad  \mc{I}_{i_0}^\infty:=\mc{I}_{i_0}^{\,2},\qquad   \mc{J}_{i_0}^{\infty}:=\mc{J}_{i_0}^{\,2}.$$
Otherwise, we continue the process until we find  $\mc{I}_{i_0}^{\ell} = \mc{I}_{i_0}^{\ell-1}$ for some $\ell$ and set $\ell_{i_0}^\infty=\ell$, $t_{i_0}^\infty=\bar{t}_{i_0}^{\ell-1}\in I(t,\frac{1}{2}\beta)$, 
 $$ U_{i_0}^\infty:=I(\bar{t}_{i_0}^{\,\ell-1}, \beta_{\ell}), \qquad  \mc{I}_{i_0}^\infty:=\mc{I}_{i_0}^{\,\ell},\qquad   \mc{J}_{i_0}^{\infty}:=\mc{J}_{i_0}^{\,\ell}.$$
 Note that for all $i_0$, $\ell_{i_0}^\infty\leq k$.

Next, we claim that if $i_1,i_2$ are such that 
 $U_{i_1}^\infty\cap U_{i_2}^\infty\neq \emptyset$, then 
\begin{equation}\label{e:EitherOr}
    \mc{I}_{i_1}^{\infty}\subset \mc{I}_{i_2}^{\infty} \qquad \text{or} \qquad \mc{I}_{i_2}^{\infty}\subset \mc{I}_{i_1}^{\infty}.
\end{equation}
Indeed, suppose $U_{i_1}^\infty\cap U_{i_2}^\infty\neq \emptyset$. 
Without loss, assume $\ell_{i_2}^\infty\geq \ell_{i_1}^\infty$. Then, $\beta_{\ell^\infty_{i_2}}\geq \beta_{\ell_{i_1}^\infty}$, and so
$
U_{{i_{1}}}^\infty \subset I(t_{i_2}^{\infty},3\beta_{\ell_{i_2}^\infty} ).
$
In particular, since 
$
\mc{I}_{i_1}^{\ell_{i_1}^\infty}=\{i \mid t_i\in U_{i_1}^\infty\}$ and 
$
\mc{I}_{i_2}^{\ell_{i_2}^\infty}=\{ i\mid t_i\in U_{i_2}^\infty\},
$
we have 
$
\mc{I}_{i_2}^\infty\supset \mc{I}_{i_1}^\infty,
$
proving the claim in \eqref{e:EitherOr}.

From the claim in \eqref{e:EitherOr} it follows that there exist $1\leq m\leq k$ and $\{i_\alpha\}_{\alpha=1}^m\subset\{1,\dots, k\}$ such that~\eqref{e:intersectionEmpty} and~\eqref{e:setsAreAll} hold.

To prove that $\beta_{\ell}<\frac{1}{2}\beta$, we actually show that for all $\ell$,
\begin{equation}\label{e:claimSoDoneWithThis}
  \beta_{\ell}\leq C^{{2}(\ell-1)}13^{\ell-1}e^{{(2(\ell-1)+1)}\Lambda|t|}(1+C^2e^{2\Lambda|t|})|\lambda_{1}|^3.  
\end{equation}
This implies $\beta_{\ell_{i_\alpha}^\infty}<\frac{1}{2}\beta$ since $\ell_{i_\alpha}^\infty \leq k$, $|\lambda_{1} | \leq 2C \beta e^{\Lambda|t|}$, and we are assuming $\beta<ce^{-(k+2)\Lambda|t|}$.
To see the claim in \eqref{e:claimSoDoneWithThis} first note that, with $\tilde{\beta}_\ell={\beta_\ell}\big({(1+C^2e^{2\Lambda|t|})|\lambda_{1}|^3}\big)^{-1}$, 
$$
\tilde{\beta}_{{\ell+1}}=C(1+12 {Ce^{\Lambda|t|}} \tilde{\beta}_\ell)e^{\Lambda|t|},\qquad \tilde{\beta}_1=e^{\Lambda|t|}.
$$
Therefore, since $\tilde{\beta}_1\geq 1$, {and we may assume $C\geq 1$, } $\tilde{\beta}_\ell\geq \tilde{\beta}_{\ell-1}$, 
and $
\tilde{\beta}_\ell\leq 13C^{{2}}e^{{2}\Lambda|t|}\tilde{\beta}_{\ell-1}.$
Hence, the claim in \eqref{e:claimSoDoneWithThis} follows.}\\

\noindent \emph{Step 2.}
{
The goal is to prove that for $1\leq l \leq r-1$ there is $c_l>0$ such that {$\mc{P}(t,\beta ,l, r)$ holds for all $t$ {and $0<\beta<c_le^{-(l+2)\Lambda|t|}$}} 
since this would yield the lemma. We {continue} by induction.

First, note $\mc{P}(t,{\beta},1,{r_0})$ holds {for all $t$, {$r_0>0$},} and {$0<\beta<ce^{-3\Lambda |t|}$}  by Lemma~\ref{l:Kernel}. Let $2\leq k\leq r-1$.
{For $1\leq l\leq k-1$, we assume there are $c_l>0$ such that $c_l\leq c_{l+1}$, and for all $r_0>0$}
\begin{equation}\label{e:inducHyp}
    \mc{P}(t,{\beta},l,r_0)\;\;\text{ holds for all $t$,\;\; {$0<\beta<c_l e^{-(l+2)\Lambda|t|}$,}\;\; and  $1\leq l\leq k-1$}.
\end{equation}

{Fix ${r_0>0}$,}  $0<\beta<ce^{-(k+2)\Lambda|t|}$, { $\{t_i\}_{i=1}^k\subset I(t,\tfrac{1}{2}\beta)$ distinct}, and suppose that the assumptions of $\mc{P}(t,\beta,k,{r_0},\{t_i\}_{i=1}^k)$ hold. To finish the proof it suffices to show that the conclusions of $\mc{P}(t,\beta,k,{r_0},\{t_i\}_{i=1}^k)$ hold {and hence that~\eqref{e:inducHyp} holds with $k-1$ replaced by $k$.}

Suppose that the conclusions of $\mc{P}(t,\beta,k,{r_0},\{t_i\}_{i=1}^k)$ do not hold: $\sum_{i=1}^k \dim \ker A(t_i)< {r_0}.$
We will arrive at a contradiction after one further induction {in which the idea is to show that the $t_i$'s  cannot be distinct.}
We claim that there are $\{(s_n, \beta_n)\}_{n=0}^\infty$ such that 
\begin{equation}
\begin{gathered}
\label{e:oneFinalInduction}
{\{t_i\}_{i=1}^k\subset I(s_n,\tfrac{1}{2}\beta_n), \text{ the assumptions of }\mc{P}(s_n,\beta_n,k,{r_0},\{t_i\}_{i=1}^k)\text{ hold}}\\
\text{and}\\
s_n\in I(s_{n-1},\tfrac{1}{2}\beta_{n-1}), \qquad \beta_n<\tfrac{1}{2}\beta_{n-1}.
\end{gathered}
\end{equation}
We prove this by induction. Let $(s_0,\beta_0)=(t,\beta)$ and assume we have found $\{(s_n,\beta_n)\}_{n=0}^{N-1}$ such that \eqref{e:oneFinalInduction} holds.
We claim that there are $(s_{_{\!N}},\beta_{_{\!N}})$ such that \eqref{e:oneFinalInduction} holds with $n=N$. 
To see this, we will apply Step 1 above with $(t,\beta)=(s_{_{\!{N-1}}},\beta_{_{\!{N-1}}})$. 
We work with the corresponding sets and intervals built in \eqref{e:intersectionEmpty} and \eqref{e:setsAreAll}. In particular, the hypotheses of $\mc{P}(t_{i_\alpha}^\infty,\beta_{\ell^{\infty}_{i_{\alpha}}},|\mc{I}_{i_\alpha}^\infty|, |\mc{J}_{i_\alpha}^\infty|,\{t_{i}\}_{i\in \mc{I}_{i_\alpha}^\infty})$ hold for $\alpha=1,\dots, m$ and $\{t_i\}_{i\in \mc{I}_{\alpha}^\infty}\subset I(t_{i_{\alpha}}^{\infty},\tfrac{1}{2}\beta_{\ell^{\infty}_{i_{\alpha}}})$.
Next, suppose that $m>1$. Then, $|\mc{I}_{i_\alpha}^\infty|<k$ for all $\alpha$ and hence, by~\eqref{e:inducHyp}  and \eqref{e:setsAreAll}
$$
\sum_{\alpha=1}^m\sum_{i\in \mc{I}_{i_\alpha}}\dim \ker A(t_i)\geq \sum_{\alpha=1}^m|\mc{J}_{i_\alpha}^{\infty}|=r_0
$$
and hence the conclusions of $\mc{P}(t,\beta,r_0,k,\{t_i\}_{i=1}^k)$ hold contradicting our assumption.
Therefore, we may assume $m=1$ and, in particular, $|\mc{I}_{i_1}^\infty|=k$. The assumptions of $\mc{P}(t_{i_1}^\infty,\beta_{\ell_{i_1}^\infty},k, r_0,\{t_i\}_{i=1}^k)$ hold with $\beta_{\ell_{i_1}^\infty}<\frac{1}{2}\beta_{_{\!{N-1}}}$, $t_{i_1}^\infty\in I(s_{_{\!{N-1}}},\tfrac{1}{2}\beta_{_{\!{N-1}}})$, and $\{t_i\}_{i\in \mc{I}_{1}^\infty}\subset I(t_{i_{1}}^{\infty},\tfrac{1}{2}\beta_{\ell^{\infty}_{i_{1}}})$. Defining, $(s_{_{\!N}},\beta_{_{\!N}})=(t_{i_1}^\infty,\beta_{\ell_{i_1}^\infty})$ we have found $(s_{_{\!N}},\beta_{_{\!N}})$ such that \eqref{e:oneFinalInduction} holds with $n=N$.


{Since~\eqref{e:oneFinalInduction} holds for all $n$, $\beta_n\leq 2^{-n}\beta$ and hence, using that $\{t_i\}_{i=1}^k\subset I(s_n,\tfrac{1}{2}\beta_n)$, for all $n$, we have that 
$$
\max_{i,j} | t_i-t_j|\leq 2^{-n}\beta Ce^{\Lambda|t|}\to 0
$$
and hence $t_i=t_j$ for all $i=1,\dots, k$, a contradiction to the fact that $t_i$ are distinct.}

}
\end{proof}


\begin{proof}[{\bf Proof of Proposition \ref{l:panda}}]
{
Let $c,C$ be as in Lemma  \ref{l:youVanishIVanish}.
 Let $\gamma$ be a geodesic and $A(t)$ solve \eqref{e:conjSol}. Let $t_*\in (t_0-\e,t_0+\e)$ and $\beta:=\tfrac{1}{C}\e e^{-\Lambda|t_*|}$. Without loss of generality assume that 
 $C$ is large enough that $\beta<c e^{-n \Lambda |t_*|}$.

By assumption, there are no more than $m$ conjugate points {to $\gamma(0)$} in $(t_*-\e,t_*+\e)$. In particular, for all $r>m$ there is no collection of times $t_1, \dots, t_{r} $ with $ \max_j|t_j-t_*|< C \beta e^{ \Lambda |t_*|}$ such that 
$\sum_{j=1}^{r}\dim \ker A(t_j)\geq r $.
By Lemma  \ref{l:youVanishIVanish} this implies that
\begin{equation}\label{e:noImaginationAt11pm}
    \|A(t_*)|_{\mc{V}}\|> \tfrac{1}{C}\e e^{-\Lambda|t_*|}
\end{equation}
for all subspaces $\mc{V} \subset (\gamma(0)')^\perp$ with $\dim \mc{V}>m$.}

We claim there is a subspace ${\mc{V}}$ of dimension $n-1-m$ and $C>0$ such that 
\begin{equation}\label{e:tinyPanda}
\|A({t_*})w\|\geq \tfrac{1}{C}\e e^{-\Lambda|t_*|} \|w\|,\qquad w\in \mc{V}.
\end{equation}

To prove this, suppose there is no such subspace ${\mc{V}}$ or $C>0$. Then, for all ${\delta>0}$ there is $v_\delta\neq 0$ such that 
$$
\|A({t_*})v_\delta\|< \delta\e e^{-\Lambda|t_*|} \|v_\delta\|,
$$
Let ${\mc{V}_0}=\{0\}$, ${\mc{V}}_1=\mathbb{R}v_\delta$, $1\leq k{\leq} m$, and suppose that we have found {$C_j>0$} and $\{{\mc{V}}_j\}_{j=1}^k$ such that ${\mc{V}}_{j-1}\subset \{{\mc{V}}_j\}$, $\dim {\mc{V}}_j=j$, and 
$$\|A({t_*})|_{\mc{V}_j}\|\leq {\delta \e C_j e^{-\Lambda|t_*|}}.$$

Note that  $\dim {{\mc{V}}_j^\perp}=n-1-j$, and hence, since 
${n-1-j}{\geq} n-1-m,
$
by assumption there is $w_k\in {{\mc{V}}_k^\perp}$ such that 
$$
\|A({t_*})w_k\|< \delta\e e^{-\Lambda|t_*|} \|w_k\|.
$$
Now, put ${\mc{V}}_{k+1}={\mc{V}}_k\oplus \mathbb{R}w_k$ and let $v=(v_k,\lambda w_k)\in {\mc{V}}_{k+1}$ with $\lambda \in \re$. Then,
$$
\|Av\|\leq \|Av_k\|+|\lambda|\|Aw_k\|\leq\delta \e e^{-\Lambda|t_*|}(C_k\|v_k\|+|\lambda|\|w_k\|)\leq {\delta \e C_{k+1}} e^{-\Lambda|t_*|}\|v\|,
$$
where in the last inequality we use that $v_k$ and $w_k$ are orthogonal.
In particular,
$$
\|A({t_*})|_{{\mc{V}}_{k+1}}\|\leq {\delta}\e C_{j+1} e^{-\Lambda|t_*|}.
$$

Finally, $\dim {\mc{V}}_{m+1}=m+1$, and 
$$
\|A({t_*})|_{{\mc{V}}_{m+1}}\|\leq \delta \e C_m e^{-\Lambda|t_*|},
$$
which contradicts~\eqref{e:noImaginationAt11pm}, provided  $\delta$ is small enough. This proves the claim in \eqref{e:tinyPanda}.

Now, let ${\mc{V}}$ as in~\eqref{e:tinyPanda}. Then, by Lemma~\ref{l:tanCotan} there is ${\bf{V}_\rho}\subset T_\rho S^*_xM$ of dimension $n-1-m$ such that  
$$d\pi {\bf{V}}_{\rho}^\sharp=0,\qquad {\bf K} {\bf{V}}_{\rho}^\sharp={\mc{V}}.$$

For ${\bf v}\in {\bf{V}}_\rho$, 
$$
d\pi (d\varphi_{t_*} {\bf v})^\sharp = A(t){\bf{K}}{\bf v}^\sharp,
$$
and, since ${\bf{K}}{\bf v}^\sharp \in {\mc{V}}$,~\eqref{e:tinyPanda} implies that  for ${\bf v}\in {\bf{V}}_\rho$
$$
\|d\pi d\varphi_{t_*} {\bf v} \|= \|d\pi (d\varphi_{t_*} {\bf v})^\sharp \|\geq \e e^{-\Lambda|t_*|}\|{\bf K}{\bf v}^\sharp\|/C=\e e^{-\Lambda|t_*|}\|{\bf v}^\sharp\|/C\geq \e e^{-\Lambda|t_*|}\|{\bf v}\|/C.
$$
Modifying the constant $C$, we can replace $|t_*|$ by $|t_0|$ in the previous estimate.
\medskip


\end{proof}
\appendix
\section{}


\subsection{Implicit function theorem with estimates on the size}
\begin{lemma}
\label{l:quantImplicit}
Suppose that $f(x_0,x_1,x_2):{\re^{m_0}}\times\re^{m_1}\times \re^{m_2}\to \re^{m_0}$ so that $f(0,0,0)=0$,
\begin{gather*}
L:=(D_{x_0}f(0,0))^{-1}\text{ exists,} \qquad 
  \sup_{|\alpha|=1}|\partial_{x_i}^\alpha f|\leq \tilde B_i, \qquad  \sup_{|\alpha|=1, |\beta|=1}|\partial_{x_i}^\alpha\partial_{x_0}^\beta f|\leq B_i.
\end{gather*}
Suppose further that $r_0,r_1, r_2>0$ satisfy
{\begin{equation}
\label{e:radiiImplicit}S:=\|L\|\sum_{i=0}^2 m_iB_ir_i<1,\qquad \text{ and }\qquad Sr_0+\|L\| \sum_{i=1}^2 m_i\tilde B_ir_i \leq r_0.
\end{equation}}
Then there exists a neighborhood $U\subset \re^{m_1}\times \re^{m_2}$ a function $x_0:U\to \re^n$ so that 
$$
f(x_0(x_1,x_2),x_1,x_2)=0
$$
and $B(0,r_1)\times B(0,r_2)\subset U.$ 
\end{lemma}
\begin{proof}
We employ the usual proof of the implicit function theorem. Let $G:\re^n\to \re^n$ have
$$
G({x_0}; x_1,x_2)={x_0}-Lf({x_0},x_1,x_2).
$$
Our aim is to choose $r_0,r_1>0$ so small that $G$ is a contraction for $x_1\in B(0,r_1)$, ${x_0}\in B(0,r_0)$ and ${x_2}\in B(0,r_2)$. Note that 
Note that
$$
|G({x_0}; x_1,x_2)-G(w; x_1,x_2)|\leq \sup \|D_{x_0}G\||{x_0}-w|
$$
and 
$$
|G({x_0}; x_1,x_2)|\leq \sup\|D_{x_0}G\||{x_0}|+|G(0; x_1,x_2)|.
$$
Therefore, we need to choose $r_i$ small enough that 
\begin{equation}
    \label{e:S1}
S_{_{\! G}}:=\sup\{\|D_{x_0}G\|:\; ({x_0},x_1,x_2)\in B(0,r_0)\times B(0,r_1)\times B(0,r_2)\}<1
\end{equation}
and
\begin{equation} 
\label{e:S2}|G({x_0}; x_1,x_2)|\leq  S_{_{\! G}}r_0+\|L\||f(0,x_1,x_2)|\leq S_{_{\! G}}r_0+\|L\|(m_1\tilde B_1r_1+ m_2\tilde B_2r_2)<r_0.
\end{equation}
Now, 
$$D_{x_0}G=\Id-LD_{x_0}f({x_0},x_1,x_2)$$
and $LD_{x_0}f(0,0,0)=\Id$. Therefore, 
$$\|D_{x_0}G\|\leq \|L\|(m_0B_0r_0+m_1 B_1r_1+m_2B_2r_2)=S<1.$$
In particular, $S_G<S$ and for $r_i$ as in~\eqref{e:radiiImplicit}, we have that~\eqref{e:S1},~\eqref{e:S2} hold. In particular, $G$ is a contraction and the proof is complete. 
\end{proof}

\bibliography{biblio}

\begin{thebibliography}{Wym20b}

\bibitem[Ano67]{Anosov}
Dmitri~V. Anosov.
\newblock Geodesic flows on closed {R}iemannian manifolds of negative
  curvature.
\newblock {\em Trudy Mat. Inst. Steklov.}, 90:209, 1967.

\bibitem[Ava56]{Ava}
Vojislav~G. Avakumovi\'c.
\newblock \uppercase{\"u}ber die {E}igenfunktionen auf geschlossenen
  {R}iemannschen {M}annigfaltigkeiten.
\newblock {\em Math. Z.}, 65:327--344, 1956.

\bibitem[B{\'e}r77]{Berard77}
Pierre~H. B{\'e}rard.
\newblock On the wave equation on a compact {R}iemannian manifold without
  conjugate points.
\newblock {\em Math. Z.}, 155(3):249--276, 1977.

\bibitem[Bla10]{BlairSasaki}
David~E. Blair.
\newblock {\em Riemannian geometry of contact and symplectic manifolds}, volume
  203 of {\em Progress in Mathematics}.
\newblock Birkh\"auser Boston, Inc., Boston, MA, second edition, 2010.

\bibitem[Bon17]{Bo16}
Yannick Bonthonneau.
\newblock The {$\Theta$} function and the {W}eyl law on manifolds without
  conjugate points.
\newblock {\em Doc. Math.}, 22:1275--1283, 2017.

\bibitem[Bou93]{B93}
Jean Bourgain.
\newblock Eigenfunction bounds for the {L}aplacian on the {$n$}-torus.
\newblock {\em Internat. Math. Res. Notices}, (3):61--66, 1993.

\bibitem[BP96]{BuPa}
Keith Burns and Gabriel~P. Paternain.
\newblock On the growth of the number of geodesics joining two points.
\newblock In {\em International {C}onference on {D}ynamical {S}ystems
  ({M}ontevideo, 1995)}, volume 362 of {\em Pitman Res. Notes Math. Ser.},
  pages 7--20. Longman, Harlow, 1996.

\bibitem[CG19]{CG17}
Yaiza Canzani and Jeffrey Galkowski.
\newblock {On the growth of eigenfunction averages: microlocalization and
  geometry}.
\newblock {\em Duke Mathematical Journal}, 168(16):2991--3055, 2019.

\bibitem[CG20a]{CG18d}
Yaiza Canzani and Jeffrey Galkowski.
\newblock Eigenfunction concentration via geodesic beams.
\newblock {\em Journal f{\"u}r die reine und angewandte Mathematik},
  1(ahead-of-print), 2020.

\bibitem[CG20b]{CG18b}
Yaiza Canzani and Jeffrey Galkowski.
\newblock Growth of high ${L}^p$ norms for eigenfunctions: an application of
  geodesic beams.
\newblock {\em \arXiv{2003.04597}}, 2020.

\bibitem[CG20c]{CG18c}
Yaiza Canzani and Jeffrey Galkowski.
\newblock Weyl remainders: an application of geodesic beams.
\newblock {\em \arXiv{2010.03969}}, 2020.

\bibitem[CGT18]{CGT}
Yaiza Canzani, Jeffrey Galkowski, and John~A. Toth.
\newblock Averages of eigenfunctions over hypersurfaces.
\newblock {\em Comm. Math. Phys.}, 360(2):619--637, 2018.

\bibitem[CS15]{CS}
Xuehua Chen and Christopher~D. Sogge.
\newblock On integrals of eigenfunctions over geodesics.
\newblock {\em Proc. Amer. Math. Soc.}, 143(1):151--161, 2015.

\bibitem[DG14]{DyGu14}
Semyon Dyatlov and Colin Guillarmou.
\newblock Microlocal limits of plane waves and {E}isenstein functions.
\newblock {\em Ann. Sci. \'Ec. Norm. Sup\'er. (4)}, 47(2):371--448, 2014.

\bibitem[Ebe73a]{Eberlein73}
Patrick Eberlein.
\newblock When is a geodesic flow of \uppercase{A}nosov type? {I}.
\newblock {\em Journal of Differential Geometry}, 8:437--463, 1973.

\bibitem[Ebe73b]{Eberlein73b}
Patrick Eberlein.
\newblock When is a geodesic flow of \uppercase{A}nosov type? {II}.
\newblock {\em J. Differential Geometry}, 8:565--577, 1973.

\bibitem[Gal18]{GJEDP}
Jeffrey Galkowski.
\newblock A microlocal approach to eigenfunction concentration.
\newblock {\em Journ{\'e}es {\'e}quations aux d{\'e}riv{\'e}es partielles},
  pages 1--14, 2018.

\bibitem[Gal19]{Gdefect}
Jeffrey Galkowski.
\newblock Defect measures of eigenfunctions with maximal {$L^\infty$} growth.
\newblock {\em Ann. Inst. Fourier (Grenoble)}, 69(4):1757--1798, 2019.

\bibitem[G{\'e}r91]{Gerard}
Patrick G{\'e}rard.
\newblock Microlocal defect measures.
\newblock {\em Comm. Partial Differential Equations}, 16(11):1761--1794, 1991.

\bibitem[Goo83]{Good}
Anton Good.
\newblock {\em Local analysis of {S}elberg's trace formula}, volume 1040 of
  {\em Lecture Notes in Mathematics}.
\newblock Springer-Verlag, Berlin, 1983.

\bibitem[Gre58]{Green}
Leon~W. Green.
\newblock A theorem of {E}. {H}opf.
\newblock {\em Michigan Math. J.}, 5:31--34, 1958.

\bibitem[Gro85]{Gros}
Emil Grosswald.
\newblock {\em Representations of integers as sums of squares}.
\newblock Springer-Verlag, New York, 1985.

\bibitem[GT17]{GT}
Jeffrey Galkowski and John~A Toth.
\newblock Eigenfunction scarring and improvements in ${L}^\infty$ bounds.
\newblock {\em Analysis \& PDE}, 11(3):801--812, 2017.

\bibitem[GT20]{GT18a}
Jeffrey Galkowski and John~A. Toth.
\newblock Pointwise bounds for joint eigenfunctions of quantum completely
  integrable systems.
\newblock {\em Comm. Math. Phys.}, 375(2):915--947, 2020.

\bibitem[Hej82]{Hej}
Dennis~A. Hejhal.
\newblock Sur certaines s\'eries de {D}irichlet associ\'ees aux g\'eod\'esiques
  ferm\'ees d'une surface de {R}iemann compacte.
\newblock {\em C. R. Acad. Sci. Paris S\'er. I Math.}, 294(8):273--276, 1982.

\bibitem[H{\"o}r68]{Ho68}
Lars H{\"o}rmander.
\newblock The spectral function of an elliptic operator.
\newblock {\em Acta Math.}, 121:193--218, 1968.

\bibitem[IS95]{I-S}
Henryk Iwaniec and Peter Sarnak.
\newblock {$L^\infty$} norms of eigenfunctions of arithmetic surfaces.
\newblock {\em Ann. of Math. (2)}, 141(2):301--320, 1995.

\bibitem[JZ16]{JZ}
Junehyuk Jung and Steve Zelditch.
\newblock Number of nodal domains and singular points of eigenfunctions of
  negatively curved surfaces with an isometric involution.
\newblock {\em J. Differential Geom.}, 102(1):37--66, 2016.

\bibitem[KH95]{KatokHasselblatt}
Anatole Katok and Boris Hasselblatt.
\newblock {\em Introduction to the modern theory of dynamical systems},
  volume~54 of {\em Encyclopedia of Mathematics and its Applications}.
\newblock Cambridge University Press, Cambridge, 1995.
\newblock With a supplementary chapter by Katok and Leonardo Mendoza.

\bibitem[Kli74]{Kling}
Wilhelm Klingenberg.
\newblock Riemannian manifolds with geodesic flow of {A}nosov type.
\newblock {\em Ann. of Math. (2)}, 99:1--13, 1974.

\bibitem[KTZ07]{KTZ}
Herbert Koch, Daniel Tataru, and Maciej Zworski.
\newblock Semiclassical {$L^p$} estimates.
\newblock {\em Ann. Henri Poincar{\'e}}, 8(5):885--916, 2007.

\bibitem[Lev52]{Lev}
Boris~M. Levitan.
\newblock On the asymptotic behavior of the spectral function of a self-adjoint
  differential equation of the second order.
\newblock {\em Izvestiya Akad. Nauk SSSR. Ser. Mat.}, 16:325--352, 1952.

\bibitem[Pes77]{Pesin}
Ja~B. Pesin.
\newblock Geodesic flows in closed {R}iemannian manifolds without focal points.
\newblock {\em Izv. Akad. Nauk SSSR Ser. Mat.}, 41(6):1252--1288, 1447, 1977.

\bibitem[Ran78]{Randol}
Burton Randol.
\newblock The {R}iemann hypothesis for {S}elberg's zeta-function and the
  asymptotic behavior of eigenvalues of the {L}aplace operator.
\newblock {\em Trans. Amer. Math. Soc.}, 236:209--223, 1978.

\bibitem[Rug07]{Ruggiero}
Rafael~O. Ruggiero.
\newblock {\em Dynamics and global geometry of manifolds without conjugate
  points}, volume~12 of {\em Ensaios Matem\'aticos [Mathematical Surveys]}.
\newblock Sociedade Brasileira de Matem\'atica, Rio de Janeiro, 2007.

\bibitem[STZ11]{SoggeTothZelditch}
Christopher~D. Sogge, John~A. Toth, and Steve Zelditch.
\newblock About the blowup of quasimodes on {R}iemannian manifolds.
\newblock {\em J. Geom. Anal.}, 21(1):150--173, 2011.

\bibitem[SXZ17]{SXZ}
Christopher~D. Sogge, Yakun Xi, and Cheng Zhang.
\newblock Geodesic period integrals of eigenfunctions on {R}iemannian surfaces
  and the {G}auss-{B}onnet theorem.
\newblock {\em Camb. J. Math.}, 5(1):123--151, 2017.

\bibitem[SZ02]{SZ02}
Christopher~D. Sogge and Steve Zelditch.
\newblock Riemannian manifolds with maximal eigenfunction growth.
\newblock {\em Duke Math. J.}, 114(3):387--437, 2002.

\bibitem[SZ16a]{SZ16I}
Christopher~D. Sogge and Steve Zelditch.
\newblock Focal points and sup-norms of eigenfunctions.
\newblock {\em Rev. Mat. Iberoam.}, 32(3):971--994, 2016.

\bibitem[SZ16b]{SZ16II}
Christopher~D. Sogge and Steve Zelditch.
\newblock Focal points and sup-norms of eigenfunctions {II}: the
  two-dimensional case.
\newblock {\em Rev. Mat. Iberoam.}, 32(3):995--999, 2016.

\bibitem[Tac19]{Ta18}
Melissa Tacy.
\newblock {$L^p$} estimates for joint quasimodes of semiclassical
  pseudodifferential operators.
\newblock {\em Israel J. Math.}, 232(1):401--425, 2019.

\bibitem[Wym17]{Wym}
Emmett~L Wyman.
\newblock Integrals of eigenfunctions over curves in compact 2-dimensional
  manifolds of nonpositive sectional curvature.
\newblock {\em \arXiv{1702.03552}}, 2017.

\bibitem[Wym19]{Wym3}
Emmett~L. Wyman.
\newblock Looping directions and integrals of eigenfunctions over submanifolds.
\newblock {\em J. Geom. Anal.}, 29(2):1302--1319, 2019.

\bibitem[Wym20a]{Wym2}
Emmett~L. Wyman.
\newblock Explicit bounds on integrals of eigenfunctions over curves in
  surfaces of nonpositive curvature.
\newblock {\em J. Geom. Anal.}, 30(3):3204--3232, 2020.

\bibitem[Wym20b]{Wym18}
Emmett~L. Wyman.
\newblock Period integrals in nonpositively curved manifolds.
\newblock {\em Math. Res. Lett.}, 27(5):1513--1564, 2020.

\bibitem[Zel92]{Zel}
Steven Zelditch.
\newblock Kuznecov sum formulae and {S}zeg{\H{o}} limit formulae on manifolds.
\newblock {\em Comm. Partial Differential Equations}, 17(1-2):221--260, 1992.

\bibitem[Zwo12]{EZB}
Maciej Zworski.
\newblock {\em Semiclassical analysis}, volume 138 of {\em Graduate Studies in
  Mathematics}.
\newblock American Mathematical Society, Providence, RI, 2012.

\end{thebibliography}
\bibliographystyle{alpha}

\end{document}